\documentclass[a4paper, oneside, reqno, 11pt]{memo-l}
\usepackage{lipsum}

\usepackage[utf8]{inputenc}
\usepackage[english]{babel}
\usepackage{stmaryrd}
\usepackage{relsize}
\RequirePackage{hyperref}

\usepackage{graphicx}
\usepackage{pstricks}
\usepackage{amscd}
\usepackage{amsmath}
\usepackage{amsxtra}
\usepackage{makecell}
\usepackage{float}
\usepackage{tikz}
\usetikzlibrary{arrows}
\usetikzlibrary{calc}
\usepackage[active]{srcltx}
\usepackage{amsfonts,amssymb,amsmath,amsthm}
\usepackage{url}
\usepackage{enumerate}
\usepackage{mathbbol}
\usepackage[active]{srcltx}
\usepackage{chngcntr}
\usepackage{imakeidx}

\newtheorem{case}{Case}[section]

\newtheorem{cas1}{Case}
\newtheorem{cas3}{Case}
\newtheorem*{caution}{Caution}

\numberwithin{section}{chapter}
\numberwithin{equation}{chapter}

\newcommand*{\W}[2]{W^{(\mf #1)}_{\lambdab_{#2}}}

\newcommand*{\Db}[2]{D^{(\mf #1)}_{\lambdab_{#2}}}
\newcommand*{\N}[2]{N^{(\mf #1)}_{\lambdab_{#2}}}
\newcommand*{\Hi}[2]{\mathscr H^{(\mf #1)}_{#2}}
\newcommand*{\D}[1]{D_{\lambdab_{#1}}}
\newcommand*{\Lmd}[2]{\lambda_{#1 #2}}

\newcommand*{\w}{\sqcup_{\mf b}}
\newcommand{\strutstretchdef}{\newcommand{\strutstretch}{\vphantom{\raisebox{1pt}{$\big($}\raisebox{-1pt}{$\big($}}}}
\theoremstyle{plain}
\newtheorem{theorem}{Theorem}[chapter]
\newtheorem{lemma}[theorem]{Lemma}
\newtheorem{proposition}[theorem]{Proposition}
\newtheorem{corollary}[theorem]{Corollary}

\theoremstyle{definition}
\newtheorem{definition}[theorem]{Definition}

\newtheorem{example}[theorem]{Example}

\theoremstyle{remark}
\newtheorem{remark}[theorem]{Remark}

\numberwithin{equation}{chapter}
\numberwithin{equation}{section}

\newlength{\struh}
\newlength{\textminustop}
\strutstretchdef
\newcommand{\gr}{\tt{graph}}

\newcommand{\mf}{\mathfrak}

\newcommand{\V}{V_{\infty}}

\newcommand*{\des}[1]{{{\mathsf{Des}}(#1)}}
\newcommand*{\desb}[1]{{{\mathsf{Des}_{\mf b}}(#1)}}
\newcommand*{\Desb}[1]{{{\mathsf{Des}_{\mf b}}[#1]}}

\newcommand*{\asc}[1]{{{\mathsf{Asc}}(#1)}}

\newcommand*{\child}[1]{\mathsf{Chi}(#1)}
\newcommand*{\childn}[2]{{\mathsf{Chi}}^{\langle#1\rangle}(#2)}

\newcommand*{\parentn}[2]{{\mathsf{par}}^{\langle#1\rangle}(#2)}
\newcommand*{\parset}[2]{\mathfrak{par}(#1, #2)}

\newcommand*{\Ge}{\geqslant}
\newcommand*{\lambdab}{\boldsymbol\lambda}

\newcommand*{\Le}{\leqslant}
\newcommand*{\parent}[1]{\mathsf{par}(#1)}

\newcommand*{\rootb}{{\mathsf{root}}}
\def \dep{\mathsf{d}}

\newcommand*{\supp}{\mathrm{supp}}

\newenvironment{psmallmatrix}
  {\left(\begin{smallmatrix}}
  {\end{smallmatrix}\right)}

\usepackage{mathrsfs}
\usepackage{epsfig}
\usepackage{ifthen}
\usepackage[all]{xy}

\makeatletter

\def\moverlay{\mathpalette\mov@rlay}
\def\mov@rlay#1#2{\leavevmode\vtop{%
   \baselineskip\z@skip \lineskiplimit-\maxdimen
   \ialign{\hfil$\m@th#1##$\hfil\cr#2\crcr}}}
\newcommand{\charfusion}[3][\mathord]{
    #1{\ifx#1\mathop\vphantom{#2}\fi
        \mathpalette\mov@rlay{#2\cr#3}
      }
    \ifx#1\mathop\expandafter\displaylimits\fi}
\makeatother

\newcommand{\cupdot}{\charfusion[\mathbin]{\sqcup}{\cdot}}
\newcommand{\bigcupdot}{\charfusion[\mathop]{\bigsqcup}{\cdot}}

\newcommand{\ncom}{\newcommand}
\ncom{\bq}{\begin{equation}}
\ncom{\eq}{\end{equation}}
\ncom{\beqn}{\begin{eqnarray*}}
\ncom{\eeqn}{\end{eqnarray*}}
\ncom{\beq}{\begin{eqnarray}}
\ncom{\eeq}{\end{eqnarray}}
\ncom{\nno}{\nonumber}
\ncom{\rar}{\rightarrow}
\ncom{\Rar}{\Rightarrow}
\ncom{\noin}{\noindent}
\ncom{\bc}{\begin{centre}}
\ncom{\ec}{\end{centre}}
\ncom{\sz}{\scriptsize}
\ncom{\rf}{\ref}
\ncom{\sgm}{\sigma}
\ncom{\Sgm}{\Sigma}
\ncom{\dt}{\delta}
\ncom{\Dt}{Delta}

\ncom{\eps}{\epsilon}
\ncom{\pcc}{\stackrel{P}{>}}
\ncom{\dist}{{\rm\,dist}}
\ncom{\sspan}{{\rm\,span}}
\ncom{\im}{{\rm Im\,}}
\ncom{\sgn}{{\rm sgn\,}}
\ncom{\ba}{\begin{array}}
\ncom{\ea}{\end{array}}
\ncom{\eop}{\hfill{{\rule{2.5mm}{2.5mm}}}}
\ncom{\eoe}{\hfill{{\rule{1.5mm}{1.5mm}}}}
\ncom{\eof}{\hfill{{\rule{1.5mm}{1.5mm}}}}
\ncom{\hone}{\mbox{\hspace{1em}}}
\ncom{\htwo}{\mbox{\hspace{2em}}}
\ncom{\hthree}{\mbox{\hspace{3em}}}
\ncom{\hfour}{\mbox{\hspace{4em}}}
\ncom{\hsev}{\mbox{\hspace{7em}}}
\ncom{\vone}{\vskip 2ex}
\ncom{\vtwo}{\vskip 4ex}
\ncom{\vonee}{\vskip 1.5ex}
\ncom{\vthree}{\vskip 6ex}
\ncom{\vfour}{\vspace*{8ex}}
\ncom{\norm}{\|\;\;\|}
\ncom{\integ}[4]{\int_{#1}^{#2}\,{#3}\,d{#4}}
\ncom{\inp}[2]{\langle{#1},\,{#2} \rangle}
\ncom{\Inp}[2]{\Langle{#1},\,{#2} \Langle}
\ncom{\vspan}[1]{{{\rm\,span}\#1 \}}}
\ncom{\dm}[1]{\displaystyle {#1}}

\makeindex[columns=3]

\begin{document}

\frontmatter

\title[Weighted Join Operators]{Weighted Join Operators on Directed Trees}
\author[S. Chavan \and R. Gupta]{Sameer Chavan \and Rajeev Gupta
}
\address{Department of Mathematics and Statistics \\ Indian Institute of  Technology   Kanpur, India }
   \email{chavan@iitk.ac.in}
   \email{rajeevg@iitk.ac.in} 
 \author[K. B. Sinha]{Kalyan
B. Sinha
}
   \address{J. N. Centre for Advanced Scientific Research, Jakkur, Bangalore 560064, India  \\ and
 Indian Statistical Institute, India
  }
\email{kbs@jncasr.ac.in}

\thanks{The first, second and third authors are supported by P.K. Kelkar Research Fellowship, Inspire Faculty Fellowship and SERB-Distinguished Fellowship respectively.}

\keywords{Directed tree, join, meet, rank one perturbation, discrete Hilbert transform, commutant, Gelfand-triplet, sectorial, complex Jordan, $n$-symmetric, form-sum}

\date{}

\subjclass[2010]{Primary 47B37, 47B15, 47H06; Secondary 05C20, 47B20}

\begin{abstract} A rooted directed tree $\mathscr T=(V, E)$ with root $\rootb$ can be extended to a directed graph $\mathscr T_\infty=(V_\infty, E_\infty)$ by adding a vertex $\infty$ to $V$ and declaring each vertex in $V$ as a parent of $\infty.$
One may associate with the extended directed tree $\mathscr
T_{\infty}$ a family of semigroup structures $\sqcup_{\mf b}$ with
extreme ends being induced by the join operation $\sqcup$ and the meet
operation $\sqcap$ from lattice theory (corresponding to $\mf
b=\rootb$ and $\mf b= \infty$ respectively). Each semigroup
structure among these leads to a family of densely defined  linear
operators $\W{b}{u}$ acting on $\ell^2(V),$ which we refer to as
weighted join operators at a given base point $\mf b \in V_{\infty}$
with prescribed vertex $u \in V$. The extreme ends of this family
are weighted join operators $\W{\mf \rootb}{u}$ and weighted meet
operators $\W{\mf \infty}{u}$. In this paper, we systematically study the weighted join operators on rooted directed trees.
We also present a more involved
counter-part of weighted join operators $\W{b}{u}$ on rootless
directed trees $\mathscr T$.  In rooted case, these operators are
either finite rank operators, diagonal
operators or  rank one perturbations of diagonal
operators.
In rootless case,  these operators are
either possibly infinite rank operators, diagonal
operators or (possibly unbounded) rank one perturbations of diagonal
operators.
In both cases, the class of weighted
join operators overlaps with the well-studied classes of complex
Jordan operators and $n$-symmetric operators.
An important half of this paper is devoted to the study of rank one extensions  $W_{f, g}$ of weighted join operators $\W{b}{u}$ on rooted directed trees, where $f \in \ell^2(V)$ and $g : V \rar \mathbb C$ is unspecified.
Unlike weighted join operators, these operators are not necessarily closed.
 We provide a couple of compatibility conditions involving the weight system $\lambdab_u$ and $g$ to ensure closedness of $W_{f, g}$.
 These compatibility conditions are intimately related to whether or not an associated discrete Hilbert transform is well-defined. We discuss the role of the Gelfand-triplet
 in the realization of the Hilbert space adjoint of $W_{f, g}$.
Further, we describe various spectral parts of $W_{f, g}$ in terms of the weight system and the tree data. We also provide sufficient conditions for $W_{f, g}$ to be a sectorial operator (resp. an infinitesimal generator of a quasi-bounded strongly continuous semigroup).
In case $\mathscr T$ is leafless, we characterize rank one extensions $W_{f, g}$, which admit compact resolvent. Motivated by the above graph-model, we also take a brief look into the general theory of rank one non-selfadjoint perturbations.
\end{abstract}

\maketitle

\setcounter{tocdepth}{1}
\tableofcontents




\mainmatter

\chapter{Background}

The present work is yet another illustration of the rich interplay
between graph theory and operator theory that includes the recent
developments pertaining to the weighted shifts on directed trees
\cite{Jablonski, BJJS, CPT, BJJS2}. This work exploits the order
structure of directed trees to introduce a class of possibly
unbounded linear operators to be referred to as weighted join
operators $\W{b}{u}$ based at the vertex $\mf b$ and with prescribed
vertex $u \in V$. We capitalize on the fact that any directed tree
has a natural partial ordering induced by the notion of the directed
path. This ordering satisfies all the requirements of the so-called
{\it spiral-like ordering} (SLO) introduced and studied by Pruss for
$p$-regular trees (see \cite[Definition 6.1]{Pr}; see also
\cite[Definition 3]{BL} for modified and extended definition). This
allows us to define join and meet operations on a directed tree
(refer to \cite[Chapter 4]{G} for the basics of lattice theory).
These operations, in turn, induce the so-called weighted join
operators $\W{\mf \rootb}{u}$ and weighted meet operators $\W{\mf
\infty}{u}$ on a directed tree. The present work is devoted to a
systematic study of this class. In case the directed tree is rooted,
it turns out that weighted join operators $\W{b}{u}$ are either
(possibly unbounded) finite rank operators, diagonal operators or
bounded rank one perturbations of (possibly unbounded) diagonal
operators. In case the directed trees are rootless, the situation
being more complex allows unbounded rank one perturbations of
(possibly unbounded) diagonal operators. In particular, weighted
join operators $\W{b}{u}$ on rootless directed trees need not be
even closable. A substantial part of this paper is devoted to the
study of rank one extensions $W_{f, g}$ of weighted join operators.
The so-called compatibility conditions (which controls the rank one
perturbation $f \obslash g$ with the help of the diagonal operator
$\Db{b}{u}$) play an important role in the spectral theory of these
operators. We also discuss the problem of determining the Hilbert
space adjoint of $W_{f, g}$. Our analysis in this problem relies on
the idea of the Hilbert rigging (refer to \cite{BSU}). The notion of
the rank one extensions of weighted join operators is partly
motivated by the graph-model arising from the semigroup structures
on directed trees. Interestingly, the above graph-model plays a decisive
role in deriving various spectral properties of these operators.

The  class of rank one perturbations of diagonal operators has been
studied extensively in the context of hyperinvariant subspace
problem \cite{St, I, FJKP-0, FJKP-1, FJKP, FX, JL, Kl} and spectral
analysis \cite{W, GK, BY0, BY, CT, BY1}. The reader is referred to
\cite{Si0} for a survey on the classical theory of self-adjoint rank
one perturbations of self-adjoint operators (refer also to \cite{LT}
for its connection with the theory of singular integral operators).
The class of rank one perturbations of diagonal operators also
arises naturally in a problem of domain inclusion in the context of
weighted shifts on directed trees (see \cite[Theorem
4.2.2]{Jablonski}). We find it necessary to comment upon the
relationship of the present work to the existing literature. The
class of weighted join operators has essentially no intersection
with the existing class ($\mathcal R \mathcal O$), as studied in
\cite{FJKP-1}, of bounded rank one perturbations of bounded diagonal
operators. Unlike the case of operators in ($\mathcal R \mathcal
O$), commutants of weighted join operators  are not necessarily
abelian. It turns out that there are no non-normal hyponormal
weighted join operators (refer to \cite{Ja, OS} for basics of
unbounded hyponormal operators). In the context of bounded rank one
perturbations of bounded normal operators, similar behaviour has
been observed in \cite{JL}. On the other hand,  the class of
weighted join operators and their rank one extensions contains
bounded as well as unbounded complex Jordan operators  (and
$n$-symmetric operators) in abundance (refer to \cite{He, BH, Ag,
BF, Ag-St, JKP, MR} for the basic theory of Jordan operators,
$n$-symmetric operators, and their connections with the classes of
$n$-normal operators and $n$-isometries). Further, it overlaps with
the class of sectorial operators, and also provides a family of
examples of non-normal compact operators with large null summand in
the sense of Anderson \cite{An}. We would also like to draw
attention to the works \cite{N0, N, ALMS, BBDHL, MS, T} on the
spectral theory of unbounded operator matrices on non-diagonal
domains (refer also to the authoritative exposition \cite{T0} on
this topic). The rank one extensions $W_{f, g}$ of weighted join
operators fit into the class of operator matrices with not
necessarily of {\it diagonal domain}.
Further,  under some compatibility conditions, $W_{f, g}$ is {\it diagonally dominant} in the sense of \cite{T} with the exception that exactly one of its entries is not closable.

In Sections 1 and 2 of this chapter, we collect preliminaries pertaining to the
directed trees and the Hilbert space operators respectively (the
reader is referred to \cite{G, Jablonski} for the basics of graph
theory, and \cite{Si, Sc} for that of Hilbert space operators). In
particular, we set notations and introduce some natural and known
classes of directed trees and unbounded Hilbert space operators,
which are relevant to the investigations in this paper. In Section
3, we collect several simple but basic properties of bounded and
unbounded rank one operators. We conclude this chapter with a
prologue including some important aspects and the lay out of the
paper.

\section{Directed trees}

\index{$\childn{n}{\cdot}$}
\index{$\mathsf{root}$}
\index{$\mathscr T= (V,E)$}
\index{$\child{\cdot}$}
\index{$\mathsf{Par}(\cdot)$}

A pair $\mathscr T= (V,E)$ is  said to be a {\it directed graph} if $V$ is
a nonempty set and $E$ is a nonempty subset of $  V \times V
\setminus \{(v,v): v \in V\}$.
An element of $V$ (resp. $E$) is
referred to as a {\it vertex} (resp. an {\it edge}) of $\mathscr T$.  A
finite sequence $\{v_i\}_{i=1}^n$ of distinct vertices is said to be
a {\it circuit} in $\mathscr T$ if $n \geqslant  2$, $(v_i,v_{i+1})
\in E$ for all $1 \leqslant i \leqslant n-1$ and $(v_n,v_1) \in E$.
Given $u, v \in V$, by a {\it directed path from $u$ to $v$ in
$\mathscr T$}, we understand a finite sequence $\{u_1, \ldots,
u_k\}$ in $V$ such that \beqn u_1=u,~ (u_j, u_{j+1}) \in E~(1 \leqslant j
\leqslant k-1)~ \mbox{and~}u_k=v. \eeqn We say that two distinct vertices $u$ and
$v$ of $\mathscr T$ are {\it connected by a path} if there exists a
finite sequence $\{u_1, \ldots,
u_k\}$ of distinct vertices of $\mathscr
T$ such that \beqn u_1=u, ~ (u_j,u_{j+1})~ \mbox{or~} (u_{j+1}, u_j)
\in E~(1 \leqslant j \leqslant k-1) ~\mbox{and~} u_k=v. \eeqn
A directed graph
$\mathscr T$ is said to be {\it connected} if any two distinct
vertices of $\mathscr T$ can be connected by a path in $\mathscr T.$
For a subset $W$ of $V$, define
$$\child{W} \,:=\, \bigcup_{u\in W} \{v\in V
\colon (u,v) \in E\}.
$$
One may define inductively $\childn{n}{W}$ for
a non-negative integer $n$ as follows:
\beqn
\childn{n}{W}\,:=\,
\begin{cases} W & ~\mbox{if }~n=0, \\
\child{\childn{n-1}{W}} & ~\mbox{if~} n \geqslant 1. \end{cases}
\eeqn Given $v\in V$ and an integer $n \geqslant 0$, we set
$\childn{n}{v}:=\childn{n}{\{v\}}$. An element of $\child{v}$ is
called a {\it child} of $v.$ For a given vertex $v \in V,$ consider
the set $$\mathsf{Par}(v)\,:=\,\{u \in V : (u, v) \in E\}.$$ If
$\mathsf{Par}(v)$ is singleton, then the unique vertex in
$\mathsf{Par}(v)$ is called the {\it parent} of $v$, which we denote
by $\parent{v}.$ Define the subset $\mathsf{Root}(\mathscr T)$ of
$V$ by
$$\mathsf{Root}(\mathscr T) \,:=\, \{v \in V : \mathsf{Par}(v)  = \emptyset\}.$$
An element of $\mathsf{Root}(\mathscr T)$ is called a {\it root} of $\mathscr T$. If $\mathsf{Root}(\mathscr T)$ is singleton, then its unique element
is denoted  by $\mathsf{root}$.
We set $V^\circ:=V \setminus \mathsf{Root}(\mathscr T)$.
A directed graph $\mathscr T= (V,E)$ is called a {\it directed tree} if
$\mathscr T$ has no circuits, $\mathscr T$ is connected and
each vertex $v \in V^\circ$ has a unique parent. A subgraph of a directed tree $\mathscr T$ which itself is a directed tree is said to be a {\it subtree} of $\mathscr T$.
A directed tree $\mathscr T$ is said to be
\begin{enumerate}
\item[(i)] {\it rooted} if it has a unique root.
\item[(ii)]
{\it locally finite} if $\mbox{card}(\child u)$ is finite for all $u \in V,$ where  $\mbox{card}(X)$ stands for the cardinality of the set $X.$
\item[(iii)]
{\it leafless} if every vertex has at least one child.
\item[(iv)] {\it narrow} if there exists a positive integer $\mf m$ such that
\beq \label{narrow-c}
\mbox{card}(\childn{n}{\rootb}) \,\leqslant\, \mf m, \quad n \in \mathbb N.
\eeq
The smallest positive integer $\mf m$ satisfying \eqref{narrow-c} will be referred to as the {\it width} of $\mathscr T.$
\end{enumerate}
\begin{remark}
Note that any narrow directed tree is necessarily locally finite.
However, the converse is not true. Consider, for instance, the
binary tree (see \cite[Example 4.3.1]{Jablonski}). It is worth
noting that there exist narrow directed trees with
$\mbox{card}(V_{\prec})=\aleph_0$, where $V_{\prec}$ denotes the set
of {\it branching vertices} of $\mathscr T$ defined by \beqn
V_{\prec}\,:=\,\{u\in V: \mbox{card}(\mathsf{Chi}(u)) \Ge 2\} \eeqn (see
Figure \ref{fig-narrow}). This is not possible if the directed tree
is leafless.
\end{remark}

\index{$V_{\prec}$}
\index{$\mbox{card}(X)$}
\index{$V^{\circ}$}

\begin{figure}
\begin{tikzpicture}[scale=.8, transform shape]
\tikzset{vertex/.style = {shape=circle,draw, minimum size=1em}}
\tikzset{edge/.style = {->,> = latex'}}
\node[vertex] (b) at  (2,0) {$1$};
\node[vertex] (c) at  (2,2) {$0$};
\node[vertex] (d) at  (4, 0) {$3$};
\node[vertex] (e) at  (4, 2) {$2$};
\node[vertex] (f) at  (6, 0) {$5$};
\node[vertex] (g) at  (6, 2) {$4$};
\node[vertex] (h) at  (8, 0) {$7$};
\node[vertex] (i) at  (8, 2) {$6$};
\node[vertex] (j) at  (10, 0) {$\ldots$};
\node[vertex] (k) at  (10, 2) {$\ldots$};

\draw[edge] (c) to (b);
\draw[edge] (e) to (d);
\draw[edge] (c) to (e);
\draw[edge] (g) to (f);
\draw[edge] (e) to (g);
\draw[edge] (g) to (i);
\draw[edge] (i) to (h);
\draw[edge] (i) to (k);
\draw[dashed] [->] (k) to (j);
\end{tikzpicture}
\caption{A narrow tree $\mathscr T=(V, E)$ with width $\mf m=2$ and $V_{\prec}=2\mathbb N$} \label{fig-narrow}
\end{figure}

\index{$\dep_v$}

Let $\mathscr T=(V, E)$ be a rooted directed tree with root $\rootb$.
For each $u \in V$, the {\it depth} of $u$ is the unique non-negative integer $\dep_{u}$
such that
$u \in \mathsf{Chi}^{\langle \dep_u\rangle}(\mathsf{root})$ (see \cite[Corollary 2.1.5]{Jablonski}).
We discuss here the convergence of nets associated with rooted directed trees induced by the depth.  Define the relation $\leq$ on $V$
as follows:
\beq
\label{p-order}
v ~\leq ~w ~~\mbox{if}~~\dep_v ~\Le ~\dep_w,
\eeq
where $\dep_v$ denotes the depth of $v$ in $\mathscr T.$
Note that $V$ is a partially ordered set with {\it partial order relation} $\leq$, that is, $\leq$ is reflexive and transitive. Further, if $V$ is infinite, then given two vertices $v, w \in V$, there exists $u \in V$ such that $v \leq u$ and $w \leq u$ (the reader is referred to the discussion prior to \cite[Remark 3.4.1]{CPT} for details).
In this text, we will frequently be interested in the nets $\{\mu_v\}_{v \in V}$ of complex numbers induced by the above partial order
(the reader is referred to \cite[Chapter 2]{Si} for the definition and elementary facts pertaining to nets).

\index{$\mathsf{Des}(\cdot)$}
\index{$\cupdot$}
\index{$\parentn{n}{\cdot}$}
\index{$\asc{\cdot}$} \index{$[u, v]$}

Let $\mathscr T=(V, E)$ be a directed tree. For a vertex $u \in V,$
we set $\parentn{0}{u}=u.$ Note that the correspondence
$\parent{\cdot} : u \mapsto \parent{u}$ is a partial function in
$V$. For a positive integer $n$, by the partial function
$\parentn{n}{\cdot}$, we understand $\parent{\cdot}$ composed with
itself $n$-times. The {\it descendant} of a vertex $u \in V$ is
defined by
$$\mathsf{Des}(u)\,:=\, \bigcupdot_{n=0}^{\infty}
\childn{n}{u}~(\mbox{disjoint sum})
$$
(see the discussion prior to \cite[Eqn
(2.1.10)]{Jablonski}).
Note that $\mathscr T_u=(\des{u}, E_u)$ is a rooted subtree of
$\mathscr T$ with root $u$, where \beq \label{E-u} E_u \,:=\,\{(v, w)
\in E : v, w \in \des{u}\}. \eeq The {\it ascendant} or {\it
ancestor} of a vertex $u \in V$ is defined by $$\asc{u}\,:=\,
\{\parentn{n}{u} : n \Ge 1\}.$$ In particular, a vertex is its
descendant, while it is {\it not} its ascendant. Although this is
not a standard practice, we find it convenient. Note that a directed
path from $u$ to $v$ in $\mathscr T$, denoted by $[u, v],$ is unique
whenever it exists. Indeed, since there exists a path from $u$ to
$v$ in $\mathscr T$,  $v \in \des{u}$ and $\dep_v \geqslant \dep_u$.
In this case, it is easy to see that \beqn [u, v] \,=\, \{\parentn{n}{v}
: n =\dep_v-\dep_u, \dep_v-\dep_u-1, \ldots, 0\}. \eeqn Further, for
$u, v \in V,$ we set \beqn (u, v]\,:=\, \begin{cases} [u, v] \setminus
\{u\} & \mbox{if~}
v \in \des{u}, \\
\emptyset & \mbox{otherwise}. \end{cases} \eeqn We also need the
following subsets of $V$: For $u \in V$ and $v \in \des{u},$
 \beq
 \label{desb}
 \left.
 \begin{array}{ccc}
 \mathsf{Des}_v[u] &:=& \des{u} \setminus (u, v] \\
\mathsf{Des}_v(u) &:=& \des{u} \setminus [u, v].
 \end{array}
\right\} \eeq
Note that $\mathsf{Des}_u[u]=\des{u}$  and $\mathsf{Des}_u(u)=\des{u} \setminus \{u\}.$

\index{$\mathscr T_u=(\des{u}, E_u)$}
\index{$\mathsf{Des}_v[u]$}
\index{$\mathsf{Des}_v(u)$}
\index{$(u, v]$}

\section{Hilbert space operators}

For a subset $\Omega$ of the complex plane $\mathbb C$, let ${\rm int}(\Omega)$, ${\overline {\Omega}}$ and
${\mathbb C} \setminus \Omega$  denote  the interior, the closure and the complement of $\Omega$
in $\mathbb C$ respectively. We use $\mathbb R$ to denote the real line, and  ${\Re}z$ and
${\Im}z$ denote the real and imaginary parts of a
complex number $z$ respectively.
The  conjugate of the complex number $z$ will be denoted by $\bar{z}$, while $\arg(z)$ stands for the argument of a non-zero complex number $z$.
We reserve the notation $\mathbb N$ for the set of non-negative integers, while $\mathbb Z$ (resp. $\mathbb Z_+$) stands for the set of all integers (resp. all positive integers).
Unless stated otherwise, all the Hilbert spaces occurring below are complex,
infinite-dimensional and separable.
Let ${\mathcal H}$ be a complex,
separable Hilbert space with the inner product
 $\langle \cdot,\cdot \rangle_{\mathcal H}$ and the corresponding norm
$\|\cdot\|_{\mathcal H}$.
Whenever there is no ambiguity,
we will suppress the suffix and simply
write $\inp{x}{y}$ and  $\|x\|$ in
place of $\inp{x}{y}_{{\mathcal H}}$ and  $\|x\|_{{\mathcal H}}$ respectively.
By $\mbox{span} \{x \in \mathcal H: x \in W\}$ (resp. $\bigvee \{x \in \mathcal H: x \in W\}$),
we mean the smallest linear subspace (resp. smallest closed linear
subspace)
generated by the subset $W$ of $\mathcal H.$ In case $W=\{x\}$, we use the simpler notation $[x]$ for the linear span of $W.$ The orthogonal complement of a closed subspace $W$ of $\mathcal H$ is denoted by $\mathcal H \ominus W$.  Sometimes  $\mathcal H \ominus W$ is denoted by  $W^{\perp}$.
\index{$\mathbb C$}
\index{$\mathbb R$}
\index{$\Re z$}
\index{$\Im z$}
\index{$\bar{z}$}
\index{$\arg{z}$}
\index{$\mathbb N$}
\index{$\mathbb Z_+$}
\index{$\mbox{span}\,W$}
\index{$\bigvee W$}
\index{$[x]$}
\index{$\mathcal H \ominus W=W^{\perp}$}

\index{$\mathscr D(S)$}
\index{$\ker S$}
\index{$\mbox{ran}\,S$}
\index{$\sigma_p(S)$}
\index{$\pi(S)$}
\index{$\sigma(S)$}
\index{$\rho(S)$}
\index{$B(\mathcal H)$}
\index{$R_S(\cdot)$}
\index{$d_S(\cdot)$}
\index{$\mathscr E_S(\mu)$}
\index{${\mf m}_S(\cdot)$}
\index{$\mbox{ind}_S(\cdot)$}
\index{$\sigma_e(S)$}

Let $S$ be a densely defined linear operator in $\mathcal H$ with
domain ${\mathscr D}(S).$ The symbols $\ker S$ and $\mbox{ran}\,S$
will stand for the kernel of $S$ and the range of $S$ respectively.
We use  $\sigma_p(S)$, $\sigma_{ap}(S)$, $\sigma(S)$ to  denote the
point spectrum,  the approximate-point spectrum, and the  spectrum
of $S$ respectively. It may be recalled that $\sigma_p(S)$ is the
set of eigenvalues of $S$, that $\sigma_{ap}(S)$ is the set of those
$\lambda$ in $\mathbb C$ for which $S-\lambda$ is not bounded below,
and that $\sigma(S)$ is the complement of the set of those $\lambda$
in $\mathbb C$ for which ${(S-\lambda)}^{-1}$ exists as a bounded
linear operator on $\mathcal H$. Here, by $S-\lambda$, we understand
the linear operator $S-\lambda I$ with $I$ denoting the identity
operator on $\mathcal H$. We reserve the symbol $B(\mathcal H)$ for
the unital $C^*$-algebra of bounded linear operators on $\mathcal
H$.  The {\it resolvent set} $\rho(S)$ of $S$ is defined as the
complement of $\sigma(S)$ in $\mathbb C$. The {\it resolvent
function} $R_S : \rho(S) \rar B(\mathcal H)$ is given by \beqn
R_S(\lambda)\,:=\,(S-\lambda)^{-1}, \quad \lambda \in \rho(S). \eeqn
The {\it regularity domain} $\pi(S)$ of $S$ is defined as the
complement of $\sigma_{ap}(S)$ in $\mathbb C$. For $\mu \in \pi(S)$,
we refer to the linear subspace $\mbox{ran}(S - \mu)^{\perp}$ of
$\mathcal H$ the {\it deficiency subspace} of $S$ at $\mu$ and its
dimension \beq \label{d-S-mu} d_S(\mu) \,:=\, \dim \mbox{ran}(S -\mu)^{\perp}, \eeq the {\it
defect number of $S$ at $\mu$}.
By the {\it multiplicity function} of $S$, we understand the function ${\mf m}_S : \sigma_p(S) \rar \mathbb Z_+ \cup \{\aleph_0\}$ assigning with each eigenvalue $\lambda$ of $S$, the dimension of the eigenspace $\mathscr E_S(\lambda)$ of $S$ corresponding to $\lambda$. We extend ${\mf m}_S$ to the entire complex plane by setting
 ${\mf m}_S(\lambda)=0,$ $\lambda \in \mathbb C \setminus \sigma_p(S).$
We say that a densely defined linear operator $S$ in $\mathcal H$ is {\it Fredholm} if the range of $S$ is closed, $\dim \ker S$ and $\dim \ker S^*$ are finite. The {\it essential spectrum} $\sigma_e(S)$ of $S$ is the complement of the set of those $\lambda \in \mathbb C$ for which $S-\lambda$ is Fredholm. The {\it Fredholm index} $\mbox{ind}_S : \mathbb C \setminus \sigma_e(S) \rar \mathbb Z$ is given by 
\beqn \mbox{ind}_S(\lambda)\,:=\,\mf m_S(\lambda)-\mf m_{S^*}(\bar{\lambda}), \quad \lambda \in \mathbb C \setminus \sigma_e(S) \eeqn
(the reader is referred to \cite{K, M, Sc} for elementary properties of various spectra of unbounded linear operators).

Let $T$ be a densely defined linear operator in $\mathcal H$ with
domain ${\mathscr D}(T).$
The closure (resp. adjoint) of $T$ is denoted by $\overline{T}$
(resp. $T^*$), whenever it exists.  A subspace ${\mathscr D}$ of
${\mathcal H}$ is said to be a {\it core} of a closable linear
operator $T$ if \beqn {\mathscr D} \subseteq{\mathscr D}(T),~
\overline{\mathscr D}=\mathcal H,~ \mbox{and}~
\overline{T|_{{\mathscr D}}}=\overline{T}.\eeqn
If $S$ is a linear operator in ${\mathcal H}$ such that $$\mathscr D(S) \subseteq \mathscr D(T)~\mbox{and~}Sh=Th~\mbox{for~every~}h \in
\mathscr D(S),$$ then we say that {\it $T$ extends $S$} (denoted by
$S \subseteq T$). Note that two operators $S$ and $T$ are same if
and only if $S \subseteq T$ and $T \subseteq S$. A closed linear
subspace $\mathcal M$ of $\mathcal H$ is said to be {\it invariant}
for $T$ if $T (\mathcal M \cap \mathscr D(T)) \subseteq {\mathcal
M}.$ In this case, the {\it restriction of $T$ to $\mathcal M$} is
denoted by $T|_{\mathcal M}$. Note that if $T$ has {\it invariant
domain}, that is, $T \mathscr D(T) \subseteq \mathscr D(T)$, then
$T$ admits polynomial calculus in the sense that $p(T)$ is a
well-defined linear operator with domain $\mathscr D(T)$ for every
complex polynomial $p$ in one variable. A closed linear subspace
$\mathcal M$ of $\mathcal H$  is {\it reducing} for $T$ if there
exist linear operators $T_0$ in $\mathcal M$ and $T_1$ in $\mathcal
M^{\perp}$ such that $T= T_0 \oplus T_1.$
The {\it commutant} of a linear operator $T$ is given by
\beqn
\{T\}'\,:=\,\{A \in B(\mathcal H) : AT \subseteq TA \}.
\eeqn
In case $T \in B(\mathcal H),$
$\{T\}'=\{A \in B(\mathcal H) : AT = TA \}.$
If $P_{\mathcal M}$ is an orthogonal projection of $\mathcal H$ onto a closed subspace $\mathcal M$ of $\mathcal H$, then $P_{\mathcal M} \in \{T\}'$ if and only if $\mathcal M$ is a reducing subspace for $T$ (see \cite[Proposition 1.15]{Sc}).
\index{$\overline{T}$}
\index{$T^*$}
\index{$S \subseteq T$}
\index{$T\vert_{\mathcal M}$}
\index{$T_0 \oplus T_1$}
\index{$\{T\}'$}
\index{$P_\mathcal M$}

We recall definitions of some well-studied classes of unbounded
linear operators, which are relevant to the present investigations
(refer to \cite{Ru, O, M, Sc}). A densely
defined linear operator $T$ in a complex Hilbert space $\mathcal H$
is said to be
\begin{enumerate}
\item {\it self-adjoint} if $\mathscr D(T) = \mathscr D(T^*)$ and $T^*x = Tx$ for all $x \in \mathscr D(T).$
\item {\it normal} if $\mathscr D(T) = \mathscr D(T^*)$ and $\|T^*x\| = \|Tx\|$ for all $x \in \mathscr D(T).$
\item {\it nilpotent} if $T$ has invariant domain such that $T^n=0$ for some positive integer $n$. The smallest positive integer with this property is referred to as the {\it nilpotency index} of $T$.
\item {\it complex Jordan} if there exist a normal operator $M$ and a nilpotent linear operator
$N$ such that any one of the following holds:
\begin{enumerate}
\item $M \in B(\mathcal H)$, $T=M+N$ and $M \in \{N\}',$
\item $N \in B(\mathcal H)$, $T=M+N$ and $N \in \{M\}'.$
\end{enumerate}
\end{enumerate}

\section{Rank one operators}

We will see that the rank one (possibly unbounded) operators form building blocks in the orthogonal decomposition of weighted join operators (see Theorem \ref{o-deco}). Hence we find
it necessary to collect below several elementary properties of rank one operators.

\index{$x \otimes y$}
Let $\mathcal H$ be  a complex Hilbert space. For $x, y \in \mathcal H,$ the
{\it injective tensor product} $x \otimes y$ of $x$ and $y$ is
defined by \beqn x \otimes y(h)\,=\,\inp{h}{y}\, x, \quad h \in \mathcal
H. \eeqn
Clearly, $x \otimes y$ is a rank one bounded linear
operator. In fact, $\|x \otimes y\|=\|x\|\|y\|.$ Conversely, every
rank one bounded linear operator arises in this fashion. Indeed, if
$T \in B(\mathcal H)$ is a rank one operator with range spanned by a
unit vector $y \in \mathcal H,$ then for every $x \in \mathcal H,$
$Tx = \alpha_x y$ for some scalar $\alpha_x \in \mathbb C$, and
hence \beqn Tx = \inp{Tx}{y}y = \inp{x}{T^*y}y = (y \otimes
T^*y)(x), \quad x \in \mathcal H \eeqn (cf. \cite[Proposition
2.1.1]{Si1}). It is worth noting that a diagonal operator $D_{\lambdab}$ on
$\mathcal H$ with respect to an orthonormal basis $\{e_j\}_{j \in
J}$ and diagonal entries $\lambdab:=\{\lambda_j\}_{j \in J} \subseteq \mathbb
C$ (counted with multiplicities) can be rewritten uniquely (up to permutation) in terms
of injective tensor products as follows: \beqn D_{\lambdab} \,=\, \sum_{j \in J}
\lambda_j\, e_j \otimes e_j, \eeqn where $J$ is a directed set.
Note that $D_{\lambdab}$ is a bounded linear operator on $\mathcal H$ if and
only if $\lambdab$ forms a bounded subset of $\mathbb
C$.

The analysis of weighted join operators relies on a thorough study of bounded and unbounded rank one operators. As we could not locate an appropriate reference for a number of facts essential in our investigations, we include their statements and elementary verifications.


\begin{lemma} \label{lem-inj-ten}
Let $\mathcal H$ be a complex Hilbert space of dimension bigger than
$1$. For unit vectors $x, y, z, w \in \mathcal H,$ we have the
following statements:
\begin{enumerate}
\item $x \otimes y$ is an algebraic operator:
\beq \label{algebraic}
p(x \otimes y)=0~\mbox{with}~p(\lambda)=\lambda(\lambda-\inp{x}{y}), ~\lambda \in \mathbb C.
\eeq
\item $\sigma(x \otimes y)=\{0, \inp{x}{y}\}=\sigma_p(x\otimes y).$
Moreover, the eigenspace $\mathscr E_{x \otimes y}(\mu)$
corresponding to the eigenvalue $\mu$ of $x \otimes y$ is given by
\beqn \mathscr E_{x \otimes y}(0)=[y]^{\perp}, \quad \mathscr E_{x
\otimes y}(\inp{x}{y})=[x]. \eeqn Thus the multiplicity ${\mf m}_{x
\otimes y}(\mu)$ of the eigenvalue $\mu$ is given by \beqn {\mf
m}_{x \otimes y}(0)=\aleph_0~\mbox{if~}\dim \mathcal H =\aleph_0,
\quad {\mf m}_{x \otimes y}(\inp{x}{y})=1~\mbox{if~} \inp{x}{y} \neq
0. \eeqn
\item
The resolvent function $R_{x \otimes y} :  \rho(x \otimes y) \rar B(\mathcal H)$ of $x \otimes y$ at $\mu$ is given by
\beqn
R_{x \otimes y}(\mu) \,=\, -\frac{1}{p(\mu)} \Big((\mu-\inp{x}{y})P_{[x]^{\perp}} + x \otimes \big(P_{[x]^{\perp}}y + \bar{\mu}\,x\big)\Big),
\eeqn
where $p$ is as given in \eqref{algebraic}.
\item The commutant $\{x \otimes y\}'$ of $x \otimes y$  is given by
\beqn \{x \otimes y\}' \,=\, \big\{A \in B(\mathcal H): Ax = \inp{Ax}{x}x,
~A^*y = \overline{\inp{Ax}{x}}y\big\}. \eeqn
\item $x \otimes y$ is normal if and only if there exists a unimodular scalar $\alpha \in \mathbb C$ such that $x = \alpha y$.
\item $x \otimes y$ is self-adjoint if and only if  $x =\pm\, y$.
\end{enumerate}
\end{lemma}
\begin{proof}
It is easy to see the following:
\beq \label{adjoint-ro}
x \otimes \alpha y &=& \bar{\alpha}\,x \otimes y, ~\alpha \in \mathbb C, \notag \\
(x \otimes y)^* &=& y \otimes x, \quad (x \otimes y)(z \otimes w)\,=\,
\inp{z}{y}\, x \otimes w. \eeq
To see (i), note that by \eqref{adjoint-ro}, $x \otimes y$
satisfies \beqn (x \otimes y)^2 \,=\, \inp{x}{y} x \otimes y. \eeqn To
see (ii), note that \beqn (x\otimes y)(z)=0 ~\mbox{if~}z \in
[y]^{\perp}, \quad (x\otimes y -
 \inp{x}{y})x=0.
\eeqn Further, by (i) and the spectral mapping property for
polynomials \cite{Si1}, \beqn p(\sigma(x \otimes y))=\sigma(p(x
\otimes y))=\{0\}, \eeqn where $p(z)=z(z - \inp{x}{y})$. The desired
conclusions in (ii) are now immediate.

To see the formula for the resolvent function $R_{x \otimes y}$ in (iii),
let $f, g \in \mathcal H$ be such that $(x \otimes y
-\mu)f=g.$ Writing $h=P_{[x]^{\perp}}h + \inp{h}{x}x$ for $h \in
\mathcal H,$ and comparing coefficients, we obtain \beqn
P_{[x]^{\perp}}f= -\frac{1}{\mu}P_{[x]^{\perp}}g, \quad
\inp{g}{x}=\inp{f}{y-\bar{\mu}\, x}. \eeqn It follows that \beqn
\inp{g}{x} &=& \inp{P_{[x]^{\perp}}f + \inp{f}{x}x}{y-\bar{\mu}\, x}
\\ &=&
\inp{P_{[x]^{\perp}}f}{y} + \inp{f}{x} (\inp{x}{y} - \mu) \\
&=& -\frac{1}{\mu}\,\inp{P_{[x]^{\perp}}g}{y} + \inp{f}{x} (\inp{x}{y} - \mu).
\eeqn
This yields
\beqn
f &=& P_{[x]^{\perp}}f + \inp{f}{x}x \\ &=&
-\frac{1}{\mu}P_{[x]^{\perp}}g + \frac{1}{\inp{x}{y} - \mu} \Big(\inp{g}{x} + \frac{1}{\mu}\,\inp{P_{[x]^{\perp}}g}{y}\Big) x.
\eeqn
It is now easy to see that $R_{x \otimes y}(\mu)$ has the desired expression.

To see (iv), note that $A$
commutes with $x \otimes y$ if and only if \beq \label{comm-r-1}
 \inp{z}{y} Ax \,=\, \inp{Az}{y}x, \quad z \in \mathcal H.
\eeq If $z \in [y]^{\perp}$, then $\inp{Az}{y}x=0$, and hence $Az
\in [y]^{\perp}.$ This shows that $A^*$ maps $[y]$ into $[y].$
Letting $z=y$ in \eqref{comm-r-1}, we obtain the necessity part of
(iv). Conversely, if  $Ax = \inp{Ay}{y}x,$ then \eqref{comm-r-1} is
equivalent to $\inp{A(z-\inp{z}{y}y)}{y}= 0$, which is equivalent to
$A^*y =\overline{\inp{Ay}{y}}\,y$. To see (v), note that by
\eqref{adjoint-ro}, $ (x \otimes y)^*=y \otimes x,$ and apply (iv)
to $A = y \otimes x.$
The sufficiency part of (vi) follows immediately from \eqref{adjoint-ro}.
Assume that $(x \otimes y)^* = x \otimes y.$ By \eqref{adjoint-ro},
$x \otimes y =y\otimes x.$ By (v), $x = \alpha y$ for some $\alpha
\in \mathbb C.$ Thus \beqn  \alpha y \otimes y = x \otimes y = y \otimes
x = \bar{\alpha} y \otimes y.\eeqn It follows that $\alpha \in
\mathbb R.$ Since $x, y$ are unit vectors, $\alpha = \pm\, 1$.
\end{proof}
Recall that a densely define linear operator $T$ in $\mathcal H$
{\it admits a compact resolvent} if there exists $\lambda \in
\mathbb C \setminus \sigma(T)$ such that $(T-\lambda)^{-1}$ is
compact. It may be concluded from Lemma \ref{lem-inj-ten}(iii) that
$x \otimes y$ has a compact resolvent if and only if $\dim \mathcal
H$ is finite.

Let us discuss a class of  unbounded, densely defined, but not
necessarily  closed  (rank one) operators, which we denote as  $f
\obslash g$, $f \in \mathcal H$ and $g$ is unspecified (to be chosen
later). Fix an orthonormal basis $\{e_j\}_{j \in J}$ of $\mathcal H$
for some directed set $J$, and let $g=\{g(j)\}_{j \in J}$.
Define $f \obslash g$ in $\mathcal H$ by
\beqn
\mathscr D(f \obslash g) &=& \Big\{h \in \mathcal H : \sum_{j \in J} h(j)\overline{g(j)} \mbox{~is convergent} \Big\}, \\
f \obslash g (h) &=& \Big(\sum_{j \in
J}^{\infty}h(j)\overline{g(j)}\Big)f, \quad h \in \mathscr D(f
\obslash g). \eeqn
Note that $f \obslash g$ is densely defined with
$\mbox{span}\{e_j : j \in J\} \subseteq \mathscr D(f \obslash g).$
Thus the Hilbert space adjoint $(f \obslash g)^*$ of $f \obslash g$ is
well-defined.
Recall that for $p \in [1, \infty]$, $\ell^p(J)$ is the Banach space of all $p$-summable complex functions  $f : J \rar \mathbb C$ endowed with the norm
\beqn
\|f\|_p \,=\, \begin{cases} \Big(\displaystyle \sum_{j \in J} |f(j)|^p \Big)^{1/p} & \mbox{if~} p < \infty, \\
\displaystyle \sup_{j \in J}|f(j)| & \mbox{if~}p=\infty.
\end{cases}
\eeqn
It turns out that $(f \obslash g)^*$ is not densely
defined unless $g \in \ell^2(J).$ Indeed, we have the following result:
\index{$f \obslash g$}
\index{$\ell^p(J)$}

\begin{lemma} \label{lem-inj-ten-unb}
Let $J$ be an infinite directed set.
Fix an orthonormal basis $\{e_j\}_{j \in J}$ of $\mathcal H$,
$f \in \mathcal H \setminus \{0\}$, and let $g=\{g(j)\}_{j \in J}$. Then the following statements are equivalent:
\begin{enumerate}
\item $f\obslash g$ admits a  bounded linear extension to $\mathcal H$.
\item $f \obslash g$ is closed.
\item $f \obslash g$ is closable.
\item $g \in \ell^2(J)$.
\item $\mathscr D(f \obslash g)=\mathcal H$.
\end{enumerate}
In case $g \notin \ell^2(J)$,
$(f \obslash g)^*$ is not densely defined,
$\sigma(f \obslash g)=\mathbb C$, and
\beqn
\sigma_p(f \obslash g)\,=\, \begin{cases} \Big\{0, ~\displaystyle \sum_{j \in J} f(j)\overline{g(j)} \Big\} & \mbox{if~}f \in \mathscr D(f \obslash g), \\
\{0\} & \mbox{otherwise}.
\end{cases}
\eeqn
\end{lemma}
\begin{proof}
To see the equivalence of (i)-(v), it suffices to check that (iii)
$\Rightarrow$ (iv) and (v) $\Rightarrow$ (i). Suppose that $f
\obslash g$ is a closable operator with closed extension $A.$ Since
$A$ is a densely defined closed operator, by \cite[Theorem
1.5.15]{M}, $\mathscr D(A^*)$ is dense in $\mathcal H$ (see also
\cite[Theorem 1.8]{Sc}). Assume that $g \notin \ell^2(J)$. Then $h
\in \mathscr D(A^*)$ if and only if there exists a positive real
number $c$ such that \beq \label{domain-adjoint} |\inp{Ax}{h}|\,\leqslant\, c
\|x\|, \quad x \in \mathscr D(A). \eeq However, for $x=\sum_{j \in
F} g(j) e_j$ with $F \subseteq J$ and $\mbox{card}(F) < \infty$,
\beqn \inp{Ax}{h} \,=\, \sum_{j \in F} |g(j)|^2 \inp{f}{h}. \eeqn It
follows that \beqn \Big(\sum_{j \in F}|g(j)|^2\Big)^{1/2}
|\inp{f}{h}| ~\Le ~c. \eeqn Since $g \notin \ell^2(J)$, we must have
$\inp{f}{h}=0.$ This shows that $\mathscr D(A^*) \subseteq \mathcal
H \ominus [f].$ Since $f$ is non-zero, this contradicts the fact
that the Hilbert space adjoint $A^*$ of $A$ is densely defined. This
completes the verification of the implication (iii) $\Rightarrow$
(iv).

The implication (v) $\Rightarrow$ (i) may be derived from the
uniform boundedness principle \cite{Co} applied to the family $\{f
\otimes g_n\}_{n \geqslant 0}$ of bounded linear operators, where
$\{g_n\}_{n \geqslant 0}$ is any finitely supported sequence
converging pointwise to $g.$ To see the remaining part, assume that
$g \notin \ell^2(J).$ Arguing as in the preceding paragraph (with
$A$ replaced by $f \obslash g$), we obtain $\mathscr D((f \obslash
g)^*) \subseteq\mathcal H \ominus [f].$ The assertion that $\sigma(f
\obslash g)=\mathbb C$ follows from the fact that any densely
defined operator with non-empty resolvent set is closed (see
\cite[Lemma 1.17]{CM}). Note that any $h \in \mathscr D(f \obslash
g)$ such that $\sum_{j \in J}h(j)\overline{g(j)}=0$ (there are
infinitely many such vectors $h$) is an eigenvector for $f \obslash
g$ corresponding to the eigenvalue $0.$ Suppose $f \in \mathscr D(f
\obslash g).$ Then, as in the proof of Lemma \ref{lem-inj-ten}(iii),
it can be seen that $\sum_{j \in J} f(j)\overline{g(j)} $ is an
eigenvalue of $f \obslash g$ corresponding to the eigenvector $f.$
Since any eigenvector of $f \obslash g$ corresponding to a non-zero
eigenvalue belongs to $[f]$, we conclude in this case that
$$\sigma_p(f \obslash g)\,=\,\Big\{0, \sum_{j \in J} f(j)\overline{g(j)}
\Big\}.$$ This also shows that if $f \notin \mathscr D(f \obslash
g)$, then $f \obslash g$ can not have a non-zero eigenvalue.
\end{proof}

In case $f \obslash g$ is closable, the bounded extension of $f
\obslash g$, as ensured by Lemma \ref{lem-inj-ten-unb}, is precisely
$f \otimes g.$ Otherwise, it may be concluded from
\eqref{domain-adjoint} that $\mathscr D((f \obslash g)^*)=\mathcal H
\ominus [f].$ 
We conclude this section with a brief discussion on an interesting
family of unbounded rank one operators.
\begin{example}
Let $J$ be an infinite directed set. Let
$f\in \ell^2(J) \setminus \{0\}$ and let $g \in\ell^p(J)$, $1 \leqslant p \leqslant \infty$. Then $f \obslash g$ defines a densely defined rank one operator in $\ell^2(J)$ with domain $\mathscr D(f \obslash g)=\ell^q(J) \cap \ell^2(J)$, where $1 \leqslant q \leqslant \infty$ is such that $\frac{1}{p} + \frac{1}{q} = 1.$ Indeed, since
\beqn
\|(f \obslash g)(h)\|_2 \,\leqslant\, \|h\|_q \|g\|_p \|f\|_2, \quad h \in \ell^q(J),
\eeqn
$f \obslash g$ is bounded when considered as a linear transformation from $\ell^q(J)$ into $\ell^2(J)$. Moreover, since $\ell^p(J) \subseteq \ell^2(J)$ for $1 \leqslant p \leqslant 2,$ by Lemma \ref{lem-inj-ten-unb}, $f \obslash g$ belongs to $B(\ell^2(J))$ if and only if $1 \leqslant p \leqslant 2.$ Thus for $g \in \ell^p(J) \setminus \ell^2(J)$ for some $2 < p \leqslant  \infty$, $f \obslash g$ is a densely defined unbounded rank one operator in $\ell^2(J)$ with domain $\ell^q(J) \cap \ell^2(J).$
\eop
\end{example}

\chapter*{Prologue}

In the following discussion, we attempt to explain some important aspects of the present work with the aid of a family of rank one perturbations of weighted join operators. In the following exposition, we have tried to minimize the graph-theoretic prerequisites.  In particular, we avoided the rather involved graph-theoretic definition of the class $\{\W{b}{u} : u, \mf b \in V\}$ of the so-called weighted join operators. 

\index{$U^{(\mf b)}_u$}
Let $\mathscr T=(V, E)$ be a rooted directed tree with root $\rootb$ and
let $\mf b, u \in V$.
Consider
the closed subspace $\ell^2(U^{(\mf b)}_u)$ of $\ell^2(V)$, where the subset $U^{(\mf b)}_u$ of $V$ is given by
\beq \label{U-u}
U^{(\mf b)}_u &=& \begin{cases}
V \setminus \{u\} & \mbox{if~}\mf b = u,\\
\displaystyle \asc{u} \cup  \desb{u}
 & \mbox{if~} \mf b \in \mathsf{Des}_u(u),  \\
 \displaystyle  \mathsf{Des}_u(u)   & \mbox{otherwise}
\end{cases}
\eeq 
(see \eqref{desb}).
Consider the standard orthonormal basis $\{e_v\}_{v \in U^{(\mf b)}_u}$ of $\ell^2(U^{(\mf b)}_u)$ and let $D^{(\mf b)}_u$ denote the diagonal operator in $\ell^2(U^{(\mf b)}_u)$ defined as
\beqn
D^{(\mf b)}_u e_v = \lambda_{uv} e_v, \quad v \in U^{(\mf b)}_u,
\eeqn
where the diagonal entries of $D^{(\mf b)}_u$ are given by
\beq \label{wt-Dub} \lambda_{uv}:=\dep_v -
\dep_u, \quad v \in U^{(\mf b)}_u \eeq with $\dep_v$ denoting the depth of $v
\in V$ in $\mathscr T$. Consider the rank one operator $N^{(\mf
b)}_u=e_u \otimes e_{_{A_u}},$ where $e_{_{A_u}} := \sum_{v \in A_u}
(\dep_v - \dep_u) e_v$ and the subset $A_u$ of $V$ is given by \beqn
A_u &=&
\begin{cases}
 [u, \mf b]
 & \mbox{if~} \mf b \in \des{u}, \\
\asc{u} \cup \{\mf b, u\}  & \mbox{otherwise}.
\end{cases}
\eeqn
We also need the (possibly unbounded) rank one operator $e_w \obslash g_{x}$, where $w \in V \setminus U^{(\mf b)}_u$, $x \in \mathbb R$ and $g_{x} : U^{(\mf b)}_u \rar \mathbb C$ is given by
\beq \label{gxv}
g_x(v)\,=\,\dep^x_v, \quad v \in U^{(\mf b)}_u.
\eeq
From the view point of spectral theory, we will be interested in the
 the family $\mathscr W:=\{W_{w, x} : w \in V \setminus U^{(\mf b)}_u, ~x \in \mathbb R\}$ of rank one extensions of weighted join operators defined as follows:
\beqn
\begin{array}{lll}
\mathscr D(W_{w, x}) &=& \big\{(h, k) : h \in \mathscr D(D^{(\mf b)}_u) \cap \mathscr D(e_w \obslash g_x), ~k \in \ell^2(V \setminus U^{(\mf b)}_u)\big\} \\ && \\
W_{w, x} &=& \left[\begin{array}{cc}
D^{(\mf b)}_u & 0 \\
e_w \obslash g_x & N^{(\mf b)}_u
\end{array}
\right].
\end{array}
\eeqn
Clearly, the linear operator $W_{w, x}$ is densely defined in $\ell^2(V)$. 
Further, it is not difficult to see that $W_{w, x}$ is unbounded unless $D^{(\mf b)}_u$ belongs to $B(\ell^2(U^{(\mf b)}_u))$ and $g_x \in \ell^2(U^{(\mf b)}_u)$.
A lot of conclusions can be drawn about the family $\mathscr W$ of rank one extensions of weighted join operators. The analysis of Chapters 4 and 5 of this paper leads us to the following.

\index{$A_u$}
\index{$g_x$}
\index{$W_{w, x}$}
\begin{theorem} \label{w-x}
Let $\mathscr T=(V, E)$ be a rooted directed tree with root $\rootb$ and
let $\mf b, u \in V.$
Assume that $(\des{u}, E_u)$ is a countably infinite narrow subtree of $\mathscr T$ $($see \eqref{E-u}$)$ and
let $U^{(\mf b)}_u$ be as defined in \eqref{U-u}.
Then, for any $w \in V \setminus U^{(\mf b)}_u$ and $x \in \mathbb R,$
the spectrum of $W_{w, x}$ is a proper subset of $\mathbb C$ if and only if $x < 1/2.$ In case $x \in (-\infty, 1/2)$,  we have the following statements:
\begin{enumerate}
\item $W_{w, x}$ is a closed operator with
domain being the orthogonal direct sum of $\mathscr D(D^{(\mf b)}_u)$ and $\ell^2(V\setminus U^{(\mf b)}_u)$.
\item $\sigma(W_{w, x})=\{\dep_v-\dep_u : v \in U^{(\mf b)}_u \cup \{u\}\}=\sigma_p(W_{w, x})$.
\item 
$\sigma_e(W_{w, x}) \setminus \{0\}= 
\big\{\dep_v-\dep_u :  v \in U^{(\mf b)}_u \cup \{u\}, ~ \mbox{card}(\childn{\dep_v}{\rootb}) = \infty\big\}.$
Moreover, $\mbox{ind}_{\,W_{w, x}}=0$ on $\mathbb C \setminus \sigma_e(W_{w, x}).$
\item $W_{w, x}$ is a sectorial operator, which generates a strongly continuous quasi-bounded semigroup.
\item $W_{w, x}$ is never normal.
\item If, in addition, $\mathscr T$ is leafless, then $W_{w, x}$ admits a compact resolvent if and only if the set
$V_{\prec}$ of branching vertices
of $\mathscr T$ is disjoint from $\asc{u}$.
\end{enumerate}
\end{theorem}
\begin{remark} If $x \geqslant1/2$, then  $\sigma(W_{w, x})=\mathbb C.$ Assume that $x < 1/2$. Then, by \eqref{U-u},  $\sigma(W_{w, x})=\mathbb N$ if $\mf b \notin \des{u}$. Otherwise, \beqn
\big\{k \in \mathbb N : k \notin \{1, \ldots, \dep_{\mf
b}-\dep_{u}\} \big\} \cup \{-1,  \ldots, -\dep_u\} ~\subseteq ~
\sigma(W_{w, x}) \, \subseteq \, \mathbb N \cup  \{-1,  \ldots,
-\dep_u\}.\eeqn The exact description of  $\sigma(W_{w, x})$, $x <
1/2$ depends on the set $V_{\prec} \cap [u, \mf b]$ and the
leaf-structure of branches emanating from this part (see Figure
\ref{fig-spec} for a pictorial representation of the spectra of
$W_{w, x}$, $x \in \mathbb R$, where $x$ varies over the $X$-axis
while the spectra $\sigma(W_{w, x})$ are plotted in the
$Y\!Z$-plane). Note that the spectral behaviour of the family $\{W_{w, x}\}_{x \in \mathbb R}$
changes at $x=1/2$ (see Figure \ref{fig-spec}).
\end{remark}

\begin{figure}
\begin{tikzpicture}[scale=.6, transform shape, circle dotted/.style={dash pattern=on .05mm off 8mm,
                                         line cap=round}]
\tikzset{vertex/.style = {shape=circle,draw,minimum size=1em}}
\tikzset{edge/.style = {->,> = latex'}}

\node (a) at  (0,0) {};
\node (a') at  (-0.15,0) {};
\node (a'') at  (0,-0.15) {};
\node (a''') at  (.15,.15) {};
\node (b) at  (0,5) {};
\node (b') at  (-.8,4.5) {$Z$-axis};
\node (c) at  (-4,-4) {};
\node (c') at  (-4,-3) {$Y$-axis};
\node (d) at  (9.5,0) {};
\node (d') at  (8.5,0.4) {$X$-axis};
\node (e) at  (2,0) {$\bullet$};
\node (e') at  (3.1,-1) {$\sigma(W_{w, x}):$ \mbox{discrete}};
\node (e'') at  (2.1,-1.5) {$x  < 1/2$};
\node (e''') at  (1.5,-0.7) {};
\node (f) at  (-1.5,-4) {};
\node (g) at  (3.7,0.3) {$1/2$};
\node (h) at  (3.7,0) {$\red \bullet$};
\node (h') at  (7.4,-1) {$\sigma(W_{w, x})$: \mbox{plane}};
\node (h'') at  (6.8,-1.5) {$x  \geqslant 1/2$};
\draw[edge] (a'') to (b);
\draw[edge] (a''') to (c);
\draw[edge] (a') to (d);
\draw[line width = .9mm,circle dotted] (e''') to (f);

\tikzset{vertex/.style = {shape=circle,draw,red, minimum size=1em}}

\node (h') at  (5.5,0) {$\bullet$};
\node (i) at  (6,4) {};
\node (i') at  (5.91,4) {};
\node (j) at  (6,-3) {};
\node (j') at  (5,3) {};
\node (j'') at  (4.8,2.7) {};
\node (j''') at  (6,-2.8) {};
\node (k) at  (5,-4.5) {};



\draw (i') to (j);
\draw[dashed] (i) to (j'');
\draw (j') to (k);
\draw[dashed] (j''') to (k);
\end{tikzpicture}
\caption{Spectra of the family $W_{w, x}$, $x \in \mathbb R$} \label{fig-spec}
\end{figure}

A proof of Theorem \ref{w-x} will be presented towards the end of
Chapter 4. We remark that the conclusion of Theorem \ref{w-x} may
not hold in case $\des{u}$ is not a narrow subtree of $\mathscr T$
(see Remark \ref{remark-narrow}). It is tempting to arrive at the
conclusion that almost everything about $W_{w, x}$ can be
determined. However, this is not the case. Although, the Hilbert
space adjoint of $W_{w, x}$ is densely defined in case $-1 \leqslant
x < 1/2,$ neither we know its domain nor we have a neat expression
for $W^*_{w, x}$. Further,
 in case $x \geqslant 1/2$, we do not know whether or not $W_{w, x}$ is closable.

\subsection*{Plan of the paper}

We conclude this chapter with the layout of the paper.

\vskip.2cm

Chapter 2 provides the graph-theoretic framework essential in the
study of weighted join operators. In particular, we introduce the
notion of extended directed tree and exploit the order structure on
a rooted directed tree $\mathscr T=(V, E)$ to introduce a family of
semigroup structures on extended directed trees to be referred to as
join operations at a base point (see Proposition \ref{sg-3}). The
join operation based at $\rootb$ (resp. $\infty$) is precisely the
join (resp. meet) operation. We exhibit a pictorial illustration of
the decomposition of the set $V$ of vertices into descendant,
ascendant, and the rest with respect to a fixed vertex (see
\eqref{cg-dec0} and Figure \ref{fig0}). This decomposition of $V$
helps to understand the action of join and meet operations. We
conclude this section with the description of the set $M^{(\mf
b)}_u(w)$ of vertices which join to a given vertex $u$ at another
given vertex $w$ (see Proposition \ref{lem3.2}).

In Chapter 3, we see that the semigroup structures on an extended
rooted directed tree $\mathscr T=(V, E)$, as introduced in Chapter
2, induce a family of operators referred to as weighted join operators
$\W{b}{u}$. We show that for $\mf b \neq \infty,$ these operators
are closed and that the linear span of standard orthonormal basis of
$\ell^2(V)$ forms a core for $\W{b}{u}$  (see Proposition
\ref{closed}). We further discuss boundedness of weighted join
operators (see Theorem \ref{bdd}). One of the main results in this
chapter decomposes a weighted join operator with a base point in $V$
as the sum of a diagonal operator and a bounded 
operator of rank at most one (see Theorem \ref{o-deco}). In case the base point is
$\infty,$ the weighted meet operator $\W{b}{u}$ turns out to be possibly an unbounded
finite rank operator. Among various properties of weighted join
operators, it is shown that $\W{b}{u}$ has large null summand except
for at most finitely many choices of $u$ (see Corollary \ref{lns}).
It is also shown that the weighted join operator is either a complex
Jordan operator of index $2$ or it has a bounded Borel functional
calculus (see Corollary \ref{coro-c-Jordan}).  We also characterize
compact weighted join operators and discuss an application to the
theory of commutators of compact operators (see Proposition
\ref{cpt} and Corollary \ref{cpt-comm}). We further provide the
description of commutant of a bounded weighted join operator (see
Theorem \ref{commutant}). We exhibit a concrete family of weighted
join operators to conclude that the commutant of a weighted join
operator is not necessarily abelian.

In Chapter 4, we capitalize on the graph-theoretic frame-work developed in
earlier chapters to introduce and study the class of rank one
extensions of weighted join operators.  In particular, we discuss
closedness, structure of Hilbert space adjoint and spectral theory
of rank one extensions of weighted join operators. We introduce two
compatibility conditions which controls the unbounded rank one
component in the matrix representation of rank one extensions of
weighted join operators and discuss their connection with certain
discrete Hilbert transforms (see Proposition \ref{lem-H-transform}).
Moreover, we characterize these conditions under some sparsity
conditions on the weight systems of weighted join operators (see
Proposition \ref{CC-prop}). We provide a complete spectral picture
(including description of spectra, point-spectra and essential
spectra) for rank one extensions of weighted join operators (see
Theorem \ref{spectrum}). It turns out that the spectra of rank one
extensions of weighted join operators satisfying the so-called
compatibility condition I can be recovered from its point spectra.
In case the compatibility condition I does not hold, either rank one
extensions of weighted join operators are not closed or their domain
of regularity is empty (see Corollary \ref{coro-spectrum}). We give
an example of a rank one extension of a weighted join operator with
spectrum properly containing the topological closure of its point
spectrum (see Example \ref{exam-spectrum}). Further, under the
assumption of compatibility condition I, we characterize rank one
extensions of weighted join operators on leafless directed trees
which admit compact resolvent (see Corollary \ref{cpt-res}).
Given an unbounded closed subset $\sigma$ of the complex plane, we construct a non-trivial rank one extension of a weighted join operator with spectrum same as $\sigma$ (see Corollary \ref{construction}).
Finally, we specialize Theorem \ref{spectrum} to weighted join
operators and conclude that these operators are not complete unless
they are complex Jordan (see Theorem \ref{spectrum1} and Corollary
\ref{completeness}).

In Chapter 5, we investigate various special classes of rank one
extensions of weighted join operators. We exhibit a family of rank
one extensions of weighted join operators, which are sectorial (see
Proposition \ref{sectorial}). A similar result is obtained for the
generators of quasi-bounded strongly continuous semigroups (see
Proposition \ref{semigroup}). Further, we characterize the classes
of hyponormal, cohyponormal, $n$-symmetric weighted join operators
(see Propositions \ref{hypo} and \ref{symm}). It turns out that
there are no non-normal hyponormal weighted join operators. On the
other hand, if a weighted join operator is $n$-symmetric, then
either $n=1$ or $n \Ge 3.$ We also discuss normality and
symmetricity of rank one extensions of weighted join operators (see
Propositions \ref{hypo-rone} and \ref{symm-r-one}). Towards the end
of this section, we present a proof of Theorem \ref{w-x}.

In Chapter 6, we discuss the counter-part of the theory of weighted
join operators for rootless directed trees. The definition of join
operation at a given base point becomes less obvious in view of the
absence of depth of a vertex in a rootless directed tree. One of the
main results in this chapter is a decomposition theorem analogous to
Theorem \ref{o-deco} (see Theorem \ref{o-deco-rl}). It turns out
that a weighted join operator $\W{b}{u}$ on a rootless directed tree
could be a possibly unbounded rank one perturbation of (unbounded)
diagonal operator. Unlike the case of rooted tree, the weighted join
operator $\W{b}{u},$ $\mf b \neq \infty,$ need not be even closable
(see Corollary \ref{coro-bdd-rl}). We also obtain a counter-part of
Corollary \ref{coro-c-Jordan} for weighted join operators on
rootless directed trees (see Corollary \ref{C-Jordan-rl}). Finally,
we briefly discuss some difficulties in the study of the rank one
extensions of these operators.

In Chapter 7,  we  glimpse at the general theory of unbounded  non-self-adjoint rank one perturbations of diagonal operators or of the forms associated with diagonal operators.  In particular, we discuss sectoriality of rank one perturbations of diagonal operators and some of its immediate applications to spectral theory. We also discuss the role of some compatibility conditions in the sectoriality of the form-sum of the form associated with a sectorial diagonal operator and a form associated with not necessarily square-summable functions $f$ and $g.$

We conclude the paper with an epilogue including several remarks and possible lines of investigations.

\chapter{Semigroup structures on extended directed trees}

In this chapter, we provide the graph-theoretic frame-work essential to introduce and study the so-called weighted join operators and their rank one extensions on rooted directed trees (see Chapters 3 and 4).
In particular, we formally introduce the notion of extended directed tree and exhibit a family of semigroup structures on it. We also present a canonical decomposition of a rooted directed tree suitable for understanding the actions of weighted join operators.
\section{Join and meet operations on extended directed trees}

In what follows,
we always assume that
$\mbox{card}(V)=\aleph_0.$
The following notion of extended directed tree plays an important role in unifying theories of weighted join operators and weighted meet operators.
\begin{definition}[Extended directed tree]
Let $\mathscr T=(V, E)$ be a directed tree. The {\it extended directed tree $\mathscr T_{\infty}$ associated with $\mathscr T$} is the directed graph $(\V, E_{\infty})$ given by
\beqn
\V = V \cupdot \{\infty\}, \quad E_{\infty} = E \cupdot \{(u, \infty) : u \in V\}.
\eeqn
\end{definition}
\begin{remark} \label{infty-child}
The element $\infty \in \V$ is descendant of every vertex in $V$. Indeed, $\infty \in \mathsf{Chi}(u)$ for every $u \in V.$ In view of the Friedman's notion of graph with boundary  (see \cite[Pg 490]{Fr}), its worth pointing out that $\infty$ is a boundary point when $\mathscr T_{\infty}$ is considered as the graph with boundary.
\end{remark}

\index{$\mathscr T_{\infty}=(V_{\infty}, E_{\infty})$}

A pictorial representation of an extended directed tree $\mathscr T_{\infty}=(V_{\infty}, E_{\infty})$ with the vertex set $V=\{\rootb, v_1, v_2, \ldots \}$ is given in Figure \ref{fig1}.
\begin{figure}
\begin{tikzpicture}[scale=.8, transform shape]
\tikzset{vertex/.style = {shape=circle,draw,minimum size=1em}}
\tikzset{edge/.style = {->,> = latex'}}
\node[vertex] (a) at  (0,0) {${}_\rootb$};
\node[vertex] (b) at  (2,-2) {$v_1$};
\node[vertex] (c) at  (2,2) {$v_2$};
\node[vertex] (d) at  (4, -2) {$v_3$};
\node[vertex] (e) at  (4, 2) {$v_4$};
\node[vertex] (f) at  (6, -2) {$\ldots$};
\node[vertex] (g) at  (6, 2) {$\ldots$};
\tikzset{vertex/.style = {shape=circle, draw, blue, minimum size=1em}}
\node[vertex] (a2) at (4,0) {$\infty$};
\draw[edge] (a) to (b);
\draw[edge] (a) to (c);
\draw[edge] (b) to (d);
\draw[edge] (c) to (e);
\draw[edge] (a) to (a2);
\draw[edge] (b) to (a2);
\draw[edge] (c) to (a2);
\draw[edge] (d) to (a2);
\draw[edge] (e) to (a2);
\draw[edge] (d) to (f);
\draw[edge] (e) to (g);
\draw[edge] (f) to (a2);
\draw[edge] (g) to (a2);
\end{tikzpicture}
\caption{An extended directed tree $\mathscr T_{\infty}$} \label{fig1}
\end{figure}

\index{$u \sqcup v$}

\begin{definition}[Join operation]
Let $\mathscr T=(V, E)$ denote a rooted directed tree with root $\rootb$
and let $\mathscr T_{\infty}=(\V, E_{\infty})$ be the extended
directed tree associated with $\mathscr T$.  Given $u, v \in \V,$ we
say that {\it $u$ joins $v$} if either $u \in \des{v}$ or $v \in
\des{u}$. Further, we set \beqn
u \sqcup v ~= ~\begin{cases} u & \mbox{if~} u \in \des{v}, \\
 v & \mbox{if~} v \in \des{u}, \\ \infty & \mbox{otherwise}.
\end{cases}
\eeqn
\end{definition}

By Remark \ref{infty-child},  $
\infty \sqcup u = \infty  = u \sqcup \infty$ for any $u \in \V.$
Further, the join operation $\sqcup$ satisfies the following:
\begin{align} \label{com-aso}
 \left.
  \begin{minipage}{71ex}
\begin{enumerate}
   \item[$\bullet$](Commutativity)
$u \sqcup v = v \sqcup u$ for all $u, v \in \V,$
\item[$\bullet$](Associativity)
$(u \sqcup v) \sqcup w = u \sqcup (v \sqcup w)$ for all $u, v, w \in \V,$
\item[$\bullet$](Neutral element) $u \sqcup \rootb= u = \rootb \sqcup u$ for all $u \in \V$,
\item[$\bullet$](Absorbing element) $u \sqcup \infty= \infty = \infty \sqcup u$ for all $u \in \V$.
\end{enumerate}
 \end{minipage}
   \right\}
\end{align}

We summarize \eqref{com-aso} in the following lemma.
\begin{lemma} \label{sg-1} Let $\mathscr T=(V, E)$ be a rooted directed tree with root $\rootb$ and let $\mathscr T_{\infty}=(\V, E_{\infty})$ be the extended directed tree associated with $\mathscr T$. Then the pair $(\V, \sqcup)$ is a commutative semigroup admitting $\rootb$ as neutral element and $\infty$ as an absorbing element.
\end{lemma}

\index{$\parset{u}{v}$}

Before we define the meet operation, let us introduce the following useful notation. For $u, v \in V$,
\beqn \label{common-p}
\parset{u}{v} &:=&\{w \in V : \parentn{n}{u} = w=\parentn{m}{v} ~\mbox{for some}~ m, n \in \mathbb N\}.
\eeqn

\index{$u \sqcap v$}

\begin{definition}[Meet operation]
Let $\mathscr T=(V, E)$ be a rooted directed tree with root $\rootb$ and let $\mathscr T_{\infty}=(\V, E_{\infty})$ be the extended directed tree associated with $\mathscr T$. Let $u, v \in V.$
We say that {\it $u$ meets $v$} if there exists a unique vertex $\omega \in V$ such that $$\displaystyle \sup_{w \in \parset{u}{v}}\dep_w\,=\,\dep_{\omega}.$$
In this case, we set
$u \sqcap v = \omega.$ In case $u \in \V$, we set
\beqn
\infty \sqcap u = u  = u \sqcap \infty.
\eeqn
\end{definition}
\begin{remark}
 Note that
\beq \label{join-rmk}
u \in \parset{v}{v} ~\Longrightarrow ~u \sqcap v= u=v \sqcap u.
\eeq
In fact, if $u=\parentn{l}{v}$ for some $l \in \mathbb N$, then
 \beqn
 \parset{u}{v} &=& \{w \in V : \parentn{n}{u} = w=\parentn{m}{v} ~\mbox{for some}~ m, n \in \mathbb N\} \\ &=& \{w \in V : \parentn{m}{u}=w ~\mbox{for some}~ m \in \mathbb N\}.
 \eeqn
 The conclusion in \eqref{join-rmk} is now immediate.
\end{remark}

\begin{lemma} \label{lem-join}
Let $\mathscr T=(V, E)$ be a rooted directed tree with root $\rootb$ and let $\mathscr T_{\infty}=(\V, E_{\infty})$ be the extended directed tree associated with $\mathscr T$.
Any two vertices $u, v \in V$ always meet. Further, they meet in a unique vertex $\omega$ belonging to $[\rootb, u] \cap [\rootb, v]$, so that $\max_{w \in \parset{u}{v}}\dep_w=\dep_{\omega}.$
\end{lemma}
\begin{proof}
Note that the set $\parset{u}{v} $ is non-empty. Indeed, $\rootb \in \parset{u}{v},$ since
\beqn
\parentn{\dep_u}{u}=\rootb=\parentn{\dep_v}{v}.\eeqn
Further, for any $w \in \parset{u}{v} $, there exist $m, n \in \mathbb N$ such that
\beqn
\dep_w = \dep_u-n= \dep_v - m \Le \min \{\dep_u, \dep_v\} < \infty.
\eeqn
This shows that \beq \label{finite-dep} \displaystyle \sup_{w \in \parset{u}{v} } \dep_w ~<~ \infty.\eeq
We claim that $\parset{u}{v} $ is finite. To see this, in view of \eqref{finite-dep}, it suffices to check that for any two distinct vertices $x, y \in \parset{u}{v} $, $\dep_x \neq \dep_y.$
Note that for some integers $n_1, n_2 \in \mathbb N,$
\beq \label{1}
\parentn{n_1}{u} = x, \quad \parentn{n_2}{u}=y.
\eeq
It follows that
\beq
\label{2}
\dep_x = \dep_u - n_1, ~ \dep_y = \dep_u - n_2 ~\Longrightarrow ~\dep_x=\dep_y + n_2 - n_1.
\eeq
Since $x \neq y,$ by \eqref{1}, $n_1 \neq n_2.$ Hence, by \eqref{2}, $\dep_x \neq \dep_y.$ This also shows that $\sup \parset{u}{v} $ is attained at a unique vertex in $\parset{u}{v}.$ The remaining part is also immediate from this.
\end{proof}

Let $\mathscr T_{\infty}=(\V, E_{\infty})$ be the extended directed tree associated with the rooted directed tree $\mathscr T$.
Then the meet operation $\sqcap$ satisfies the following:
\begin{align} \label{com-aso-2}
 \left.
  \begin{minipage}{75ex}
\begin{enumerate}
   \item[$\bullet$](Commutativity)
$u \sqcap v = v \sqcap u$ for all $u, v \in \V,$
\item[$\bullet$](Associativity)
$(u \sqcap v) \sqcap w = u \sqcap (v \sqcap w)$ for all $u, v, w \in \V,$
\item[$\bullet$](Neutral element) $u \sqcap \infty= u = \infty \sqcap u$ for all $u \in \V$,
\item[$\bullet$](Absorbing element) $u \sqcap \rootb= \rootb = \rootb \sqcap u$ for all $u \in \V.$
\end{enumerate}
 \end{minipage}
   \right\}
\end{align}

We summarize \eqref{com-aso-2} in the following lemma.
\begin{lemma} \label{sg-2} Let $\mathscr T=(V, E)$ be a rooted directed tree with root $\rootb$ and let $\mathscr T_{\infty}=(\V, E_{\infty})$ be the extended directed tree associated with $\mathscr T$. Then the pair $(\V, \sqcap)$ is a commutative semigroup admitting $\infty$ as neutral element and $\rootb$ as an absorbing element.
\end{lemma}

The operations meet and join can be unified in the following manner.

\index{$u \sqcup_{\mf b} v$}
\begin{definition}[Join operation at a base point]
Let $\mathscr T=(V, E)$ be a rooted directed tree with root $\rootb$ and let $\mathscr T_{\infty}=(\V, E_{\infty})$ be the extended directed tree associated with $\mathscr T$. Fix $\mf b  \in \V$ and let $u, v \in \V.$
Define the binary operation $\sqcup_{\mf b}$ on $\V$ by
\beqn
u \sqcup_{\mf b} v \,=\, \begin{cases} u \sqcap v & \mbox{if~} u, v \in \asc{\mf b}, \\
 u & \mbox{if~} v = \mf b, \\
 v & \mbox{if~} \mf b = u, \\
 u \sqcup v & \mbox{otherwise}.
\end{cases}
\eeqn
\end{definition}
\begin{remark}
Clearly, $\sqcup_{\rootb}= \sqcup.$ Further, by Remark \ref{infty-child},
$\sqcup_{\infty} = \sqcap.$
Thus we have a family of countably many operations, the first of which (corresponding to $\rootb$) is the join operation, while the farthest operation (corresponding to $\infty$) is the meet operation.
\end{remark}

\begin{proposition} \label{sg-3} Let $\mathscr T=(V, E)$ be a rooted directed tree with root $\rootb$ and let $\mathscr T_{\infty}=(\V, E_{\infty})$ be the extended directed tree associated with $\mathscr T$. Then, for every $\mf b \in V,$ the pair $(\V, \sqcup_{\mf b})$ is a commutative semigroup admitting $\mf b$ as a neutral element and $\infty$ as an absorbing element.
\end{proposition}
\begin{proof}
Let $\mf b \in V.$
The fact that  $(\V, \sqcup_{\mf b})$ is commutative and associative may be deduced from Lemmata \ref{sg-1} and \ref{sg-2}.
To complete the proof, note that $\mf b$ is a neutral element, while $\infty$ is an absorbing element for $(\V, \sqcup_{\mf b})$.
\end{proof}

\section{A canonical decomposition of an extended directed tree}

For a directed tree $\mathscr T=(V, E)$ and $u \in V,$ the set $V$ of vertices can be decomposed into three disjoint subsets:
\beq \label{cg-dec0}
V\,=\, \des{u} \cupdot \asc{u} \cupdot V_u,
\eeq
where $V_u$ is the complement of $\des{u} \sqcup \asc{u}$ in $V$ (see Figure \ref{fig0} for a pictorial representation of this decomposition with $u:=\mf v_4$).
Note that if $\mathscr T_{\infty}=(\V, E_{\infty})$ is the extended directed tree associated with $\mathscr T$, then $\V$ decomposes as follows:
\beq
\label{cg-dec}
\V\,=\, \des{u} \cupdot \asc{u} \cupdot V_u,
\eeq
where $\infty \in \des{u}$ by the very definition of extended directed tree.

\index{$V_u$}

\begin{caution}
Whenever we consider the decomposition in \eqref{cg-dec0}, it is understood that directed tree under consideration is $\mathscr T$ (and not the extended directed graph $\mathscr T_{\infty}$), so that $\infty \notin \des{u}$.
\end{caution}

It turns out that the cardinality of $V_u$ being infinite is intimately related to the large null summand property of weighted join operators (see Corollary \ref{lns}). We record the following general fact for ready reference.
\begin{proposition}
\label{lem-lns}
Let $\mathscr T=(V, E)$ be a directed tree and let $u \in V$. If $V_u$ is as given in \eqref{cg-dec0}, then the following statements hold:
\begin{enumerate}
\item If $\mathscr T$ is rooted, then $\mbox{card}(V_u)=\aleph_0$ if and only if $\mbox{card}(V \setminus \des{u})=\aleph_0.$
\item If $\mathscr T$ is leafless, then $\mbox{card}(V_u)=\aleph_0$ if and only if there exists a branching vertex $w \in \asc{u}.$
\item If $\mathscr T$ is leafless, then either $\mbox{card}(V_u)=0$ or $\mbox{card}(V_u)=\aleph_0$.
\end{enumerate}
\end{proposition}
\begin{proof} 
Note that (i) follows from \eqref{cg-dec0} and the fact that
$\mbox{card}(\asc{u}) < \infty$ for any $u \in V$ provided $\mathscr T$ is rooted. 
To see (ii), suppose $\mathscr T$ is leafless and assume that there exists $w
\in \asc{u}$ such that $\mbox{card}(\child{w}) \Ge 2.$ Thus there
exists a vertex $v \in \child{w}$ such that $v \notin \asc{u} \cup
\des{u}.$ Further, $\des{v}$ is contained in $V_u$. Since $\mathscr
T$ is leafless, $\mbox{card}(\des{v})=\aleph_0.$ This proves the
sufficiency part of (ii). On the other hand, if all vertices in
$\asc{u}$ are non-branching, then $V = \des{u} \cupdot \asc{u},$
which, by \eqref{cg-dec0}, implies that $V_u=\emptyset.$ This also
yields (iii).
\end{proof}

To see the role of canonical decompositions \eqref{cg-dec0} and \eqref{cg-dec} in determining join and meet of two vertices, let us see an example.
\begin{example}
Consider the rooted directed tree $\mathscr T=(V, E)$ as shown in Figure \ref{fig0}.
To get essential idea about the operations of meet and join, let us compute $v_4 \sqcup v$ and $v_4 \sqcap v$ for $v \in V$.  Note that
\beqn
v_4 \sqcup v \,=\, \begin{cases} v &  \mbox{if~} v \in \{\mf v_4, \mf v_7, \ldots, \}, \\
v_4 &  \mbox{if~} v \in \{v_2, v_0, \rootb\}, \\
\infty & \mbox{if~} v \in \{\bf v_1, \bf v_3, \ldots \} \cup \{\bf v_5, \bf v_8, \ldots\}. \end{cases}
\eeqn
Similarly, one can see that
\beqn
v_4 \sqcap v \,=\, \begin{cases} v_4 & \mbox{if~} v \in \{\mf v_4, \mf v_7, \ldots \}, \\
v & \mbox{if~} v \in \{v_2, v_0, \rootb\}, \\
v_0 & \mbox{if~} v \in \{\bf v_1, \bf v_3, \ldots \}, \\
v_2 & \mbox{if~} v \in  \{\bf v_5, \bf v_8, \ldots\}. \end{cases}
\eeqn
Note that two vertices in $V$ can join at the vertex $\infty,$ while two vertices in $V$ always meet in $V.$
\eop
\end{example}

\begin{figure}
\begin{tikzpicture}[scale=.8, transform shape]
\tikzset{vertex/.style = {shape=circle,draw,blue, minimum size=1em}}
\tikzset{edge/.style = {->,> = latex'}}
\node[vertex] (a) at  (-2,0) {${}_\rootb$};

\node[vertex] (a0) at  (0,0) {$v_0$};
\node[vertex] (c) at  (2,2) {$ v_2$};

\tikzset{vertex/.style = {shape=circle,draw,magenta, minimum size=1em}}
\node[vertex] (d1) at  (4, 1) {$\mf v_4$};
\node[vertex] (h) at  (6, 1) {$\mf v_7$};
\node[vertex] (k) at  (8, 1) {$\ldots$};
\node[vertex] (h1) at  (8, -.6) {$\ldots$};

\tikzset{vertex/.style = {shape=circle,draw,black, minimum size=1em}}

\node[vertex] (b) at  (2,-2) {$\bf v_1$};

\node[vertex] (d) at  (4, -2) {$\bf v_3$};
\node[vertex] (e) at  (4, 3) {$\bf v_5$};
\node[vertex] (f) at  (6, -2) {$\bf v_6$};
\node[vertex] (g) at  (6, 3) {$\bf v_8$};
\node[vertex] (j) at  (8, 3) {$\bf \ldots$};

\node[vertex] (i) at  (8, -2) {$\bf \ldots$};

\draw[edge] (a) to (a0);
\draw[edge] (a0) to (b);
\draw[edge] (a0) to (c);
\draw[edge] (b) to (d);
\draw[edge] (c) to (e);
\draw[edge] (c) to (d1);
\draw[edge] (d) to (f);
\draw[edge] (e) to (g);
\draw[edge] (d1) to (h);
\draw[edge] (f) to (i);
\draw[edge] (g) to (j);
\draw[edge] (h) to (k);
\draw[edge] (h) to (h1);
\end{tikzpicture}
\caption{A rooted directed tree $\mathscr T=(V, E)$, where $V$ is disjoint union of $\des{\mf v_4}=\{\mf v_4, \mf v_7, \ldots \}$, $\asc{\mf v_4}=\{v_2, v_0, \rootb\}$, and $V_{\mf v_4}=\{\bf v_1, \bf v_3, \ldots \} \cup \{\bf v_5, \bf v_8, \ldots\}$} \label{fig0}
\end{figure}


Let $\mathscr T=(V, E)$ be a rooted directed tree with root $\rootb$ and let $\mathscr T_{\infty}=(\V, E_{\infty})$ be the extended directed tree associated with $\mathscr T$. Fix $u  \in V$. Then, for any $\mf b \in \V \setminus \{u\}$ and $v \in V,$
the binary operation $\sqcup_{\mf b}$ on $\V$ satisfies
the following:
\beqn
u \sqcup_{\mf b} v \,=\, \begin{cases} u  & \mbox{if~} \mf b \in \asc{u} \cup  V_u, ~v \in \asc{u}, \\
 v & \mbox{if~} \mf b \in \des{u}, ~v \in \asc{u}, \\
 v & \mbox{if~}  \mf b \in \asc{u}, ~v \in \des{u}, \\
 u & \mbox{if}~ \mf b \in \des{u}, ~v \in [u, \mf b], \\
 v & \mbox{if~} \mf b \in\des{u}\setminus \{\infty\}, ~v \in \desb{u},\\
 u & \mbox{if~} \mf b =\infty, ~v \in \desb{u},\\
 v & \mbox{if~} \mf b \in V_u, ~v \in \des{u},
\end{cases}
\eeqn where $[u, v]$ denotes the directed path from $u$ to $v$ in a
directed tree $\mathscr T$. The above discussion is summarized in the following
table:
\begin{table}[H]
\begin{tabular}{|c|c|c|c|c|c|}
\hline $\frac{\mf b \rightarrow}{v \downarrow}$  & $\asc{u}$ &
$\des{u} \setminus \{u, \infty\}$ & $u$ & $V_u$ & $\{\infty\}$ \\
\hline $\asc{u}$ & $u$ & $v$ & $v$ & $u$ & $v$ \\ \hline $\desb{u}$
& $v$ & $v$ & $v$ & $v$ & $u$ \\ \hline $[u, \mf b]$ & $-$ & $u$ &
$u$ & $-$ & $u$ \\ \hline $V_u \setminus \{\mf b\}$ & $\infty$ &
$\infty$ & $v$ & $\infty$ & $u \sqcap v$ \\ \hline $\{\mf b\}$ & $u$
& $u$ & $\mf b$ & $u$ & $u$ \\ \hline
\end{tabular}
\vskip.2cm
\caption{\label{Table1} Join operation $u \sqcup_{\mf b} v$ at the base point $\mf b$}
\end{table}

We conclude this section with a useful result describing the set of vertices,
which join to a given vertex (with respect to a base point) at another given vertex.

\index{$M^{(\mf b)}_u(w)$}

\begin{proposition} \label{lem3.2} Let $\mathscr T=(V, E)$ be a rooted directed tree with root $\rootb$ and let $\mathscr T_{\infty}=(\V, E_{\infty})$ be the extended directed tree associated with $\mathscr T$.
For $u \in V$ and $w \in \V,$ define \beq \label{M-u-b} M^{(\mf
b)}_u(w)\,:=\,\{v \in V : u \w v=w\}. \eeq Then the following statements
hold:
\begin{enumerate}
\item If $\mf b = \infty,$ then
\beqn
M^{(\mf b)}_u(w) \,=\, \begin{cases}
\des{\parentn{j}{u}} \setminus \des{\parentn{j-1}{u}}  & \mbox{if~} w =\parentn{j}{u}, ~j=1, \ldots, \dep_u,
\\
\mathsf{Des}(u)   & \mbox{if~}   w=u, \\
\emptyset & \mbox{otherwise}.
\end{cases}
\eeqn
\item If $\mf b \in V$ and $u \in \asc{\mf b}$, then
\beqn
M^{(\mf b)}_u(w) \,=\, \begin{cases}
\{w\} & \mbox{if}~w \in   \asc{u} \sqcup \desb{u}, \\
  [u, \mf b] & \mbox{if}~w =u,
  \\
 V_u & \mbox{if~}w = \infty, \\
\emptyset & \mbox{otherwise}.
\end{cases}
\eeqn
\item
If $\mf b \in V$ and $u \notin \asc{\mf b}$, then

\beqn
M^{(\mf b)}_u(w) \,=\, \begin{cases}
\{w\} & \mbox{if}~\mf b = u ~\mbox{or~} w \in \mathsf{Des}_u(u), \\
 \asc{u} \cup \{u, \mf b\} & \mbox{if}~u \neq \mf b~\mbox{and~}w =u, \\
 V_u \setminus \{\mf b\} & \mbox{if~}w = \infty, \\
\emptyset & \mbox{otherwise}.
\end{cases}
\eeqn
\item If $\mf b \in V \setminus \{u\},$ then $$\V \setminus M^{(\mf b)}_u(\infty) \,=\, \asc{u} \cup  \des{u} \cup \{\mf b\}.$$
\end{enumerate}
\end{proposition}
\begin{proof}
By the definition of join operation $\sqcup_{\mf b}$,  $M^{(\mf b)}_u(w)$ equals
\beqn
\begin{cases}
\{v \in \asc{\mf b}  : u \sqcap v=w\} \cup \{v \in V \setminus \asc{\mf b}  : u \sqcup v=w\}  & \mbox{if~} u \in \asc{\mf b}, w \neq u, \\
\{v \in \asc{\mf b}  :  v \in \mathsf{Des}(w)\} \cup (\{\mf b\} \cap V)  & \mbox{if~} u \in \asc{\mf b}, w = u,  \\
 \{w\} & \mbox{if~} \mf b = u, \\
 \{v \in V : u \sqcup v=w\} & \mbox{if~}u \notin \asc{\mf b}, ~u \neq \mf b.
\end{cases}
\eeqn
The desired conclusions in (i)-(iii) can be easily deduced from the facts that
\beqn \asc{\infty}=V, ~u \sqcap v \in \parset{u}{v}, ~ u \sqcup v \in \des{u} \cap \des{v},\quad u, v \in V. \eeqn
The parts (i)-(iii) may also be deduced from Table \ref{Table1}.
To see (iv), let  $\mf b \in V \setminus \{u\}$.
As seen above,
\beqn
M^{(\mf b)}_u(\infty) \,=\, \begin{cases} \{v \in \asc{\mf b} : u \sqcap v=\infty\} \cup V_u & \mbox{if~} u\in \asc{\mf b}, \\
\{v \in V \setminus \{\mf b\} : u \sqcup v =\infty\} & \mbox{otherwise},
\end{cases}
\eeqn
where $V_u$ is as given in \eqref{cg-dec0}.
Thus, in any case, $M^{(\mf b)}_u(\infty)=V_u \setminus \{\mf b\},$ and hence $$\asc{u} \cup  \des{u} \cup \{\mf b\} ~\subseteq ~\V \setminus M^{(\mf b)}_u(\infty).$$ To see the reverse inclusion,
let $v \in \V \setminus M^{(\mf b)}_u(\infty).$
Since $\infty \in \des{u},$ we may assume that $v \neq \infty.$
Then $u \sqcup_{\mf b} v \in V,$ and hence we may further assume that $v \neq \mf b.$
If $u, v \in \asc{\mf b}$, then $v \in  \asc{u} \cup  \des{u}.$
Otherwise, $u \sqcup v = u \sqcup_{\mf b} v \in V,$ and hence
$u \in \des{v}$ or $v \in \des{u}$. In this case also, $v \in  \asc{u} \cup  \des{u}.$ This yields the desired equality in (iv).
\end{proof}

The last result turns out to be crucial in decomposing the so-called weighted join operator as a direct sum of a diagonal operator and a finite rank operator.

\chapter{Weighted join operators on rooted directed trees}

Let $\mathscr T=(V, E)$ be a rooted directed tree with root $\rootb$ and let $\mathscr T_{\infty}=(\V, E_{\infty})$ be the extended directed tree associated with $\mathscr T$.
In what follows, $\ell^2(V)$ stands for the Hilbert
space of square summable complex functions on $V$
equipped with the standard inner product. Note that
the set $\{e_u\}_{u\in V}$ is an
orthonormal basis of $\ell^2(V)$, where $e_u : V \rar \mathbb C$
is the indicator function of $\{u\}$. The convention $e_{\infty}=0$ will be used throughout this text.
Note that $\ell^2(V)$ is a reproducing kernel Hilbert space. Indeed,  for every $v \in V,$
the evaluation map $f \mapsto f(v)$ is a bounded linear functional from $\ell^2(V)$ into $\mathbb C$.
For a nonempty subset $W$ of $V$, $\ell^2(W)$ may be considered as a subspace of $\ell^2(V)$. Indeed, if one extends $f : W \rar \mathbb C$ by setting $F:=f$ on $W$ and $0$ on $V \setminus W$, then the mapping $U : \ell^2(W) \rar \ell^2(V)$ given by $Uf = F$ is an isometric homomorphism.
Sometimes, the orthogonal projection $P_{\ell^2(W)}$ of $\ell^2(V)$ onto $\ell^2(W)$ will be denoted  by $P_{_W}$. We say that a closed subspace $\mathcal M$ of $\ell^2(V)$ is {\it supported on a subset $W$ of $V$} if $\mathcal M= \bigvee\{ e_v : v \in W\}$. In this case, we refer to $W$ as the {\it support of $\mathcal M$} and
we write $\supp\,\mathcal M:=W.$

\index{$P_{\ell^2(W)}=P_W$}
\index{$\supp\, \mathcal M$}
\index{$\supp\, \ell^2(W)=W$}
\index{$\D{u}$}
\index{$\W{b}{u}$}
\index{$\varLambda^{(\mf b)}_{u}$}
\index{$\mathscr D_V$}
\index{$\lambdab_u =\{\lambda_{uv}\}_{v\in \V}$}
\begin{definition}
Let $\mathscr T=(V, E)$ be a rooted directed tree with root $\rootb$
and let $\mathscr T_{\infty}=(\V, E_{\infty})$ be the extended
directed tree associated with $\mathscr T$. For $u \in V$ and  $\mf
b \in \V$, by the
the {\it weight system} $\lambdab_u =
\{\lambda_{uv}\}_{v\in \V},$ we understand the subset 
$\{\lambda_{uv}\}_{v\in \V}$ of complex numbers such that
$\lambda_{u\infty}=0.$
\begin{enumerate}
\item The {\it diagonal operator
$\D{u}$ on $\mathscr T$} is given by
\beqn
    \label{diagonal-u}
\mathscr D(\D{u}) &:=&  \Big\{f \in \ell^2(V) : \sum_{v \in V} |f(v)|^2 |\lambda_{uv}|^2 < \infty\Big\} \\
\D{u}f &:= &\sum_{v \in V} f(v) \lambda_{uv}\, e_{v}, \quad f \in \mathscr D(\D{u}).
\eeqn
\item The {\em weighted join
operator $\W{b}{u}$ on ${\mathscr T}$} is given by
\beqn
{\mathscr D}(\W{b}{u}) & := & \big\{f \in \ell^2(V) \colon
\varLambda^{(\mf b)}_{u} f \in \ell^2(V)\big\},
   \\
\W{b}{u} f & := & \varLambda^{(\mf b)}_{u} f, \quad f \in {\mathscr
D}(\W{b}{u}),
\eeqn
where $\varLambda^{(\mf b)}_{u}$ is the mapping defined on
complex functions $f$ on $V$ by
   \beq \label{Lambda-u-b}
(\varLambda^{(\mf b)}_{u} f) (w)~ := \displaystyle \sum_{v \in
M^{(\mf b)}_u(w)} \!\!\!\!\!\!\Lmd{u}{v}\,  f(v), \quad w \in V
\eeq
with the set $M^{(\mf b)}_u(w)$ given by \eqref{M-u-b}.
The operator $\W{\mf \infty}{u}$ is referred to as the {\it weighted
meet operator}.
\end{enumerate}
\end{definition}
\begin{remark} \label{rmk-action}
Several remarks are in order.
\begin{enumerate}
\item It is well-known that $\D{u}$
is a densely defined closed linear operator in $\ell^2(V)$. Its adjoint $D^*_{\lambdab_u}$ is the diagonal operator with diagonal entries $\{\overline{\lambda}_{uv}\}_{v\in V}$. Furthermore, $\D{u}$ is normal and $\mathscr D_V$ is a core for $\D{u},$ where
\beq
\label{d-supp-hbu}
\mathscr D_{U}\,=\,\mbox{span}\,\{e_v : v \in U\}, \quad U \subseteq V,
\eeq
(see \cite[Lemma 2.2.1]{Jablonski}).
\item Note that ${\mathscr D}(\W{b}{u})$ forms a subspace of $\ell^2(V)$. Indeed, if $f, g \in {\mathscr D}(\W{b}{u})$, then for every $w \in V$, the series $(\varLambda^{(\mf b)}_{u} f) (w)$ and $(\varLambda^{(\mf b)}_{u} g) (w)$ converge, and hence so does $(\varLambda^{(\mf b)}_{u} (f+g)) (w).$ In particular, by Proposition \ref{lem3.2}, these series are finite sums in case $\mf b \neq \infty.$ Also, $\varLambda^{(\mf b)}_{u} (f+g) \in \ell^2(V)$ if $\varLambda^{(\mf b)}_{u} (f) \in \ell^2(V)$ and $\varLambda^{(\mf b)}_{u} (g) \in \ell^2(V)$. Indeed, $$\|\varLambda^{(\mf b)}_{u} (f+g)\|^2 \,\leqslant\, 2\big(\|\varLambda^{(\mf b)}_{u} (f)\|^2 + \|\varLambda^{(\mf b)}_{u} (g)\|^2\big).$$
\item For every $v \in V,$ $e_v \in {\mathscr D}(\W{b}{u})$ and \beq \label{W-action}
(\W{b}{u} e_v)(w) ~= \displaystyle \sum_{\eta \in M^{(\mf
b)}_u(w)}\!\!\!\!\!\!\lambda_{u\eta}\, e_v(\eta) = \lambda_{uv}\,
e_{u \w v}(w), \quad w \in V. \eeq In particular, \beq
\label{inv-dom} \mathscr D_{V}:=\mbox{span}\,\{e_v : v \in V\}
\subseteq {\mathscr D}(\W{b}{u}), \quad \W{b}{u} \mathscr D_{V}
\subseteq \mathscr D_{V}. \eeq Thus all positive integral powers of
$\W{b}{u}$ are densely defined and the Hilbert space adjoint
${\W{b}{u}}^*$ of $\W{b}{u}$ is well-defined. 
To
see the action of ${\W{b}{u}}^*,$ let $f \in \ell^2(V)$ and $w \in V$. Note that
\beqn \inp{\W{b}{u}f}{e_w} =  \sum_{v \in V} f(v) \lambda_{uv}\,
e_{u \w v}(w)  =\sum_{v \in M^{(\mf b)}_u(w)}\!\!\!\!\!\!f(v)
\lambda_{uv}=\inp{f}{g_w}, \eeqn where $M^{(\mf b)}_u(w)$ is as
given in \eqref{M-u-b} and $g_w =  \sum_{{v \in M^{(\mf b)}_u(w)}}
\bar{\lambda}_{uv}e_v$. Thus we obtain \beqn \label{adjoint}
\Big(\W{b}{u}\Big)^*g \,=\, \sum_{w \in V}  \sum_{v \in M^{(\mf
b)}_u(w)}\!\!\!\!\!\!g(w) \bar{\lambda}_{uv}e_v, \quad g \in
\ell^2(V). \eeqn
\item Finally, note that  if $\mf b =u$, then $\W{b}{u}$ is the diagonal operator $\D{u}.$
\end{enumerate}
\end{remark}


\index{$\mathscr T_1$}

Let us see two simple yet instructive examples of rooted directed trees in which the associated weighted join operators take a concrete form. Both these examples of rooted directed trees have been discussed in \cite[Eqn (6.2.10)]{Jablonski} in the context of weighted shifts on directed trees.
\begin{example}[With no branching vertex] 
Consider the directed tree $\mathscr T_1$ with the set of vertices
$V=\mathbb{N}$ and $\mathsf{root}=0$. We further require that
$\mathsf{Chi}(n)=\{n+1\}$ for all $n \in \mathbb N$. Let $m \in V$
and $n, \mf b \in \V$. By \eqref{W-action}, the weighted join
operator $\W{b}{m}$ on ${\mathscr T}$ is given by \beqn
\W{b}{m} e_n \,=\,   \begin{cases}  \lambda_{mn}\, e_{\min\{m, n\}} & \mbox{if~} m < \mf b~\mbox{and~}n < \mf b, \\
  \lambda_{mn}\, e_m & \mbox{if~} n = \mf b, \\
  \lambda_{mn}\, e_n & \mbox{if~} m = \mf b~\mbox{and~}n \in \mathbb N, \\
  \lambda_{mn}\, e_{\max\{m, n\}} & \mbox{otherwise},
\end{cases}
\eeqn
where we used the assumption that $\lambda_{m \infty}=0$ and the convention that $\max\{m, \infty\}=\infty.$
In particular,
\beqn
\W{b}{m} e_n \,=\,   \begin{cases}
  \lambda_{mn}\, e_{\max\{m, n\}} & \mbox{if~} \mf b= 0, \\
  \lambda_{mn}\, e_{\min\{m, n\}} & \mbox{if~} \mf b = \infty.
\end{cases}
\eeqn
The matrix representations of $\W{\mf 0}{m}$ and $\W{\mf
\infty}{m}$ with respect to the ordered orthonormal basis
$\{e_n\}_{n \in \mathbb N}$ are given by
\beqn \W{\mf 0}{m} &=& \begin{psmallmatrix}
0 & \cdots &  &   &   &   &  \\
\vdots &  &  &   &   &   &  \\
0 &  \cdots &  &   &   &   &  \\
\lambda_{m0} & \cdots  & \lambda_{mm} & 0 & \cdots & &\\
0 & \cdots & 0 & \lambda_{mm+1} &   0 & \cdots   &\\
 \vdots & & \vdots &  0 & \lambda_{mm+2}  &   & \\
 & & &  \vdots & 0  &  \lambda_{mm+3} & \\
 & & &  &   \vdots &   & \ddots
 \end{psmallmatrix} \\ \\ \\
 \W{\mf \infty}{m} &=& \begin{psmallmatrix}
\lambda_{m0} & 0  & \cdots &  &  & &\\
0 & \lambda_{m1} & 0 & \cdots &    &    &\\
 \vdots & 0 & \ddots &  0 & \cdots  &   & \\
 & \vdots & & \lambda_{mm-1}  & 0  &  \cdots & \\
 & & &  0 &   \lambda_{mm} & \lambda_{mm+1}  & \cdots \\
 & & &  \vdots &   0 & 0  & \cdots \\
 & & &   &   \vdots & \vdots  &
\end{psmallmatrix}.\eeqn
Thus $\W{\mf 0}{m}$ is a at most rank one perturbation of a diagonal operator, while $\W{\mf \infty}{m}$ is a finite rank operator with range contained in the linear span of $\{e_k : k=0, \ldots, m\}.$
\eop
\end{example}

\index{$\mathscr T_2$}

\begin{example}[With one branching vertex] \label{ex-T2}
Consider the directed tree $\mathscr T_2$ with set of vertices
$$V\,=\,\{0\}\cup\{2j-1, 2j : j \Ge 1\}$$ and
$\mathsf{root}=0$. We further require that
$\mathsf{Chi}(0)=\{1,2\},$ $\mathsf{Chi}(2j-1)=\{2j+1\}$ and
$\mathsf{Chi}(2j)=\{2j+2\},$ $j \Ge 1.$ Let $m \in V$ and $n \in
\V$. By \eqref{W-action}, the weighted join operator
$\W{b}{m}$ on ${\mathscr T}$ is given by \beqn
\W{b}{m} e_n \,=\, \begin{cases} \lambda_{mn} e_{m \sqcap n} & \mbox{if~} m, n \in \asc{\mf b}, \\
 \lambda_{mn} e_m & \mbox{if~} n = \mf b, \\
\lambda_{mn}  e_n & \mbox{if~} m = \mf b~\mbox{and~}n \in \mathbb N, \\
 \lambda_{mn} e_{m \sqcup n} & \mbox{otherwise}.
\end{cases}
\eeqn
In particular, if $m$ and $n$ are positive integers, then
\beqn
\W{\mf 0}{m} e_n \,=\,   \begin{cases}
  \lambda_{mn}\, e_{\max\{m, n\}} & \mbox{if~} m, n~\mbox{are odd or} ~m, n~\mbox{are even},\\
0 & \mbox{otherwise},
\end{cases}
\eeqn
\beqn
\W{\mf \infty}{m} e_n \,=\,   \begin{cases}
  \lambda_{mn}\, e_{\min\{m, n\}} & \mbox{if~} m, n~\mbox{are odd or} ~m, n~\mbox{are even},\\
\lambda_{mn}\,e_0 & \mbox{otherwise}.
\end{cases}
\eeqn The matrix representations of $\W{\mf 0}{m}$ and $\W{\mf
\infty}{m}$ with respect to the ordered orthonormal basis
$\{e_{2n}\}_{n \in \mathbb N} \cup \{e_{2n+1}\}_{n \in \mathbb N}$
are given by \beqn
   \W{\mf 0}{m} &=& \begin{psmallmatrix}
0 & \cdots &  &   &   &   &  &\\
\vdots &  &  &   &   &   &  &\\
0 &  \cdots &  &   &   &   & & \\
\lambda_{m0} &  \lambda_{m2} & \cdots & \lambda_{mm} & 0 & \cdots & &\\
0& 0 & \cdots & 0 & \lambda_{mm+2} &   0 & \cdots   & \\
 \vdots &\vdots  & &  \vdots & 0 &  \lambda_{mm+4}   & 0 &\\
 & & & &  \vdots & 0  & \lambda_{mm+6}  & \\
 & & & &  &   \vdots &   & \ddots \end{psmallmatrix}
 \oplus {\bf 0},
\\ \\ \\
  \W{\mf \infty}{m} &=& \begin{psmallmatrix}
\lambda_{m0} & 0  & \cdots &  &  & & \lambda_{m1} & \lambda_{m3} & \cdots \\
0 & \lambda_{m2} & 0 & \cdots &    &    &\\
 \vdots & 0 & \ddots &  0 & \cdots  &   & \\
 & \vdots & & \lambda_{mm-2}  & 0  &  \cdots & \\
 & & &  0 &   \lambda_{mm} & \lambda_{mm+2}  & \cdots \\
 & & &  \vdots &   0 & 0  & \cdots \\
 & & &   &   \vdots & \vdots  &  \end{psmallmatrix} \eeqn
 (see Figure \ref{fig3}). It turns out that $\W{\mf
0}{m}$ is a at most a rank one perturbation of a diagonal operator,
while $\W{\mf \infty}{m}$ is a finite rank operator with range
contained in the linear span of $\{e_m, e_{m-2}, \ldots, e_0\}$.
\eop
\end{example}

\begin{figure}
\begin{tikzpicture}[scale=.8, transform shape]
\tikzset{vertex/.style = {shape=circle,draw, minimum size=1em}}
\tikzset{edge/.style = {->,> = latex'}}
\node[vertex] (a) at  (0.5,0) {$0$};
\node[vertex] (b) at  (2,-1) {$1$};
\node[vertex] (c) at  (2,1) {$2$};
\node[vertex] (d) at  (4, -1) {$3$};
\node[vertex] (e) at  (4, 1) {$4$};
\node[vertex] (f) at  (6, -1) {$\ldots$};
\node[vertex] (g) at  (6, 1) {$\ldots$};
\node[vertex] (h) at  (8, -1) {$\ldots$};
\node[vertex] (j) at  (10, -1) {$\ldots$};
\node[vertex] (k) at  (10, 1) {$\ldots$};

\tikzset{vertex/.style = {shape=circle,draw, blue, minimum size=1em}}
\node[vertex] (i) at  (8, 1) {$m$};

\draw[edge] (a) to (b);
\draw[edge] (a) to (c);
\draw[edge] (b) to (d);
\draw[edge] (c) to (e);
\draw[edge] (d) to (f);
\draw[edge] (e) to (g);
\draw[edge] (f) to (h);
\draw[edge] (g) to (i);
\draw[edge] (h) to (j);
\draw[edge] (i) to (k);
\end{tikzpicture}
\caption{A pictorial representation of $\mathscr T_2$ with prescribed vertex $m$} \label{fig3}
\end{figure}

The fact, as illustrated in the preceding examples, that weighted join operator is either diagonal,  a rank one perturbation of a diagonal operator or a finite rank operator holds in general (see Theorem \ref{dichotomy}).

\section{Closedness and boundedness}

In this section, we discuss closedness and boundedness  of weighted join operators on rooted directed trees.
Unless stated otherwise, $\mf b \in \V$ denotes the base point of the weighted join operator $\W{b}{u}$.
\begin{proposition} \label{closed}
Let $\mathscr T=(V, E)$ be a rooted directed tree with root $\rootb$ and let $\mathscr T_{\infty}=(\V, E_{\infty})$ be the extended directed tree associated with $\mathscr T$.
Let $\mf b, u \in V$ and let
$\lambdab_u = \{\lambda_{uv}\}_{v\in \V}$ be a weight system of complex numbers.
Then the weighted join operator $\W{b}{u}$ on ${\mathscr T}$
defines a densely defined closed linear operator. Moreover, $\mathscr D_V:= \mbox{span}\{e_v : v \in V\}$ forms a core for $\W{b}{u}.$
\end{proposition}
\begin{proof}
We have already noted that $\W{b}{u}$ is densely defined (see Remark
\ref{rmk-action}). Let $\{f_n\}_{n \in \mathbb N}$ be a sequence
converging to $f$ in $\ell^2(V)$. Suppose that $\{\W{b}{u}f_n\}_{n
\in \mathbb N}$ converges to some $g \in \ell^2(V).$ Since
$\ell^2(V)$ is a reproducing kernel Hilbert space, for every $w \in V,$
\beqn \lim_{n
\rar \infty} f_n(w) = f(w), \quad \lim_{n \rar \infty} \!\!\!\!\sum_{v \in
M^{(\mf b)}_u(w)}\!\!\!\!\!\!\Lmd{u}{v}\,  f_n(v)  = g(w), \eeqn where $M^{(\mf b)}_u(w)$ is given by \eqref{M-u-b}.
However, since $\mf b \neq \infty,$ $\mbox{card}(M^{(\mf b)}_u(w)) <
\infty$ for each $w \in V$ (see (ii) and (iii) of Proposition \ref{lem3.2}). It
follows that $f \in \mathscr D(\W{b}{u})$ and $\W{b}{u}f=g.$ Thus
$\W{b}{u}$ is a closed linear operator.

To see that $\mathscr D_V$ is a core for $\W{b}{u}$, note that by the preceding discussion,
${\W{b}{u}}|_{\mathscr D_V}$ is a closable operator such that
$\overline{{\W{b}{u}}|_{\mathscr D_V}} \subseteq \W{b}{u}$. To see
the reverse inclusion, let $f =\sum_{v \in V} f(v) e_v \in \mathscr
D(\W{b}{u})$ and let \beqn f_n\,:=\,\sum_{\substack{v \in V \\ \dep_v
\Le n}} f(v) e_v, \quad n \in \mathbb N. \eeqn Then $\{f_n\}_{n \in
\mathbb N} \subseteq \mathscr D_V$, $\|f_n-f\|_{\ell^2(V) } \rar 0$
as $n \rar \infty$ and \beqn \|\W{b}{u}f_n-\W{b}{u}f\|^2_{\ell^2(V)}
&=& \Big\|\sum_{\substack{v \in V \\ \dep_v > n}} f(v)
\lambda_{uv}\, e_{u \sqcup_{\mf b} v}\Big\|^2_{\ell^2(V)} \\ &=&
\sum_{\substack{w \in V \\ \dep_w > n}} \Big|\sum_{\substack{v \in V
\\ u \sqcup_{\mf b} v=w}}\!\!\!\!f(v) \lambda_{uv} \Big|^2,
\eeqn which converges to $0$ as $n \rar \infty,$ since $\W{b}{u}f
\in \ell^2(V).$ It follows that $f \in \mathscr D(
\overline{{\W{b}{u}}|_{\mathscr D_V}})$ and
$\W{b}{u}f=\overline{{\W{b}{u}}|_{\mathscr D_V}}f.$
This yields
$\overline{{\W{b}{u}}|_{\mathscr D_V}} = \W{b}{u}$.
\end{proof}
\begin{remark}
This result no more holds true for the weighted meet operator $\W{\infty}{u}$.
Indeed, it may be concluded from Theorem \ref{bdd} below and Lemma \ref{lem-inj-ten-unb} that $\W{\infty}{u}$ may not be even closable.
\end{remark}

 We discuss next the boundedness of weighted join operators $\W{b}{u}$.
\begin{theorem}
\label{bdd}
Let $\mathscr T=(V, E)$ be a rooted directed tree with root $\rootb$ and let $\mathscr T_{\infty}=(\V, E_{\infty})$ be the extended directed tree associated with $\mathscr T$.
Let $u \in V,$ $\mf b \in \V$ and let $\lambdab_u = \{\lambda_{uv}\}_{v\in \V}$ be a weight system
 of complex numbers.
Then the weighted join operator $\W{b}{u}$ on ${\mathscr T}$
is bounded if and only if
\beq \label{star}
 \lambdab_u   ~\mbox{belongs to~}  \begin{cases}  \ell^2(V)  & \mbox{if~}  \mf b = \infty, \\
 \ell^{\infty}(V) & \mbox{if} ~\mf b = u, \\
  \ell^{\infty}(\mathsf{Des}(u)) & \mbox{otherwise}.
\end{cases}
\eeq
\end{theorem}
\begin{proof}
Let $f = \sum_{v \in V} f(v) e_v \in  \ell^2(V)$ be of norm $1$ and let $\varLambda^{(\mf b)}_{u}$ be as defined in \eqref{Lambda-u-b}.
Recall that $f \in \mathscr D(W^{(\mf b)}_{u})$ if and only if $\varLambda^{(\mf b)}_{u}f \in \ell^2(V)$.
By \eqref{W-action}, \beqn \label{action-W}
\varLambda^{(\mf b)}_{u}f \,=\,\sum_{v \in V} f(v) \lambda_{uv}\, e_{u \w v}.
\eeqn
It follows that
 $f \in \mathscr D(W^{(\mf b)}_{u})$ if and only if
\beq
\label{condition}
 \|\varLambda^{(\mf b)}_{u}f\|^2 \,=\, \sum_{w \in V}  \Big|\sum_{v \in M^{(\mf b)}_u(w)}  \lambda_{uv} \, f(v) \Big|^2
\eeq
is finite,
where $M^{(\mf b)}_u(w)$ is as given in \eqref{M-u-b}.
We divide the proof into the following three cases:
\begin{case}
$\mf b = \infty:$
\end{case}
\noindent
Let $u_j:=\parentn{j}{u}$,  $j =0, \ldots, \dep_u$.
By  Proposition \ref{lem3.2}(i) and \eqref{condition}, we obtain \beq \label{condition1}
 \|\varLambda^{(\infty)}_{u}f\|^2 &=& \notag \sum_{{w \in V}}  \Big|\sum_{\substack{v \in M^{(\infty)}_u(w)}}  \lambda_{uv} \, f(v) \Big|^2 \\
&=&  \sum_{j=0}^{\dep_u} \Big|\sum_{v \in \des{u_j} \setminus \des{u_{j-1}}}  \lambda_{uv} \, f(v) \Big|^2,
\eeq
where we used the convention that $\des{\parentn{-1}{u}}=\emptyset.$
We claim that $\W{\mf \infty}{u}$ belongs to  $B(\ell^2(V))$ if and only if
\beq \label{lambda-du}
\lambdab_u   ~\in ~ \ell^2(\des{u_j} \setminus \des{u_{j-1}}), \quad j=0, \ldots, \dep_u.
\eeq
If \eqref{lambda-du} holds, then by \eqref{condition1} and the Cauchy-Schwarz inequality,
\beqn
 \|\varLambda^{(\infty)}_{u}f\|^2 ~\Le ~\|f\|^2\, \sum_{j=0}^{\dep_u} \sum_{v \in \des{u_j} \setminus \des{u_{j-1}}}  |\lambda_{uv}|^2,
\eeqn
which shows that $\W{\mf \infty}{u} \in B(\ell^2(V))$. Conversely, if $\W{\mf \infty}{u} \in B(\ell^2(V))$, then
by \eqref{condition1},
\beqn
\sup_{\|f\|=1} \Big|\sum_{v \in \des{u_j} \setminus \des{u_{j-1}}}  \lambda_{uv} \, f(v) \Big| ~< ~\infty, \quad j=0, \ldots, \dep_u.
\eeqn
By the standard polar representation, the series above is indeed absolutely convergent, and hence
by Riesz representation theorem \cite{Co},  $\lambdab_u   \in  \ell^2(\des{u_j} \setminus \des{u_{j-1}})$, $j=0, \ldots, \dep_u.$ Thus the claim stands verified.
To complete the proof, it now suffices to check that
\beq \label{contend}
\V \,=\,  \bigcupdot_{j=0}^{\dep_u} \Big(\des{u_j} \setminus \des{u_{j-1}}\Big)
\eeq
(see Figure \ref{fig-decom}).
To see that, let $v \in \V.$ Clearly, $\des{u} =\des{u_0} \setminus \des{u_{-1}}$. Thus we may assume that $v \in \V \setminus \des{u}.$ In view of \eqref{cg-dec}, we must have $v \in \asc{u} \cupdot V_u.$ If $v \in \asc{u}$, then there exists $j \in \{1, \ldots, \dep_u\}$ such that $v= \parentn{j}{u}=u_j,$ and hence $v \in \des{u_j} \setminus \des{u_{j-1}}$. Hence we may further  assume  that $v \in V_u.$ By Lemma \ref{lem-join}, \beqn u \sqcap v \in [\rootb, u]=\{u_{\dep_u}, \ldots, u_0\}.\eeqn
Thus $u \sqcap v =u_j$ for some $j=0, \ldots, \dep_u$, and therefore $v \in \des{u_j}.$ However, by the uniqueness of the meet operation, $v \notin \des{u_{j-1}}$.
This completes the verification of \eqref{contend}.

\begin{case}
$\mf b \in V$ ~$\&$~ $\mf b \in \des{u}:$
\end{case}
\noindent
If $\mf b = u$, then by Remark \ref{rmk-action}, $\W{b}{u}=\D{u}$, and hence $\W{b}{u} \in B(\ell^2(V))$ if and only if $\lambdab_u   \in  \ell^{\infty}(V).$
Assume that $u \neq \mf b$. Then, by Proposition \ref{lem3.2}(ii),
\beqn
 \sum_{w \in V}  \Big|\sum_{\substack{v \in  M^{(\mf b)}_u(w)}}  \lambda_{uv} \, f(v) \Big|^2 &=&  \sum_{w \in \asc{u}}
\big| \lambda_{uw} \, f(w) \big|^2   ~+~
 \sum_{\substack{w \in  \desb{u}}}
 \big| \lambda_{uw} \, f(w) \big|^2 \\ &+&    \Big|\sum_{\substack{v \in [u, {\mf b}] }}  \lambda_{uv} \, f(v) \Big|^2.
\eeqn
Since $\asc{u}, [u, \mf b]$ are finite sets,
\beq
\label{domain-w-d}
\varLambda^{(\mf b)}_{u}f \in \ell^2(V) \quad \Longleftrightarrow  \quad \sum_{\substack{w \in \mathsf{Des}(u)}}
 \big| \lambda_{uw} \, f(w) \big|^2 < \infty.
\eeq
It is now clear that $\varLambda^{(\mf b)}_{u}f \in \ell^2(V)$ for every $f \in \ell^2(V)$ if and only if $\lambdab_u   \in  \ell^{\infty}(\mathsf{Des}(u)).$ This shows that \eqref{star} is a necessary condition.
Conversely, if \eqref{star} holds then $\varLambda^{(\mf b)}_{u}f \in \ell^2(V)$ for every $f \in \ell^2(V)$, and hence by the closed graph theorem together with Proposition \ref{closed}, $\W{b}{u}$ defines a bounded linear operator on $\ell^2(V).$

\begin{case}
$\mf b \in V$ ~$\&$~ $\mf b \notin \des{u}:$
\end{case}
\noindent
By Proposition \ref{lem3.2}(iii),
\beqn
 \sum_{w \in V}  \Big|\sum_{\substack{v \in M^{(\mf b)}_u(w)}}  \lambda_{uv} \, f(v) \Big|^2 ~ = ~ \sum_{\substack{w \in \mathsf{Des}_u(u)}}  \big|\lambda_{uw} \, f(w) \big|^2  ~ +   ~ \Big| \sum_{v \in \asc{u} \cup \{u, \mf b\}}\lambda_{uv} \, f(v) \Big|^2 .
\eeqn
Once again, since $\asc{u}$ is a finite set, we must have \eqref{domain-w-d}.
It follows that $\varLambda^{(\mf b)}_{u}f \in \ell^2(V)$ for every $f \in \ell^2(V)$ if and only if $\lambdab_u   \in  \ell^{\infty}(\mathsf{Des}(u)).$
The verification of the remaining part in this case is now similar to that of Case II.
\end{proof}
\begin{remark}  Note that the weighted join operator $\W{\rootb}{u}$  on ${\mathscr T}$
is bounded if and only if
$
\lambdab_u  \in  \ell^{\infty}(\mathsf{Des}(u)).$
Further, the weighted meet operator $\W{\infty}{u}$ on ${\mathscr T}$ is bounded if and only if
$
\lambdab_u   \in  \ell^2(V).$
\end{remark}

An examination of Cases II and III of the proof of Theorem \ref{bdd} yields a neat expression for the domain of weighted join operator $\W{b}{u}$, $\mf b  \neq \infty$:
\begin{corollary}
\label{bdd-coro}
Let $\mathscr T=(V, E)$ be a rooted directed tree with root $\rootb$ and let $\mathscr T_{\infty}=(\V, E_{\infty})$ be the extended directed tree associated with $\mathscr T$.
Let $u \in V,$ $\mf b \in V$ and let $\lambdab_u = \{\lambda_{uv}\}_{v\in \V}$ be a weight system
 of complex numbers. For $u \in V$, consider the weight system
$\lambdab_u = \{\lambda_{uv}\}_{v\in \V}$ of complex numbers and let $\D{u}$ be the diagonal operator with diagonal entries $\lambdab_u$. Then, for any $\mf b \in V$,
 the domain of the weighted join operator $\W{b}{u}$ on ${\mathscr T}$
is given by
 $\mathscr D(\W{b}{u}) = \mathscr D(P_{_{\des{u}}}\D{u}).$
\end{corollary}
\begin{proof} This is immediate from \eqref{domain-w-d}, which holds for any $\mf b \in V.$
\end{proof}

\begin{figure}
\begin{tikzpicture}[scale=.8, transform shape]
\tikzset{vertex/.style = {shape=circle,draw, blue, minimum size=.1em}}
\tikzset{edge/.style = {->,> = latex'}}
\node[] () at  (0,.4) {${}_{u_2\,=\,\rootb}$};
\node[] () at  (-.3,-2) {$u_0$};
\node[] () at  (-1.45,-1) {$u_1$};

\node[vertex] (a1) at  (-1,-1) {$$};
\node[vertex] (a2) at  (-2,-2) {$ $};
\node[vertex] (a4) at  (-3, -3) {$$};
\node[vertex] (a5) at  (-7/4, -3) {$$};
\node[vertex] (a8) at  (-4, -4) {$$};
\node[vertex] (a9) at  (-5/2, -4) {$$};
\node[vertex] (a10) at  (-2, -4) {$$};
\node[vertex] (a11) at  (-3/2, -4) {$$};

\tikzset{vertex/.style = {shape=circle,draw, magenta, minimum size=.1em}}

\node[vertex] (a3) at  (-3/4, -2) {$$};
\node[vertex] (a6) at  (-1, -3) {$$};
\node[vertex] (a7) at  (-1/2, -3) {$$};
\node[vertex] (a15) at (-1/2,-4) {$$};
\node[vertex] (a13) at  (-1, -4) {$$};

\tikzset{vertex/.style = {shape=circle,draw, black, minimum size=.1em}}


\node[vertex] (a0) at  (0,0) {};
\node[vertex] (b1) at  (1,-1) {};
\node[vertex] (b2) at  (2,-2) {};
\node[vertex] (b3) at  (3/4, -2) {};
\node[vertex] (b4) at  (3, -3) {};
\node[vertex] (b5) at  (7/4, -3) {};
\node[vertex] (b6) at  (1, -3) {};
\node[vertex] (b7) at  (1/2, -3) {};
\node[vertex] (b8) at  (4, -4) {};
\node[vertex] (b9) at  (5/2, -4) {};
\node[vertex] (b10) at  (2, -4) {};
\node[vertex] (b11) at  (3/2, -4) {};
\node[vertex] (b13) at  (1, -4) {};
\node[vertex] (b15) at (1/2,-4) {};
\draw[edge] (a0) to (a1);
\draw[edge] (a1) to (a2);
\draw[edge] (a1) to (a3);
\draw[edge] (a2) to (a4);
\draw[edge] (a2) to (a5);
\draw[edge] (a3) to (a6);
\draw[edge] (a3) to (a7);
\draw[edge] (a4) to (a8);
\draw[edge] (a4) to (a9);
\draw[edge] (a5) to (a10);
\draw[edge] (a5) to (a11);
\draw[edge] (a6) to (a13);
\draw[edge] (a7) to (a15);
\draw[edge] (a0) to (b1);
\draw[edge] (b1) to (b2);
\draw[edge] (b1) to (b3);
\draw[edge] (b2) to (b4);
\draw[edge] (b2) to (b5);
\draw[edge] (b3) to (b6);
\draw[edge] (b3) to (b7);
\draw[edge] (b4) to (b8);
\draw[edge] (b4) to (b9);
\draw[edge] (b5) to (b10);
\draw[edge] (b5) to (b11);
\draw[edge] (b6) to (b13);
\draw[edge] (b7) to (b15);
\end{tikzpicture}
\caption{The decomposition \eqref{contend} of $\V$ with $u=u_0$ of depth $2$ } \label{fig-decom}
\end{figure}



\section{A decomposition theorem}

One of the main results of this section shows that the weighted join operator $\W{b}{u}$ on a rooted directed tree
can be one of the
following three types, viz.  a diagonal operator, a rank one perturbation of a diagonal operator or a  finite rank operator.
Further, we obtain an orthogonal decomposition of $\W{b}{u}$ into a diagonal operator and a rank one operator provided $\mf b \neq u.$ Among various applications, we exhibit a family of weighted join operators with large null summand. It turns out that either a weighted join operator is complex Jordan or it has bounded Borel functional calculus.

Before we state the first main result of this section, we introduce the function $e_{\mu, A}$, which appears in the decomposition of weighted join operators. For a subset $A \subseteq V$ and $\mu :=\{\mu_{v} : v \in A\} \subseteq \mathbb C$, consider the function $e_{_{\mu, A}} : V \rar \mathbb C$ given by
\beq \label{e-lam}
e_{_{\mu, A}} \,:=\, \sum_{v \in A} \bar{\mu}_{v} e_v.
\eeq
Note that $e_{_{\mu, A}} \in \ell^2(V)$  if and only if $\mu \in \ell^2(A).$

\index{$e_{_{\mu, A}}$}

\begin{theorem} \label{dichotomy}
Let $\mathscr T=(V, E)$ denote a rooted directed tree with root $\rootb$ and let $\mathscr T_{\infty}=(\V, E_{\infty})$ be the extended directed tree associated with $\mathscr T$.
For $u \in V$, consider the weight system
$\lambdab_u = \{\lambda_{uv}\}_{v\in \V}$ of complex numbers and let $\D{u}$ be the diagonal operator with diagonal entries $\lambdab_u$. Then, for any $\mf b \in V \setminus \{u\}$,
 the weighted join operator $\W{b}{u}$ on ${\mathscr T}$
is given by
\begin{align*}
   \begin{aligned}      & \mathscr D(\W{b}{u}) \,=\, \mathscr D(P_{_{\des{u}}}\D{u}), \\
\W{b}{u}  ~= ~& \begin{cases}
\displaystyle P_{_{\asc{u} \cup  \Desb{u}}}\D{u}  + e_u \otimes e_{\lambdab_u, (u, \mf b]}
 & \mbox{if~} \mf b \in \des{u},\\
 \displaystyle P_{_{\des{u}}}\D{u}  + e_u \otimes e_{\lambdab_u, \asc{u} \cup \{\mf b\}}  & \mbox{otherwise}.
\end{cases}
\end{aligned}
\end{align*}
\end{theorem}
\begin{remark} \label{rmk-dicho}
In case $\mf b=u$, by Remark \ref{rmk-action}(iv), $\W{b}{u}$ is the diagonal operator $\D{u}$.
In case $\mf b = \infty,$  by \eqref{W-action}, \beq \label{meet-action}
\W{\mf \infty}{u} e_v \,=\,  \lambda_{uv}\, e_{u \sqcap v}, \quad v \in V.
\eeq
It now follows from Lemma \ref{lem-join} that $\W{\mf \infty}{u}$ is a finite rank operator.
Let us find an explicit expression for $\W{\mf \infty}{u}$. By \eqref{contend}, $\ell^2(V)$ admits the orthogonal decomposition
\beqn
\ell^2(V) \,=\,   \bigoplus_{j=0}^{\dep_u} \ell^2(\des{u_j} \setminus \des{u_{j-1}}),
\eeqn
where $\des{u_{-1}}=\emptyset$ and $u_j:=\parentn{j}{u}$ for $j =0, \ldots, \dep_u.$
By \eqref{meet-action}, with respect to the above decomposition, $\W{\mf \infty}{u}$ decomposes as
\beq \label{meet-deco}
\left.
\begin{array}{ccc}
\mathscr D(\W{\mf \infty}{u}) & = &  \displaystyle \bigoplus_{j=0}^{\dep_u} \mathscr D\big(e_{u_j} \obslash e_{_{\lambdab_u,  \des{u_j} \setminus \des{u_{j-1}}}}\big),  \\
\W{\mf \infty}{u} & = &  \!\!\!\!\!\!\!\!\!\!\displaystyle \bigoplus_{j=0}^{\dep_u} e_{u_j} \obslash e_{_{\lambdab_u,  \des{u_j} \setminus \des{u_{j-1}}}}.
\end{array}
\right\}
\eeq
Thus $\W{\mf \infty}{u}$ is an orthogonal direct sum of rank one operators
\beqn e_{u_j} \obslash e_{_{\lambdab_u,  \des{u_j} \setminus \des{u_{j-1}}}}, \quad j =0, \ldots, \dep_u. \eeqn
\end{remark}
\begin{proof} Let  $\mf b \in V \setminus \{u\}$. 
By Corollary \ref{bdd-coro}, $\mathscr D(\W{b}{u}) = \mathscr D(P_{_{\des{u}}}\D{u}).$
To see the decomposition of $\W{b}{u}$, consider the subset $M^{(\mf b)}_u(\infty)$ of $V$ as given in \eqref{M-u-b}.
By Proposition \ref{lem3.2}(iv),  
\beq \label{deco}  V \setminus M^{(\mf b)}_u(\infty)\,=\, \asc{u} \cup \des{u} \cup \{\mf b\} = (V \setminus V_u) \cup \{\mf b\}, \eeq
where $V_u$ is as given in \eqref{cg-dec0}.
Note that \eqref{deco} induces the orthogonal decomposition
\beq \label{deco-l2-V}
\ell^2(V) \,=\, \begin{cases} \ell^2(\asc{u}) \oplus  \ell^2(\des{u})  \oplus  \ell^2(M^{(\mf b)}_u(\infty)) & \mbox{if~} \mf b \in V \setminus V_u, \\
\ell^2(\asc{u}) \oplus  \ell^2(\des{u}) \oplus \ell^2(M^{(\mf b)}_u(\infty)) \oplus  \ell^2(\{\mf b\}) &  \mbox{otherwise}.
\end{cases}
\eeq
It may be concluded from Table \ref{Table1} that \beqn \W{b}{u}(\ell^2(\asc{u})) &\subseteq & \ell^2(V) \ominus \ell^2(M^{(\mf b)}_u(\infty)), \\ \W{b}{u}(\ell^2(\des{u})) & \subseteq & \ell^2(\des{u}), \\
\W{b}{u}(\ell^2(M^{(\mf b)}_u(\infty))) &=& \{0\}, \\
\W{b}{u}(\ell^2(\{\mf b\})) & \subseteq & \ell^2(\{u\}).
\eeqn
With respect to the orthogonal decomposition \eqref{deco-l2-V} of $\ell^2(V)$, the weighted join operator $\W{b}{u}$ decomposes as follows:
\beq \label{m-deco}
 & \W{b}{u} \,=\, \left[
\begin{array}{ccc}
W^{(\mf b)}_{11} & 0 & 0 \\
W^{(\mf b)}_{21} & W^{(\mf b)}_{22} & 0 \\
0 & 0 & 0 \\
\end{array}
\right]~& \mbox{if~}\mf b \in V \setminus V_u, \\ \label{m-deco1}
  & \W{b}{u} \,=\,
\left[
\begin{array}{cccc}
W^{(\mf b)}_{11} & 0 & 0  &  0\\
W^{(\mf b)}_{21} & W^{(\mf b)}_{22} & 0 &  \lambda_{u\mf b} e_u \otimes  e_{\mf b}\\
0 & 0 & 0 & 0\\
0 & 0 & 0 & 0 \\
\end{array}
\right]~&\mbox{if~}\mf b \in V_u.
\eeq
Note that for any $v \in \asc{u}$, by  Table \ref{Table1},
\beq \label{W11}
W^{(\mf b)}_{11} e_v &=&  P_{_{\asc{u}}}\W{b}{u} e_v \notag \\ &=&\lambda_{uv}\, P_{_{\asc{u}}} e_{u \sqcup_{\mf b} v} \notag \\ & = & \begin{cases}  0 & \mbox{if~} \mf b \in \asc{u}~\mbox{or~} \mf b \in V_u,  \\
\D{u} e_v & \mbox{if~} \mf b \in \des{u}.
\end{cases}
\eeq
A similar argument using Table \ref{Table1} shows that for any $v \in \asc{u}$, 
\beq \label{W21}
W^{(\mf b)}_{21} e_v &=& \lambda_{uv}\, P_{_{\des{u}}} e_{u \sqcup_{\mf b} v} \notag \\ &=& \begin{cases} \Big(\displaystyle \sum_{w \in \asc{u}}\lambda_{uw}\, e_u \otimes e_w\Big) e_v & \mbox{if~} \mf b \in \asc{u}~\mbox{or~} \mf b \in V_u, \\
 0 & \mbox{if~}\mf b \in \des{u}.
\end{cases}
\eeq
Further, for any $v \in \des{u}$, by Table \ref{Table1},
\beq \label{W22}
W^{(\mf b)}_{22} e_v &=& \lambda_{uv}\, e_{u \sqcup_{\mf b} v} \notag \\ &=& \begin{cases}  \D{u}e_v & \mbox{if~} \mf b \in \asc{u}~\mbox{or}~\mf b \in V_u, \\
\displaystyle \Big(\sum_{w \in [u, \mf b]}\lambda_{uw}\, e_u \otimes e_w\Big)e_v & \mbox{if~} \mf b \in \des{u}~\mbox{and~}v \in [u, \mf b], \\
 \D{u}e_v & \mbox{if~} \mf b \in \des{u}~\mbox{and~}v \notin [u, \mf b].
\end{cases}
\eeq
It is now easy to see that \beqn W^{(\mf b)}_{22} \,=\, \begin{cases} \D{u}|_{_{\ell^{2}(\des{u})}} & \mbox{if~} \mf b \in \asc{u}~\mbox{or}~\mf b \in V_u,
\\
\D{u}|_{_{\ell^{2}(\desb{u})}} ~ +  ~  \displaystyle\sum_{w \in [u, \mf b]}\lambda_{uw}\, e_u \otimes e_w & \mbox{if~} \mf b \in \des{u}.
\end{cases}
\eeqn
In view of \eqref{W11}, \eqref{W21}, \eqref{W22}, one may now deduce the desired decomposition from \eqref{m-deco} and \eqref{m-deco1}.
\end{proof}

Theorem \ref{dichotomy} together with Remark \ref{rmk-dicho} yields the following:

\begin{corollary}[Dichotomy]
Let $\mathscr T=(V, E)$ be a rooted directed tree with root $\rootb$ and let $\mathscr T_{\infty}=(\V, E_{\infty})$ be the extended directed tree associated with $\mathscr T$.
For $\mf b, u \in V$, consider the weight system
$\lambdab_u = \{\lambda_{uv}\}_{v\in \V}$ of complex numbers.
Then the weighted join operator $\W{b}{u}$ on ${\mathscr T}$ is at most rank one perturbation of a diagonal operator, while the weighted meet operator $\W{\mf \infty}{u}$ on ${\mathscr T}$ is a finite rank operator.
\end{corollary}

By Theorem \ref{dichotomy}, any $\W{b}{u} \in B(\ell^2(V))$ can be rewritten as $C + M$, where $C$ is a diagonal operator and $M$ is a nilpotent operator of nilpotency index $2$ given by
\beqn
C &=& \begin{cases}
\displaystyle P_{_{\asc{u} \cup  \Desb{u}}}\D{u}
 & \mbox{if~} \mf b \in \des{u}, \\
 \displaystyle P_{_{\des{u}}}\D{u}   & \mbox{otherwise},
\end{cases}
\\
M &=& \begin{cases}
e_u \otimes e_{\lambdab_u, (u, \mf b]}
 & \mbox{if~} \mf b \in \des{u}, \\
 e_u \otimes e_{\lambdab_u, \asc{u} \cup \{\mf b\}}  & \mbox{otherwise}.
\end{cases}
\eeqn
However, $\W{b}{u}$ is not a complex Jordan operator unless $\lambda_{uu}=0$. Indeed,
\beqn
CM-MC \,=\, \lambda_{uu} M.
\eeqn
It is worth noting that $(CM-MC)^2=0.$
Unfortunately, the above decomposition is not orthogonal.
Here is a way to get such a decomposition of $\W{b}{u}$.

\index{$\Db{b}{u}$}
\index{$\N{b}{u}$}

\begin{theorem} \label{o-deco}
Let $\mathscr T=(V, E)$ be a rooted directed tree with root $\rootb$ and let $\mathscr T_{\infty}=(\V, E_{\infty})$ be the extended directed tree associated with $\mathscr T$.
For $u \in V$, $\mf b \in V \setminus \{u\}$ and the weight system
$\lambdab_u = \{\lambda_{uv}\}_{v\in \V}$ of complex numbers,
let $\D{u}$ be the diagonal operator on $\mathscr T$ and let $\W{b}{u}$ be a weighted join operator on $\mathscr T$.
Consider
the closed subspace $\Hi{b}{u}$ of $\ell^2(V)$, given by
\beq \label{H-u}
\Hi{b}{u} &=& \begin{cases}
\displaystyle \ell^2(\asc{u} \cup  \desb{u})
 & \mbox{if~} \mf b \in \des{u},  \\
 \displaystyle \ell^2(\mathsf{Des}_u(u))   & \mbox{otherwise}.
\end{cases}
\eeq
Then the following statements hold:
\begin{enumerate}
\item The weighted join operator $\W{b}{u}$ admits the decomposition \beq \label{deco-W} \W{b}{u}\,=\,\Db{b}{u} \oplus \N{b}{u}~ \mbox{on~} \ell^2(V)\,=\,\Hi{b}{u} \oplus \big(\ell^2(V) \ominus \Hi{b}{u}\big),\eeq where $\Db{b}{u}$ is a densely defined diagonal operator in $\Hi{b}{u}$ and $\N{b}{u}$ is a bounded linear rank one operator on $\ell^2(V) \ominus \Hi{b}{u}$.
\item $\Db{b}{u}$ and $\N{b}{u}$ are given by
\beq \label{til-D}
\Db{b}{u} &=& \D{u}|_{_{\Hi{b}{u}}}, \quad \mathscr D(\Db{b}{u})\,=\,\big\{f \in \Hi{b}{u} : \D{u}(f \oplus 0) \in \Hi{b}{u}\big\},
\\
\label{til-N}
\N{b}{u} &=&
e_u \otimes e_{_{\lambdab_u, A_u}},
\eeq
where $e_{_{\lambdab_u, A_u}}$ is as given in \eqref{e-lam} and the subset $A_u$ of $V$ is given by
\beq
\label{A-u}
A_u &=& \begin{cases}
 [u, \mf b]
 & \mbox{if~} \mf b \in \des{u}, \\
\asc{u} \cup \{\mf b, u\}  & \mbox{otherwise}.
\end{cases}
\eeq
\end{enumerate}
\end{theorem}
\begin{proof} Since the orthogonal projection $P_{\Hi{b}{u}}$ commutes with $\D{u}$, one may appeal to \cite[Proposition 1.15]{Sc} to conclude that
$\Hi{b}{u}$ (identified with a subspace of $\ell^2(V)$) is a reducing subspace for $\D{u}$. The desired conclusions in (i) and (ii) now follow from Theorem \ref{dichotomy}.
\end{proof}

\index{$(\Db{b}{u}, \N{b}{u}, \Hi{b}{u})$}
\index{$\Hi{b}{u}$}
\index{$\Hi{\mf u}{u}$}

We find it convenient to denote the orthogonal decomposition \eqref{deco-W} of $\W{b}{u},$ $\mf b \neq u$, as ensured by Theorem \ref{o-deco}, by the triple $(\Db{b}{u}, \N{b}{u}, \Hi{b}{u})$, where $\Hi{b}{u}$, $\Db{b}{u}$ and $\N{b}{u}$ are given by \eqref{H-u}, \eqref{til-D} and \eqref{til-N} respectively. For the sake of convenience, 
we set 
\beq \label{H-u-1}
\Hi{\mf u}{u} \,:=\,\ell^2(V \setminus \{u\}). \eeq 
Note that $\W{\bf u}{u}$ admits the decomposition \eqref{deco-W} with $\N{\bf u}{u}=\lambda_{uu} e_u \otimes e_u.$ 
{\it In what follows, we will be interested in only those vertices $u \in V$ for which $\Hi{b}{u}$ is of infinite dimension.}

In the remaining part of this section, we present some immediate consequences of Theorem \ref{o-deco}.
\begin{corollary}
Let $\mathscr T=(V, E)$ be a rooted directed tree with root $\rootb$ and let $\mathscr T_{\infty}=(\V, E_{\infty})$ be the extended directed tree associated with $\mathscr T$. For $u \in V$, $\mf b \in V \setminus \{u\}$ and the weight system
$\lambdab_u = \{\lambda_{uv}\}_{v\in \V}$ of complex numbers, let $\W{b}{u}$ denote
 the weighted join operator  on ${\mathscr T}$ and
let $(\Db{b}{u}, \N{b}{u}, \Hi{b}{u})$ denote the orthogonal decomposition of $\W{b}{u}.$
If $\lambdab_u \in \ell^{\infty}(\des{u})$, then $\|\W{b}{u}\| = \max \big \{\|\Db{b}{u}\|, \|\N{b}{u}\|\big\}.$ Further,
\beqn  \|\Db{b}{u}\|  ~=
\displaystyle \sup_{v \in \supp\,\Hi{b}{u}} |\lambda_{uv}|,
\quad
\|\N{b}{u}\| \,=\,
\displaystyle \Big(\sum_{v \in A_u}|\lambda_{uv}|^2\Big)^{1/2},
\eeqn
where $A_u$ is given by \eqref{A-u}.
\end{corollary}
%

The following two corollaries give more insight into the structure of weighted join operators on rooted directed trees. The first of which is motivated by the work \cite{An}. We say that a densely defined linear operator $T$ in $\mathcal H$ admits {\it a large null summand} if it has an infinite dimensional reducing subspace contained in its kernel.
\begin{corollary} \label{lns}
Let $\mathscr T=(V, E)$ be a leafless, rooted directed tree with root $\rootb$ and let $\mathscr T_{\infty}=(\V, E_{\infty})$ be the extended directed tree associated with $\mathscr T$.
For $u \in V$, $\mf b \in V \setminus \{u\}$ and the weight system
$\lambdab_u = \{\lambda_{uv}\}_{v\in \V}$ of complex numbers, let $\W{b}{u}$ denote
 the weighted join operator  on ${\mathscr T}$.
If there exists a branching vertex $w \in \asc{u},$ then
 $\W{b}{u}$ has a large null summand.
\end{corollary}
\begin{proof}
Suppose there exists a branching vertex $w \in \asc{u}.$
By Theorem \ref{o-deco},
\beqn (\ell^2(V) \ominus \Hi{b}{u}) \ominus \ell^2(A_u) ~\subseteq ~
\ker \W{b}{u} \cap \ker \big(\W{b}{u}\big)^*, \eeqn where $A_u$ is a finite set given by \eqref{A-u}. In view of \eqref{cg-dec0} and \eqref{H-u}, it suffices to check that $\mbox{card}(V_u)= \aleph_0.$ However, since $\mathscr T$ is leafless, this is immediate from Proposition \ref{lem-lns}(ii).
\end{proof}


Complex Jordan weighted join operators exist in abundance.
\begin{corollary} \label{coro-c-Jordan}
Let $\mathscr T=(V, E)$ be a rooted directed tree with root $\rootb$ and let $\mathscr T_{\infty}=(\V, E_{\infty})$ be the extended directed tree associated with $\mathscr T$. For $u \in V$, $\mf b \in V \setminus \{u\}$ and the weight system
$\lambdab_u = \{\lambda_{uv}\}_{v\in \V}$ of complex numbers, let $\W{b}{u}$ denote
 the weighted join operator  on ${\mathscr T}$.
 Then the following holds:
\begin{enumerate}
\item If $\lambda_{uu}=0$, then $\W{b}{u}$ is a complex Jordan operator of index $2.$
\item If $\lambda_{uu} \neq 0$, then there exists a bounded homomorphism $\Phi : \mathcal B_{\infty}(\sigma(\W{b}{u})) \rar \mathcal B(\ell^2(V))$ given by \beqn \Phi(f)\,=\,f(\W{b}{u}), \quad f \in  \mathcal B_{\infty}(\sigma(\W{b}{u})), \eeqn where $\mathcal B_{\infty}(\Omega)$ denotes the algebra of bounded Borel functions from a closed subset $\Omega$ of $\mathbb C$ into $\mathbb C$.  In this case,  $\Phi$ extends the polynomial functional calculus.
\end{enumerate}
\end{corollary}
\begin{proof} Let $(\Db{b}{u}, \N{b}{u}, \Hi{b}{u})$ denote the orthogonal decomposition of $\W{b}{u}.$
To verify (i),
assume that $\lambda_{uu}=0.$
By Theorem \ref{dichotomy},
\beqn
\W{b}{u} ~ = ~
\displaystyle P_{\!\!\Hi{b}{u}}\D{u}  + \N{b}{u}.
\eeqn
By a routine inductive argument, we obtain 
\beq \label{nilp-powers}
{\N{b}{u}}^k \,=\,
 \lambdab^{k-1}_{uu} e_u \otimes e_{_{\lambdab_u, A_u}}, \quad k \geqslant 1.
\eeq
It follows that
$\N{b}{u}$ is nilpotent of nilpotency index $2$. Also, it is easily seen that \beqn \N{b}{u}(\W{b}{u}-\N{b}{u}) \subseteq (\W{b}{u}-\N{b}{u})\N{b}{u}=0, \eeqn
which completes the verification of (i).

\index{$\mathcal B_{\infty}(\Omega)$}

Assume next that $\lambda_{uu} \neq 0$.
In view of \eqref{deco-W} and \cite[Theorem 13.24]{R}, it suffices to check that $\N{b}{u}$ given by \eqref{til-N} admits a Borel functional calculus.
By \eqref{nilp-powers},
for any Borel measurable function $f : \sigma(\W{b}{u}) \rar \mathbb C$, $f(\N{b}{u})$ is a well-defined bounded linear operator given by
\beqn
f(\N{b}{u}) \,=\,
 \frac{f(\lambda_{uu})}{\lambda_{uu}} e_u \otimes e_{_{\lambdab_u, A_u}}.
\eeqn
Further, for any bounded Borel measurable function $f$ on $\sigma(\W{b}{u})$, \beqn \|f(\N{b}{u})\| ~\Le ~ \frac{|f(\lambda_{uu})|}{|\lambda_{uu}|}  \|e_{_{\lambdab_u, A_u}}\| ~\Le ~\|f\|_{\infty}\, \frac{\|e_{_{\lambdab_u, A_u}}\|}{|\lambda_{uu}|}.\eeqn
This completes the verification of (ii).
\end{proof}

The following is immediate from Theorem \ref{dichotomy} and the well-known characterization of diagonal compact operators \cite{Co}, once it is
observed that the class of bounded finite rank operators  is a  subset of  compact  operators, dense in the operator norm (in fact, it is also dense in Schatten $p$-class in its norm for every $p \geqslant 1$) 
(cf. \cite[Corollary 3.4.5]{Jablonski}). The
reader is referred to \cite{Si1} for the basics of operators in
Schatten classes.
\begin{proposition} \label{cpt}
Let $\mathscr T=(V, E)$ be a rooted directed tree with root $\rootb$ and let $\mathscr T_{\infty}=(\V, E_{\infty})$ be the extended directed tree associated with $\mathscr T$. For $u \in V,$ $\mf b \in V \setminus \{u\}$ and the weight system
$\lambdab_u = \{\lambda_{uv}\}_{v\in \V}$ of complex numbers,
let $\W{b}{u}$ denote
 the weighted join operator  on ${\mathscr T}$. Then, for any $p \in [1, \infty)$, the following hold:
\begin{enumerate}
\item $\W{b}{u}$ is compact if and only if $\displaystyle \lim_{v \in \des{u}} \lambda_{uv} =0$.
\item $\W{b}{u}$ is Schatten $p$-class if and only if $\displaystyle \sum_{v \in \des{u}} |\lambda_{uv}|^p < \infty$.
\end{enumerate}
Here the limit and sum are understood in a generalized sense (see \eqref{p-order}$)$.
\end{proposition}


We conclude this section with an application to the theory of commutators of compact operators. 
\begin{corollary} \label{cpt-comm}
Let $\mathscr T=(V, E)$ be a rooted directed tree with root $\rootb$ and let $\mathscr T_{\infty}=(\V, E_{\infty})$ be the extended directed tree associated with $\mathscr T$. For $u \in V,$ $\mf b \in V \setminus \{u\},$ and the weight system
$\lambdab_u = \{\lambda_{uv}\}_{v\in \V}$ of complex numbers,
 let $\W{b}{u}$ denote
 the weighted join operator  on ${\mathscr T}$. Then the following statements hold:
\begin{enumerate}
\item If $\displaystyle \lim_{v \in \des{u}} \lambda_{uv} =0$, then there exist compact operators $K, L \in B(\ell^2(V))$ such that
$\W{b}{u} = KL-LK.$
\item If $\displaystyle \sum_{v \in \des{u}} |\lambda_{uv}|^p < \infty$ for some $p \in [1, \infty)$, then there exist Schatten class operators $K, L \in B(\ell^2(V))$  such that
$\W{b}{u} =KL-LK.$
\end{enumerate}
\end{corollary}
\begin{proof}
This is immediate from Proposition \ref{cpt}, Corollary
\ref{lns} and Anderson's Theorems \cite[Theorems 1 and 3]{An}.
\end{proof}

\section{Commutant}

The main result of this section describes commutants of weighted join operators (cf. \cite[Proposition 5.4]{I}, \cite[Theorem 1.8]{FJKP-1}). In general,  the weighted join operators $\W{b}{u}$ do not belong to the class ($\mathcal R \mathcal O$) as introduced in \cite{FJKP-1}. Indeed, in contrast with \cite[Theorem 1.8]{FJKP-1}, the commutant of $\W{b}{u}$ need not be abelian (see Corollary \ref{coro-commutant}). 

\begin{theorem}[Commutant] \label{commutant}
Let $\mathscr T=(V, E)$ be a rooted directed tree with root $\rootb$ and let $\mathscr T_{\infty}=(\V, E_{\infty})$ be the extended directed tree associated with $\mathscr T$.
For $u \in V$ and $\mf b \in V \setminus \{u\}$, consider the weight system
$\lambdab_u = \{\lambda_{uv}\}_{v\in \V}$ of  complex numbers and assume that the weighted join operator $\W{b}{u}$ on $\mathscr T$ belongs to $B(\ell^2(V))$. Let $(\Db{b}{u}, \N{b}{u}, \Hi{b}{u})$ denote the orthogonal decomposition of $\W{b}{u}$, $A_u$  be as given in \eqref{A-u},  and let $W_u$ be given by
\beq
 \label{W-u} W_u &:=& \big\{v \in  \supp\, \Hi{b}{u} : \lambda_{uv} = \lambda_{uu} \big\}. \eeq
If $\ker \Db{b}{u}=\{0\}$ and $\N{b}{u} \neq 0$, then the following statements are equivalent:
\begin{enumerate}
\item $X \in B(\ell^2(V))$ belongs to the commutant $\{\W{b}{u}\}'$ of $\W{b}{u}$.
\item $X \in B(\ell^2(V))$ admits the orthogonal decomposition
\beqn
X \,=\, \left[
\begin{array}{cc}
P & f_0 \otimes e_{_{\lambdab_u, A_u}}  \\
e_u \otimes e_{{\mu}_u, W_u}  & S \\
\end{array}
\right]~\mbox{on~}\ell^2(V) = \Hi{b}{u} \oplus \big(\ell^2(V) \ominus \Hi{b}{u}\big),
\eeqn
where $P$ is a block diagonal operator in $\{\Db{b}{u}\}'$, $f_0 \in \ker(\Db{b}{u}-\lambda_{uu})$,
${\mu}_u:=\{\mu_{uv}\}_{v \in \supp\, \Hi{b}{u}}$ belongs to $\ell^2(\supp\, \Hi{b}{u}),$  and $S$ is any operator in  $B(\ell^2(V) \ominus \Hi{b}{u})$ such that
$Se_u= \inp{Se_{u}}{e_{u}}e_u$ and $S^*e_{_{\lambdab_u, A_u}}=\overline{\inp{Se_{u}}{e_{u}}}\, e_{_{\lambdab_u, A_u}}$.
\end{enumerate}
\end{theorem}
\begin{proof}
Assume that $\ker \Db{b}{u}=\{0\}$ and $\N{b}{u} \neq 0$.
Let $X \in B(\ell^2(V))$ be such that $X\W{b}{u}=\W{b}{u}X.$
We decompose $X$ as follows:
\beqn
X = \left[
\begin{array}{cc}
P & Q  \\
R & S \\
\end{array}
\right]~\mbox{on~}\ell^2(V) = \Hi{b}{u} \oplus \big(\ell^2(V) \ominus  \Hi{b}{u}\big).
\eeqn
A simple calculation shows that $X\W{b}{u}=\W{b}{u}X$ is equivalent to
\beq
\label{four-eq}
P\Db{b}{u}=\Db{b}{u} P, \quad \N{b}{u}S = S \N{b}{u}, \quad \N{b}{u}R=R\Db{b}{u}, \quad \Db{b}{u}Q=Q\N{b}{u}.
\eeq
We contend that $R$ is a finite rank operator with range contained in $[e_u].$ To see that, let $f \in \Hi{b}{u}$.
By \eqref{til-D} and \eqref{til-N},
\beqn
\N{b}{u}Rf=\inp{Rf}{e_{_{\lambdab_u, A_u}}} e_u, \quad  R\Db{b}{u}f =  \sum_{v \in \supp\, \Hi{b}{u}} \lambda_{uv} f(v) Re_v.
\eeqn
Thus, by the third equation of \eqref{four-eq}, we obtain
\beq
\label{5.10}
\inp{Rf}{e_{_{\lambdab_u, A_u}}} e_u ~=  \sum_{v \in \supp\, \Hi{b}{u}} \lambda_{uv} f(v) Re_v.
\eeq
Since \eqref{5.10} holds for arbitrary $f \in \Hi{b}{u}$ and $\lambda_{uv} \neq 0$ for every $v \in \supp\, \Hi{b}{u}$,
there exists a system ${\mu}_u:=\{\mu_{uv}\}_{\supp\, \Hi{b}{u}} \subseteq \mathbb C$ such that \beq \label{R} Re_v \,=\, \mu_{uv} e_u, \quad v \in \supp\, \Hi{b}{u}.\eeq
This immediately yields
\beqn
\inp{Rf}{e_{_{\lambdab_u, A_u}}} e_u ~= \sum_{v \in \supp\, \Hi{b}{u}} \lambda_{uu} f(v) \mu_{uv}   e_u.
\eeqn
Combining this with \eqref{5.10}, we obtain
\beqn
 \sum_{v \in \supp\, \Hi{b}{u}} \lambda_{uu} f(v) \mu_{uv}   e_u   ~= \sum_{v \in \supp\, \Hi{b}{u}} \lambda_{uv} f(v) \mu_{uv} e_u.
\eeqn
Since $f \in \Hi{b}{u}$ is arbitrary, this yields that $\mu_{uv}(\lambda_{uu}-\lambda_{uv})=0$ for every $v \in \supp\, \Hi{b}{u}.$
If $v \in \supp\,\Hi{b}{u} \setminus W_u$, then $\lambda_{uu} \neq \lambda_{uv}$ (see \eqref{W-u}), and hence $\mu_{uv}=0$. It may be now concluded from \eqref{R} that
$R = e_u \otimes e_{{\mu_u}, W_u}.$ Thus the claims stands verified.

Next we consider the equation $\Db{b}{u}Q=Q\N{b}{u}$ (see \eqref{four-eq}). Note that by \eqref{til-N}
 \beqn
 \Db{b}{u}Qe_u = Q\N{b}{u}e_u = \lambda_{uu} Qe_u,
 \eeqn
 which simplifies to $(\Db{b}{u}-\lambda_{uu})Qe_u=0$.
Since $\Db{b}{u}-\lambda_{uu}=0$ on $\ell^2(W_u)$ and injective on $\Hi{b}{u} \ominus \ell^2(W_u)$,
$Qe_u \in \ell^2(W_u).$ Moreover, \beqn \Db{b}{u}Q=Q(e_u \otimes e_{_{\lambdab_u, A_u}}) = (Qe_u) \otimes e_{_{\lambdab_u, A_u}}. \eeqn
 If $Qe_u=0$, then so is $\Db{b}{u}Q$, and hence $Q=0$ (since, by assumption, $\Db{b}{u}$ is injective). Suppose $Qe_u \neq 0.$ Then $Q : \ell^2(V) \ominus \Hi{b}{u} \rar \Hi{b}{u}$ is of rank one, since so is $\Db{b}{u}Q$ and $\Db{b}{u}$ is injective. Thus $Q=f_0 \otimes g_0$ for some $f_0 \in \Hi{b}{u}$ and $g_0 \in \ell^2(V) \ominus \Hi{b}{u}.$ It follows from $\Db{b}{u}Q=Q\N{b}{u}$ and \eqref{adjoint-ro} that
 \beq
 \label{Df0}
 (\Db{b}{u}f_0) \otimes g_0 = f_0 \otimes g_0\, \N{b}{u} = (f_0 \otimes g_0)  (e_u \otimes e_{_{\lambdab_u, A_u}})=
 \overline{g_0(u)}\, f_0 \otimes e_{_{\lambdab_u, A_u}}.
 \eeq
 Thus for any $h \in  \ell^2(V) \ominus \Hi{b}{u},$ $$\inp{h}{g_0} \Db{b}{u}f_0 \,=\, \overline{g_0(u)}\,  \inp{h}{e_{_{\lambdab_u, A_u}}} f_0.$$ Letting $h=g_0 - \frac{\inp{g_0}{e_{_{\lambdab_u, A_u}}}}{\|e_{_{\lambdab_u, A_u}}\|^2} e_{_{\lambdab_u, A_u}} \in \ell^2(V) \ominus \Hi{b}{u}$, we get
 $\inp{h}{g_0} \Db{b}{u}f_0 =0.$ However, since $\Db{b}{u}$ is assumed to be injective and $f_0 \neq 0$ (since $Q \neq 0$), $\inp{h}{g_0}=0.$ It follows that  $ |\inp{g_0}{e_{_{\lambdab_u, A_u}}}|=\|g_0\| \|e_{_{\lambdab_u, A_u}}\|.$
By the Cauchy-Schwarz inequality, we must have $g_0 = \bar{\alpha}\, e_{_{\lambdab_u, A_u}}$ for some  $\alpha \in \mathbb C.$  Further, $\alpha \neq 0$, since $g_0 \neq 0$ (otherwise $Q=0$). This also shows that $Q=\alpha f_0 \otimes e_{_{\lambdab_u, A_u}}.$
One may now infer from \eqref{Df0} (evaluated at $g_0= \bar{\alpha}\, e_{_{\lambdab_u, A_u}}$) that $\Db{b}{u}f_0 = \lambda_{uu}f_0.$
Further,
by \eqref{four-eq}, $P \in \{\Db{b}{u}\}'$ and $S \in \{\N{b}{u}\}'.$
The fact that $P$ is a block diagonal operator is a routine verification (see \cite[Proposition 6.1, Chapter IX]{Co}). The remaining part now follows from Lemma \ref{lem-inj-ten}(vi). This completes the proof of (i) $\Rightarrow$ (ii). The reverse implication is a routine verification using \eqref{four-eq}.
\end{proof}
\begin{remark} The injectivity of $\Db{b}{u}$ can be relaxed by replacing the basis $\{e_v\}_{v \in V}$ by $\{e_{\alpha(v)}\}_{v \in V}$ for some permutation $\alpha : V \rar V$. We leave the details to the reader.
\end{remark}

The following result is applicable to the case when the weight system $\lambdab_u : V \rar \mathbb C$ of the weighted join operator $\W{b}{u}$ is injective.
\begin{corollary} \label{coro-commutant}
Let $\mathscr T=(V, E)$ be a rooted directed tree with root $\rootb$ and let $\mathscr T_{\infty}=(\V, E_{\infty})$ be the extended directed tree associated with $\mathscr T$.
For $u \in V$ and $\mf b \in V \setminus \{u\}$, consider the weight system
$\lambdab_u = \{\lambda_{uv}\}_{v\in \V}$ of complex numbers and assume that the weighted join operator $\W{b}{u}$ on $\mathscr T$ belongs to $B(\ell^2(V))$. Let $(\Db{b}{u}, \N{b}{u}, \Hi{b}{u})$ denote the orthogonal decomposition of $\W{b}{u}$. Assume that $\ker \Db{b}{u}=\{0\}$ and $\N{b}{u} \neq 0$. If $\lambda_{uu} \notin \sigma_p(\Db{b}{u})$,
then
\beqn
\{\W{b}{u}\}' \,=\, \big\{P \oplus S : P \in \{\Db{b}{u}\}', ~S \in \{\N{b}{u}\}' \big\}.
\eeqn
\end{corollary}
\begin{proof}
Assume that $\lambda_{uu} \notin \sigma_p(\Db{b}{u})$ and let $X \in \{\W{b}{u}\}'$. Thus $X$ admits the decomposition as given in (ii) of Theorem \ref{commutant}. However, by \eqref{W-u}, $W_u= \emptyset,$ and hence $e_u \otimes e_{{\mu}_u, W_u} =0.$ Also, since $f_0 \in \ker(\Db{b}{u}-\lambda_{uu}),$
by our assumption, $f_0=0$. This completes the proof.
\end{proof}

Even under the assumptions of the preceding corollary, the commutant of a weighted meet operator can be non-abelian.
\begin{example}
Let $\mathscr T=(V, E)$ be a leafless, rooted directed tree and let $u \in V$ be such that $V_u=\emptyset$ (for example, take the directed tree $\mathscr T$ as shown in Figure \ref{fig0} and let $u=v_0$), where $V_u$ is as given in \eqref{cg-dec0}.
For $\mf b \in \des{u}$, consider the weight system
$\lambdab_u = \{\lambda_{uv}\}_{v\in \V}$ of distinct positive numbers and assume that the weighted join operator $\W{b}{u}$ on $\mathscr T$ belongs to $B(\ell^2(V))$. Let $(\Db{b}{u}, \N{b}{u}, \Hi{b}{u})$ denote the orthogonal decomposition of $\W{b}{u}$. Since $V_u = \emptyset,$  by \eqref{H-u} and \eqref{A-u}, we have
\beq \label{dim-A-u}
\ell^2(V) \ominus \Hi{b}{u} \,=\, \ell^2(A_u), \quad \mbox{where}~A_u=[u, \mf b].
\eeq
We claim that $\{\W{b}{u}\}'$ is non-abelian if and only if $\dim \big(\ell^2(V) \ominus \Hi{b}{u}\big) \Ge 3.$
By the preceding corollary, it suffices to check that $\{\N{b}{u}\}'$ is non-abelian if and only if $\dim \big(\ell^2(V) \ominus \Hi{b}{u}\big) \Ge 3.$
We consider the following cases:

\begin{cas1}
$\dim \big(\ell^2(V) \ominus \Hi{b}{u}\big) = 1:$
\end{cas1}
\noindent
By \eqref{dim-A-u},  $\mf b = u$, and hence $\W{b}{u}$ is the diagonal operator $\D{u}$ with distinct diagonal entries. In this case, $\{\W{b}{u}\}'$ is indeed maximal abelian \cite{Co}.

\begin{cas1}
$\dim \big(\ell^2(V) \ominus \Hi{b}{u}\big) = 2:$
\end{cas1}
\noindent
Consider the basis $\{e_u, e_{_{\lambdab_u, A_u}}\}$ of $\ell^2([u, \mf b])$ and
let $T \in \{\N{b}{u}\}'$. By Lemma \ref{lem-inj-ten}(iv), 
\beq \label{comm-eqn1} Te_u &=& \inp{Te_u}{e_u}e_u, \\ \label{comm-eqn2} T^*e_{_{\lambdab_u, A_u}} &=&\overline{\inp{Te_u}{e_u}}\,e_{_{\lambdab_u, A_u}}.\eeq  Let $\alpha_{_T}$ and $\beta_{_T}$ be scalars such that  $Te_{_{\lambdab_u, A_u}}=\alpha_{_T} e_u + \beta_{_T} e_{_{\lambdab_u, A_u}}.$ Then 
\beqn \inp{Te_{_{\lambdab_u, A_u}}}{e_{_{\lambdab_u, A_u}}}&=& \alpha_{_T} \lambda_{uu} + \beta_{_T} \|e_{_{\lambdab_u, A_u}}\|^2, \\ \inp{e_{_{\lambdab_u, A_u}}}{T^*e_{_{\lambdab_u, A_u}}} &\overset{\eqref{comm-eqn2}}=& \inp{Te_u}{e_u}  \|e_{_{\lambdab_u, A_u}}\|^2.\eeqn
It follows  that \beq \label{comm-dim2} \alpha_{_T} = \frac{\|e_{_{\lambdab_u, A_u}}\|^2}{\lambda_{uu}}(\inp{Te_u}{e_u}-\beta_{_T}).\eeq Hence, for any $S \in \{\N{b}{u}\}'$, we have
\beqn
STe_{_{\lambdab_u, A_u}} &=& \alpha_{_T} Se_u + \beta_{_T} Se_{_{\lambdab_u, A_u}} \\
&\overset{\eqref{comm-eqn1}}=& (\alpha_{_T} \inp{Se_u}{e_u} + \beta_{_T} \alpha_{_S}) e_u + \beta_{_T} \beta_S e_{_{\lambdab_u, A_u}}. \eeqn
By symmetry, $STe_{_{\lambdab_u, A_u}}=TSe_{_{\lambdab_u, A_u}}$ if and only if
$$\alpha_{_T} (\inp{Se_u}{e_u}-\beta_{_S}) \,=\,  \alpha_{_S} (\inp{Te_u}{e_u}-\beta_{_T}).$$
The later equality is immediate from \eqref{comm-dim2}. On the other hand, for any $S \in \{\N{b}{u}\}'$, by \eqref{comm-eqn1}, $STe_u=TSe_u$ always holds.
This shows that $\{\N{b}{u}\}'$ is abelian.

\begin{cas1}
$\dim \big(\ell^2(V) \ominus \Hi{b}{u}\big) = 3:$
\end{cas1}
\noindent
Let $f \in \ell^2(A_u)$ be orthogonal to $\{e_u, e_{_{\lambdab_u, A_u}}\}$.
Consider a bounded linear operator $T$ on $\ell^2(V) \ominus \Hi{b}{u}$ governed by
\beqn Te_u=\alpha e_u, \quad T e_{_{\lambdab_u, A_u}}=\beta e_u + \gamma f, \quad Tf=0, \eeqn where $\alpha, \beta, \gamma$ are complex numbers. 
Clearly, $T\N{b}{u}e_u =  \N{b}{u}Te_u,$ $T\N{b}{u}f =  \N{b}{u}Tf.$ Moreover, $T\N{b}{u}e_{_{\lambdab_u, A_u}}=\N{b}{u}T e_{_{\lambdab_u, A_u}}$ if and only if
$$\beta  = (\|e_{_{\lambdab_u, A_u}}\|^2 \inp{Te_u}{e_u})/\lambda_{uu}.$$
Consider another bounded linear operator $S$ on $\ell^2(V) \ominus \Hi{b}{u}$ governed by
\beqn Se_u=a e_u, \quad S e_{_{\lambdab_u, A_u}}=b e_u, \quad Sf=c e_u + de_{_{\lambdab_u, A_u}}, \eeqn where $a, b, c, d$ are complex numbers. A routine calculation shows that $S \in \{\N{b}{u}\}'$ if and only if
\beqn
b \lambda_{uu} - a \|e_{_{\lambdab_u, A_u}}\|^2 =0, \quad c \lambda_{uu} + d \|e_{_{\lambdab_u, A_u}}\|^2 =0.
\eeqn
On the other hand,
 $STf=TSf$ implies that $d \gamma  =0$, and hence for non-zero choices of $\gamma$ and $d,$ $S$ and $T$ do not commute. This shows that $\{\N{b}{u}\}'$ is not abelian.
 
Finally, in case $\dim \big(\ell^2(V) \ominus \Hi{b}{u}\big) \Ge 4,$
$\{\N{b}{u}\}'$ contains a copy of $B(\mathbb C^2)$, and hence it is
not abelian. \eop
\end{example}

\chapter{Rank one extensions of weighted join operators}

In this chapter, we introduce and study the class of rank one extensions of weighted join operators. We introduce the so-called compatibility conditions and discuss their roles in the closedness of these operators. We also discuss the problem of determining Hilbert space adjoint of these operators.
Further, we provide a complete spectral picture for members in this class and discuss some of its applications.

\index{$\W{b}{u}[f, g]=W_{f, g}$}

\begin{definition} \label{def-r}
Let $\mathscr T=(V, E)$ be a rooted directed tree with root $\rootb$ and let $\mathscr T_{\infty}=(\V, E_{\infty})$ be the extended directed tree associated with $\mathscr T$.
For $\mf b, u \in V$ and the weight system
$\lambdab_u = \{\lambda_{uv}\}_{v\in \V}$ of complex numbers,
let $\W{b}{u}$ be a weighted join operator on $\mathscr T$.
Consider the orthogonal decomposition $(\Db{b}{u}, \N{b}{u}, \Hi{b}{u})$ of $\W{b}{u}$.
By {\it a rank one extension of $\W{b}{u}$ on $\mathscr T$}, we understand the linear operator $\W{b}{u}[f, g]$ in $\ell^2(V)$ given by
\beq
\label{rone-extn}
\left.
\begin{array}{lll}
\mathscr D(\W{b}{u}[f, g]) &=& \big\{(h, k) : h \in \mathscr D(\Db{b}{u}) \cap \mathscr D(f \obslash g), ~k \in \ell^2(V) \ominus \Hi{b}{u}\big\} \\ && \\
\W{b}{u}[f, g] &=& \left[\begin{array}{cc}
\Db{b}{u} & 0 \\
f \obslash g & \N{b}{u}
\end{array}
\right],
\end{array}
\right\}
\eeq
where $f \in \ell^2(V) \ominus \Hi{b}{u}$ is non-zero and $g : \supp\,\Hi{b}{u} \rar \mathbb C$ is unspecified.
For the sake of convenience, we use the simpler notation $W_{f, g}$ in place of $\W{b}{u}[f, g]$.
\end{definition}
\begin{remark} \label{rmk-rone-extn}
Since $\N{b}{u}$ is bounded and the domains of $\Db{b}{u}$ and $f \obslash g$ contains the dense subspace $\mathscr D_{\supp\,\Hi{b}{u}}$ of $\Hi{b}{u}$ (see \eqref{d-supp-hbu}), $W_{f, g}$ is densely defined.
Since sum of a closed operator and a bounded linear operator is closed,  the rank one extension $W_{f, g}$ of $\W{b}{u}$ is closed provided $g \in \Hi{b}{u}.$ This happens in particular when $\des{u}$ has finite cardinality (see \eqref{H-u}).
In case
$g \notin \Hi{b}{u}$, $W_{f, g}$ need not be closed (cf. Corollary \ref{coro-spectrum}).  To see this assertion, consider the situation in which $\Db{b}{u}$ is bounded and $g \notin \Hi{b}{u}$.
 By Lemma \ref{lem-inj-ten-unb}, $f \obslash g$ is not closable, and hence there exists a sequence $\{h_n\}_{n \geqslant 0}$ in $\Hi{b}{u}$ such that $h_n \rar 0$, $\{(f \obslash g)(h_n)\}_{n \geqslant 0}$ is convergent but $(f \obslash g)(h_n) \nrightarrow 0$ as $n \rar \infty.$
Then $(h_n, 0) \rar (0, 0)$, $\{W_{f, g}(h_n, 0)\}_{n \geqslant 0}$ is convergent,  however, $W_{f, g}(h_n, 0) \nrightarrow 0,$ and hence $W_{f, g}$ is not even closable.
\end{remark}

Here is a remark about the manner in which $W_{f, g}$ is defined.
Certainly, one could have defined the rank one extension of
$\W{b}{u}$ with the entry $f \obslash g$ appearing on the extreme
upper right corner in \eqref{rone-extn}. It turns out, however, that
the operators defined this way are closed if and only if $g \in
\ell^2(V) \ominus \Hi{b}{u}$. From the view point of spectral
theory, these operators are of little importance in case $g \notin
\ell^2(V) \ominus \Hi{b}{u}$, and otherwise, these are bounded
finite rank perturbations of diagonal operators. Needless to say,
the later class has been studied extensively in the literature
(refer, for example, to \cite{St, I, FJKP-0, FJKP-1, FJKP, FX, JL,
Kl}). Also, the way in which $W_{f, g}$ is defined (cf.
\cite{Ag-St}, \cite{R-S}), it should be referred to as {\it rank one
co-extension} of $\W{b}{u}$. However, by abuse of terminology, we
refer to it as rank one extension of $\W{b}{u}$.

In what follows, we will be particularly interested in the following family of rank one extensions of weighted join operators (cf. Proposition \ref{symm-r-one} below).

\index{$\supp(h)$}

\begin{example} Let $W_{f, g}$ be a rank one extension of the weighted join operator $\W{b}{u}$ satisfying the intertwining relation:
\beq \label{inter-r}
(f \obslash g)\Db{b}{u} + \N{b}{u}(f \obslash g) = 0,
\eeq
where the linear operator on the left hand side is defined on the space $\mathscr D_{\supp\,\Hi{b}{u}}$ (see \eqref{d-supp-hbu}).
Note that \eqref{inter-r} is equivalent to
\beq \label{(i)} \lambda_{uv} f + \inp{f}{e_{_{\lambdab_u, A_u}}}e_u=0, \quad v \in \supp(g),
\eeq
where $A_u$ is given by \eqref{A-u} and the {\it support} $\supp(h)$ of the function $h : V \rar \mathbb C$ is given by $$\supp(h)\,:=\,\{v \in V : h(v) \neq 0\}.$$
Suppose  $g \neq 0$ and note that $\supp(g)$ is non-empty.
We make the following observations:
\begin{enumerate}
\item If $\lambda_{uv}=0$ for some $v \in \supp(g),$ then by \eqref{(i)}, $\inp{f}{e_{_{\lambdab_u, A_u}}}=0.$ Since $f \neq 0,$ by another application of \eqref{(i)},  $\lambda_{uw}=0$ for all $w \in \supp(g).$
\item If  $\lambda_{uv} \neq 0$ for some $v \in \supp(g),$ then by \eqref{(i)}, $f \in [e_u]\setminus \{0\}$ and $\lambda_{uw} = -\lambda_{uu}$ for every $w \in \supp(g).$
\end{enumerate}
The above discussion provides the following examples of $W_{f, g}$ satisfying
\eqref{inter-r}.
\begin{enumerate}
\item[(a)] $\lambda_{uv}=0$ for $v \in \supp(g)$, $\supp\,\Hi{b}{u} \setminus  \supp(g)$ is infinite and \beqn \sum_{w \in A_u} f(w) \lambda_{uw} =0.\eeqn
\item[(b)] $\lambda_{uv} =-\lambda_{uu} \neq 0$ for every $v \in \supp(g)$ and $f \in [e_u] \setminus \{0\}$.
\end{enumerate}
For example, if $\mathscr T$ is leafless and $u$ is a branching vertex, then the condition that $\supp\,\Hi{b}{u} \setminus  \supp(g)$ is infinite in (a) is ensured for any $g$ such that $\supp(g)=\des{w}$ provided
\beqn
w ~\mbox{belongs to~} \begin{cases} \child{\mf b} & \mbox{if~}\mf b  \in \des{u},\\
\child{u} & \mbox{if~}\mf b  \notin \des{u}.
  \end{cases}
\eeqn
In case (b) holds, then the rank one extension of $\W{b}{u}$ can be rewritten as the sum of a diagonal operator and the rank one operator $e_u \obslash \big(\overline{f(u)}\,g + e_{_{\lambdab_u, A_u}}\big).$
In these cases, $W_{f, g}$ satisfies
\beqn
W^2_{f, g} = (\Db{b}{u})^2 \oplus (\N{b}{u})^2~\mbox{on~}\mathscr D_{\supp\,\Hi{b}{u}} \oplus (\ell^2(V) \ominus \Hi{b}{u}).
\eeqn
Thus although $W_{f, g}$ does not have an diagonal decomposition, the intertwining relation \eqref{inter-r} ensures the same for its square.
\eop
\end{example}

The bounded rank one extensions of weighted join operators can be characterized easily.

 \begin{proposition}
Let $\mathscr T=(V, E)$ be a rooted directed tree with root $\rootb$ and let $\mathscr T_{\infty}=(\V, E_{\infty})$ be the extended directed tree associated with $\mathscr T$.
For $u, \mf b \in V,$ consider the weight system
$\lambdab_u = \{\lambda_{uv}\}_{v\in \V}$ of complex numbers and let $W_{f, g}$ be the rank one extension of the weighted join operator $\W{b}{u}$ on $\mathscr T$. Then the following are equivalent:
\begin{enumerate}
\item $W_{f, g}$ defines a bounded linear operator on $\ell^2(V)$.
\item $g \in \Hi{b}{u}$ and $\Db{b}{u}$ defines a bounded linear operator on $\Hi{b}{u}$.
\item $g \in \Hi{b}{u}$ and \beqn
 \lambdab_u   ~\mbox{belongs to~}  \begin{cases}
 \ell^{\infty}(V) & \mbox{if} ~\mf b = u, \\
  \ell^{\infty}(\mathsf{Des}(u)) & \mbox{otherwise}.
\end{cases}
\eeqn
\end{enumerate}
\end{proposition}
\begin{proof}
In view of Theorem \ref{bdd}, it suffices to see the equivalence of (i) and (ii). Since $\N{b}{u}$ is a bounded linear operator on $\ell^2(V) \ominus \Hi{b}{u}$, (ii) implies (i). To see the reverse implication, assume that $W_{f, g}$ is bounded linear on $\ell^2(V)$. Thus $\mathscr D(\Db{b}{u}) \cap \mathscr D(f \obslash g)=\Hi{b}{u}$. By the closed graph theorem, $\Db{b}{u}$, being closed operator defined on $\Hi{b}{u}$, is a bounded linear operator on $\Hi{b}{u}$. Further, for any $h \in \mathscr D(\Db{b}{u}) \cap \mathscr D(f \obslash g)$,
\beqn
\|W_{f, g}(h, 0)\|^2 &=& \Big\|(\Db{b}{u}h, \Big(\!\!\!\!\!\!\sum_{v \in \supp\,\Hi{b}{u}} \!\!\!\!\!\! h(v) \overline{g(v)}\Big)f)\Big\|^2 \\ &=& \|\Db{b}{u}h\|^2 + \|(f \obslash g)(h)\|^2.
\eeqn
Since $f \neq 0$ and $W_{f, g}$ is bounded, $f \obslash g$ must be bounded linear, and hence by Lemma \ref{lem-inj-ten-unb}, $g \in \Hi{b}{u}$. This completes the proof.
\end{proof}

\section{Compatibility conditions and discrete Hilbert transforms}

It turns out that the inclusion $\mathscr D(\Db{b}{u}) \subseteq \mathscr D(f \obslash g)$ of domains of $\Db{b}{u}$ and $f \obslash g$ plays a central role in deciding whether or not a given rank one extension $W_{f, g}$ of a weighted join operator is closed (cf. Remark \ref{rmk-rone-extn}). Indeed, we will see that the so-called compatibility conditions on $W_{f, g}$ always ensure the above inclusion as well as closedness of $W_{f, g}$. We formally introduce these conditions below (cf. \cite[Proposition 2.4(iv)]{I}, \cite[Proposition 4.1]{Kl}).

\index{$\mbox{dist}(\mu, \lambdab_u)$}
\index{$\varGamma_{\lambdab_u}$}
\index{$g_{_{\lambdab_u, \mu_0}}$}

\begin{definition}
Let $\mathscr T=(V, E)$ be a rooted directed tree with root $\rootb$ and let $\mathscr T_{\infty}=(\V, E_{\infty})$ be the extended directed tree associated with $\mathscr T$.
For $u, \mf b \in V,$ consider the weight system
$\lambdab_u = \{\lambda_{uv}\}_{v\in \V}$ of complex numbers and let $W_{f, g}$ be the rank one extension of the weighted join operator $\W{b}{u}$ on $\mathscr T$, where $f \in \ell^2(V) \ominus \Hi{b}{u}$ is non-zero and $g : \supp\,\Hi{b}{u} \rar \mathbb C$ is given. Set \beq \label{varG-0}
\mbox{dist}(\mu, \lambdab_u) &=& \inf \big\{|\mu - \lambda_{uv}| : v \in \supp\, \Hi{b}{u}\big\}, ~\mu \in \mathbb C, \notag \\
\varGamma_{\lambdab_u} &=& \big\{\mu \in \mathbb C : \mbox{dist}(\mu, \lambdab_u) > 0 \big\}.
\eeq
\begin{enumerate}
\item We say that {\it $W_{f, g}$ satisfies compatibility condition I} ~if there exists $\mu_0 \in \varGamma_{\lambdab_u}$ such that $g_{_{\lambdab_u, \mu_0}} \in \Hi{b}{u}$, where
\beq \label{g-lambda-mu}
g_{_{\lambdab_u, \mu_0}}(v) \,:=\, \frac{g(v)}{\lambda_{uv}-\mu_0}, \quad v \in \supp\, \Hi{b}{u}.
\eeq
\item We say that {\it  $W_{f, g}$ satisfies compatibility condition II}~ if the function $g$ satisfies
\beq \label{g-lambda-1}
\sum_{v \in \supp\,\Hi{b}{u}} \frac{|g(v)|^2}{|\lambda_{uv}|^2+1} ~<~ \infty.
\eeq
\end{enumerate}
If one of the above conditions holds, then we say that {\it $W_{f, g}$ satisfies a  compatibility condition}.
\end{definition}
\begin{remark} \label{one-all}
It turns out that $g_{_{\lambdab_u, \mu_0}} \in \Hi{b}{u}$ for some $\mu_0 \in \varGamma_{\lambdab_u}$, then $g_{_{\lambdab_u, \mu}} \in \Hi{b}{u}$ for every $\mu \in \varGamma_{\lambdab_u}$. This may be derived from
\beqn
|\lambda_{uv} - \mu_0|^2 ~\Le ~2(|\lambda_{uv} - \mu|^2 + |\mu_0 - \mu|^2), \quad v \in \supp\,\Hi{b}{u}, ~\mu \in \varGamma_{\lambdab_u} \setminus \{\mu_0\},
\eeqn
and the fact that $|\mu_0 - \mu| \Le c \dist(\mu, \lambdab_u)$ for some $c > 0$ (see \eqref{varG-0}).

\end{remark}

\index{$H_{\lambdab_u, g}$}

Here we discuss the relationships between the above compatibility conditions and the domain inclusion $\mathscr D(\Db{b}{u}) \subseteq \mathscr D(f \obslash g)$.
\begin{proposition} \label{lem-H-transform}
Let $\mathscr T=(V, E)$ be a rooted directed tree with root $\rootb$ and let $\mathscr T_{\infty}=(\V, E_{\infty})$ be the extended directed tree associated with $\mathscr T$.
For $u, \mf b \in V,$ consider the weight system
$\lambdab_u = \{\lambda_{uv}\}_{v\in \V}$ of complex numbers and let $W_{f, g}$ be the rank one extension of the weighted join operator $\W{b}{u}$ on $\mathscr T$, where $f \in \ell^2(V) \ominus \Hi{b}{u}$ is non-zero and $g : \supp\,\Hi{b}{u} \rar \mathbb C$ is given.
If $W_{f, g}$ satisfies a compatibility condition, then we have the domain inclusion $\mathscr D(\Db{b}{u}) \subseteq \mathscr D(f \obslash g)$. Moreover, if $\varGamma_{\lambdab_u}$ is non-empty, then the following statements are equivalent:
\begin{enumerate}
\item $\mathscr D(\Db{b}{u}) \subseteq \mathscr D(f \obslash g)$.
\item $W_{f, g}$ satisfies compatibility condition I.
\item The discrete Hilbert transform
$H_{\lambdab_u, g}$ given by
\beqn
H_{\lambdab_u, g}(h) ~= \sum_{v \in \supp\, \Hi{b}{u}} \frac{h(v) \overline{g(v)}}{\mu-\lambda_{uv}}
\eeqn
is well-defined for every $\mu \in \varGamma_{\lambdab_u}$ and every $h \in \Hi{b}{u}.$
\item For every $\mu \in \varGamma_{\lambdab_u}$, the linear operator $L_{\lambdab_u, \mu}:=(f \obslash g)(\Db{b}{u}-\mu)^{-1}$ defines a bounded linear transformation from $\Hi{b}{u}$  into $\ell^2(V) \ominus \Hi{b}{u}$.
\end{enumerate}
\end{proposition}
\begin{remark}
Note that $\sigma(\Db{b}{u})=\mathbb C \setminus \varGamma_{\lambdab_u}$ (see \cite[Example 3.8]{Sc}). This clarifies the expression $(\Db{b}{u}-\mu)^{-1}$ appearing in (iv).
It is worth noting  that a discrete Hilbert transform appears in \cite[Corollary 2.5]{I}, which characterizes the set of eigenvalues of a bounded rank one perturbation of a diagonal operator (see also \cite[Corollary 2.6]{I}). Also, the operator $L_{\lambdab_u, \mu}$, as appearing in Proposition \ref{lem-H-transform}(iv), is precisely the operator $G(\mu)$, as appearing in the Frobenius-Schur-type factorization in  \cite[Equation (1.6)]{ALMS}.
\end{remark}
\begin{proof}
Let $h \in \mathscr D(\Db{b}{u})$. We divide the verification of the domain inclusion $\mathscr D(\Db{b}{u}) \subseteq \mathscr D(f \obslash g)$ into the following cases:

Assume that $W_{f, g}$ satisfies the compatibility condition I. Then, for some $\mu_0 \in \varGamma_{\lambdab_u}$, by the Cauchy-Schwarz inequality,
\beqn
\Big|\sum_{v \in \supp\, \Hi{b}{u}} \!\!\!\!\!\!\!\! h(v)\overline{g(v)}\Big| \,\leqslant\, \|(\Db{b}{u}-\mu_0)h\|\|g_{_{\lambdab_u, \mu_0}}\|.
\eeqn
where we used \eqref{g-lambda-mu} and the fact that $\mathscr D(\Db{b}{u}-\mu_0) = \mathscr D(\Db{b}{u})$.
This shows that $h \in \mathscr D(f\obslash g).$

Next assume that $W_{f, g}$ satisfies the compatibility condition II. Since $\mathscr D(\Db{b}{u})=\mathscr D((\Db{b}{u})^*\Db{b}{u} +I)^{1/2})$, by the Cauchy-Schwarz inequality,
\beqn
\big|\sum_{v \in \supp\, \Hi{b}{u}} \!\!\!\!\!\!\!\! h(v)\overline{g(v)}\Big|^2 \,\leqslant\,\|(\Db{b}{u})^*\Db{b}{u} +I)^{1/2}h\|^2 \sum_{v \in \supp\,\Hi{b}{u}} \frac{|g(v)|^2}{|\lambda_{uv}|^2+1}.
\eeqn
Thus $h \in \mathscr D(f\obslash g)$ in this case, as well.

The preceding discussion also yields the implication (ii) $\Rightarrow$ (i).
To see the equivalence of (i)-(iv), assume that $\varGamma_{\lambdab_u}$ is non-empty.

(i) $\Rightarrow$ (ii): Let $\mu \in \varGamma_{\lambdab_u}$.
For $h \in \Hi{b}{u}$, consider the function $k_h : \supp\, \Hi{b}{u} \rar \mathbb C$ defined by 
\beq \label{k-h} k_h(v) \,=\, \frac{h(v)}{\lambda_{uv}-\mu}, \quad v \in \supp\, \Hi{b}{u}.
\eeq
Clearly, $k_h$
belongs to $\mathscr D(\Db{b}{u})$ for every $h \in \Hi{b}{u}.$
By assumption, $
\mathscr D(\Db{b}{u}) \subseteq \mathscr D(f \obslash g)$, and hence the linear functional $\phi_g : \Hi{b}{u} \rar \mathbb C$ given by
\beq \label{phi-g-h}
\phi_g(h)\,=\,\sum_{v \in \supp\, \Hi{b}{u}} k_h(v) \overline{g(v)}, \quad h \in \Hi{b}{u}
\eeq
is well-defined.
By the standard polar representation, the series $\sum_{v \in \supp\, \Hi{b}{u}} k_h(v) \overline{g(v)}$ is absolutely convergent for every $h \in \Hi{b}{u}.$
One may now apply the uniform boundedness principle \cite{Si} to the family of linear functionals $\phi_{F, g}(h)=\sum_{v \in F} h(v) \frac{\overline{g(v)}}{\lambda_{uv}-\mu},$ $F$ is a finite subset of $\supp\, \Hi{b}{u}$ to derive the boundedness of $\phi_g.$ By \eqref{k-h} and \eqref{phi-g-h}, the boundedness of $\phi_g$ in turn is equivalent to $g_{_{\lambdab_u, \mu}} \in \Hi{b}{u}$.


(iii) $\Rightarrow$ (ii): This may be derived from the uniform boundedness principle (see the verification of (i) $\Rightarrow$ (ii)).

(ii) $\Rightarrow$ (iii): In view of the Cauchy-Schwarz inequality, it suffices to check that $g_{_{\lambdab_u, \mu_0}} \in \Hi{b}{u}$ for some $\mu_0 \in \varGamma_{\lambdab_u}$, then $g_{_{\lambdab_u, \mu}} \in \Hi{b}{u}$ for every $\mu \in \varGamma_{\lambdab_u}$. This is observed in Remark \ref{one-all}.

(ii) $\Rightarrow$ (iv):
By the Cauchy-Schwarz inequality, for any $h \in \Hi{b}{u},$
\beqn
\|L_{\lambdab_u, \mu}\| & = & \|(f \obslash g)(\Db{b}{u}-\mu)^{-1}h\| \\ & = & \Big|\sum_{v \in \supp\, \Hi{b}{u}} \!\!\!\!\!\!\!\! h(v)\frac{\overline{g(v)}}{\lambda_{uv}-\mu}\Big| \|f\| \leqslant\|h\|\|g_{_{\lambdab_u, \mu}}\|.
\eeqn
Since $g_{_{\lambdab_u, \mu}} \in \Hi{b}{u}$, this shows that $L_{\lambdab_u, \mu}$ is bounded linear.

 (iv) $\Rightarrow$ (iii): This is straight-forward.
\end{proof}

Here is an instance in which the compatibility condition II implies
the compatibility condition I.
\begin{corollary}
\label{coro-H-tr}
Under the hypotheses of Proposition \ref{lem-H-transform} and the assumption
 that $\varGamma_{\lambdab_u} \neq \emptyset$, if $W_{f, g}$ satisfies the compatibility condition II, then it satisfies the compatibility condition I.
\end{corollary}
\begin{proof}
If $W_{f, g}$ satisfies the compatibility condition II, then by the first half of Proposition \ref{lem-H-transform}, we obtain $\mathscr D(\Db{b}{u}) \subseteq \mathscr D(f \obslash g)$. The desired conclusion now follows from the implication (i) $\Rightarrow$ (ii) of Proposition \ref{lem-H-transform}.
\end{proof}

\index{$ (\varGamma, g)^*$}

Following \cite{BMS},  for any $g : \supp\,\Hi{b}{u} \rar \mathbb C \setminus \{0\},$ we set
 \beq
 \label{gamma-g}
 (\varGamma, g)^* \,=\, \{\mu \in \mathbb C : g_{\lambdab_u, \mu} \in \Hi{b}{u}\}.
 \eeq
The following has been motivated by the discussion from \cite[Pg 2]{BMS}  on discrete Hilbert transforms in a slightly different context. Note that the compatibility condition II is nothing but the existence of {\it admissible weight sequence} in the sense of \cite{BMS}.
\begin{proposition} \label{CC-prop}
 Let $\mathscr T=(V, E)$ be a rooted directed tree with root $\rootb$ and let $\mathscr T_{\infty}=(\V, E_{\infty})$ be the extended directed tree associated with $\mathscr T$.
For $u, \mf b \in V,$ consider the weight system
$\lambdab_u = \{\lambda_{uv}\}_{v\in \V}$ of complex numbers and let $W_{f, g}$ be the rank one extension of the weighted join operator $\W{b}{u}$ on $\mathscr T$, where $f \in \ell^2(V) \ominus \Hi{b}{u}$ is non-zero and $g : \supp\,\Hi{b}{u} \rar \mathbb C \setminus \{0\}$ be given. Then the following statements are true:
\begin{enumerate}
\item  If $\{\lambda_{uv} : v \in \supp\,\Hi{b}{u}\}$ is closed, then $W_{f, g}$ satisfies the compatibility condition I if and only if $$\varGamma_{\lambdab_u} = \mathbb C \setminus \{\lambda_{uv} : v \in \supp\,\Hi{b}{u}\}= (\varGamma, g)^*$$
$($see \eqref{varG-0} and \eqref{gamma-g}$)$.
\item If $\{\lambda_{uv} : v \in \supp\,\Hi{b}{u}\}$ has accumulation point only at $\infty$ with each of its entries appearing finitely many times, then  $W_{f, g}$ satisfies compatibility condition II if and only if $\varGamma_{\lambdab_u}=(\varGamma, g)^*.$
\end{enumerate}
\end{proposition}
\begin{proof}
To see (i), assume that $\{\lambda_{uv} : v \in \supp\,\Hi{b}{u}\}$ is closed.
Clearly, \beq \label{inc-1} \varGamma_{\lambdab_u} \,=\, \mathbb C \setminus \{\lambda_{uv} : v \in \supp\,\Hi{b}{u}\}.\eeq Since $g$ is nowhere vanishing,  \beq \label{inc-2} (\varGamma, g)^* ~\subseteq ~
 \mathbb C \setminus \{\lambda_{uv} : v \in \supp\,\Hi{b}{u}\}.\eeq
 Note further that if $W_{f, g}$ satisfies the compatibility condition I, then
 $g_{_{\lambdab_u, \mu}}$ belongs to $\Hi{b}{u}$ for every $\mu \in \varGamma_{\lambdab_u}$ (see Remark \ref{one-all}).
 In this case, $\varGamma_{\lambdab_u} \subseteq  (\varGamma, g)^*,$ and hence the necessity part in (i) follows from \eqref{inc-1} and \eqref{inc-2}. Since
$ \mathbb C \setminus \{\lambda_{uv} : v \in \supp\,\Hi{b}{u}\}$ is always a nonempty set (by the assumption that $\mbox{card}(V)=\aleph_0$, $\supp\,\Hi{b}{u}$ is always countable), the sufficiency part of (i) is immediate from \eqref{gamma-g}.

To see (ii), assume that $\{\lambda_{uv} : v \in \supp\,\Hi{b}{u}\}$ has accumulation point only at $\infty$ with each of its entries appearing finitely many times. The necessity part follows from (i), \eqref{inc-1} and Corollary \ref{coro-H-tr}. To see the sufficiency part of (ii), suppose that $(\varGamma, g)^*=\varGamma_{\lambdab_u}$. By (i), for some $\mu_0 \in \mathbb C,$ we must have
\beq \label{star1}
\sum_{v \in \supp\,\Hi{b}{u}} \frac{|g(v)|^2}{|\mu_0-\lambda_{uv}|^2} ~< ~\infty.
\eeq
However, since $\infty$ is the only accumulation point for $\{\lambda_{uv} : v \in \supp\,\Hi{b}{u}\}$,
there exists (sufficiently large) $M >0$ such that $|\mu_0-\lambda_{uv}|^2 \leqslant M(|\lambda_{uv}|^2+1)$ for every $v \in \supp\,\Hi{b}{u}$. It now follows from \eqref{star1} that $W_{f, g}$ satisfies the compatibility condition II.
\end{proof}


\section{Closedness and relative boundedness}

In this section, we show that any rank one extension $W_{f, g}$ of weighted join operator satisfying a compatibility condition is closed (cf. \cite[Theorem 1.1]{ALMS}, \cite[Theorems 2.5 and 2.6]{T}). This is achieved by decomposing $W_{f, g}$ as $A+B$, where $A$ is closed and $B$ is $A$-bounded.  We begin recalling some definitions from \cite{K}. Given densely defined linear operators $A$ and $B$ in $\mathcal H,$ we say that $B$ is {\it $A$-bounded} if $\mathscr D(B) \supseteq \mathscr D(A)$ and there
exist non-negative real numbers $a$ and $b$ such that
\beqn \|Bx\|^2 \,\leqslant\, a\|Ax\|^2 + b\|x\|^2, \quad x \in \mathscr D(A). \eeqn
The infimum of all $a \geqslant 0$ for which there exists a number $b \geqslant 0$ such that the above inequality holds is called the {\it $A$-bound of $B$}. 
Note that $B$ is $A$-bounded if and only if $\mathscr D(B) \supseteq \mathscr D(A)$ and there
exist non-negative real numbers $a$ and $b$ such that
\begin{eqnarray}
\label{rel-bdd-2}
\|Bx\| \,\leqslant\, a \|Ax\| + b\|x\|,\quad x \in \mathscr D(S).
\end{eqnarray}
For basic facts pertaining to $A$-bounded operators, the reader is referred to \cite{K, SS, Sc}. For the sake of convenience, we recall here the statement of Kato-Rellich theorem from \cite{K}. 
Suppose that $A$ is a closed operator in $\mathcal H$. Let $B$ be a linear operator such that $\mathscr D(A) \subseteq \mathscr D(B)$ and there exist $a \in (0,1)$ and $b \in (0, \infty)$ with the property \eqref{rel-bdd-2},
then  the linear operator $A+B$ with domain $\mathscr D(A)$ is a closed operator in $\mathcal H$ (see \cite[Theorem 1.1, Chapter IV]{K}).

\begin{theorem}
\label{closed-0}
Let $\mathscr T=(V, E)$ be a rooted directed tree with root $\rootb$ and let $\mathscr T_{\infty}=(\V, E_{\infty})$ be the extended directed tree associated with $\mathscr T$.
For $u, \mf b \in V,$ consider the weight system
$\lambdab_u = \{\lambda_{uv}\}_{v\in \V}$ of complex numbers and let $W_{f, g}$ be the rank one extension of the weighted join operator $\W{b}{u}$ on $\mathscr T$, where $f \in \ell^2(V) \ominus \Hi{b}{u}$ is non-zero and $g : \supp\,\Hi{b}{u} \rar \mathbb C$ is given.
Suppose that $W_{f, g}$ satisfies a compatibility condition.
Then $W_{f, g}$ defines a closed linear operator with domain given by
\beq
\label{domain-W-fg}
\mathscr D(W_{f, g})\,=\, \big\{(h, k) : h \in \mathscr D(\Db{b}{u}), ~k \in \ell^2(V) \ominus \Hi{b}{u}\big\}.
\eeq
Moreover, $\mathscr D_{V}$, as given by \eqref{d-supp-hbu}, forms a core for $W_{f, g}.$
\end{theorem}
\begin{proof} 
Suppose that $W_{f, g}$ satisfies the compatibility condition I for some $\mu \in  \varGamma_{\lambdab_u}$. Let $a$ be a positive real number less than $1.$
Since $g_{_{\lambdab_u, \mu}} \in \Hi{b}{u}$, there exists a finite subset $F$ of $\supp\,\Hi{b}{u}$ such that
\beq \label{tail}
\sum_{\supp\, \Hi{b}{u} \setminus F}\frac{|g(v)|^2}{|\lambda_{uv}-\mu|^2}
\,\leqslant\, \frac{a}{4\,\|f\|^2}.
\eeq
Define closed linear operators $N_{F}$ and $D_F$ in $\Hi{b}{u}$ by \beqn N_F=\sum_{v \in F}\lambda_{uv} e_v \otimes e_v, \quad  D_F=\Db{b}{u}-N_F.\eeqn
Further let $g_{_F}= \sum_{v \in \supp\, \Hi{b}{u} \setminus F}g(v)e_v$.
We rewrite $W_{f, g}$ as $A + B + C,$ where $A, B, C$ (with their natural domains) are densely defined operators in $\ell^2(V)$ given by
 \beqn
\begin{array}{lll}
A &=& \left[\begin{array}{cc}
D_F & 0 \\
0 &  0
\end{array}
\right],
\end{array}
\begin{array}{lll}
B &=& \left[\begin{array}{cc}
0 & 0 \\
f \obslash g_{_F} &  0
\end{array}
\right],
\end{array}
\begin{array}{lll}
C &=& \left[\begin{array}{cc}
N_F & 0 \\
f \otimes (g-g_{_F}) &  \N{b}{u}
\end{array}
\right].
\end{array}
\eeqn
Note that $\mathscr D(D_F)=\mathscr D(\Db{b}{u})$ and $\mathscr D(f \obslash g_{_F})=\mathscr D(f \obslash g)$. Furthermore, $A$ is a closed linear operator in $\ell^2(V)$ and $C$ is bounded linear operator on $\ell^2(V).$ Moreover,  $\mathscr D(D_F)
\subseteq \mathscr D(f \obslash g_{_F})$ (cf. Proposition \ref{lem-H-transform}).
Indeed, by the Cauchy-Schwarz inequality and \eqref{tail}, for any $h \in \mathscr D(D_F),$
\beq
\label{f-g-D-domain}
\Big|\sum_{v \in \supp\, \Hi{b}{u} \setminus F} \!\!\!\!\!\!\!\! h(v)\overline{g(v)}\,\Big| \,\leqslant\,\frac{\sqrt{a}}{2}\, \frac{\|(D_F-\mu)h\|}{\|f\|}.
\eeq
Since $C$ is a bounded operator, we obtain that $\mathscr D(A)
\subseteq \mathscr D(B+C)$.
We claim that for all $h \in \mathscr D(A),$
\beq
\label{kbs-estimate}
\|(B+C)h\|^2\,\leqslant\,a \, \|Ah\|^2 + b(\mu)\,\|h\|^2,
\eeq
where $b(\mu)=a|\mu|^2 + 2\|C\|^2.$
To see the claim, let $h=(h_1, h_2) \in \mathscr D(A)$. By repeated applications of $|\alpha+\beta|^2 \leqslant2(|\alpha|^2+|\beta|^2)$, $\alpha, \beta \in \mathbb C$, we obtain
\beqn
\|(B+C)h\|^2 & \leqslant & 2\|(f \obslash g_{_F})h_1\|^2 + 2\|C\|^2\|h\|^2 \notag \\
&  \overset{\eqref{f-g-D-domain}} \leqslant & \frac{a}{2} \|(D_F-\mu)h_1\|^2  + 2\|C\|^2\|h\|^2 \notag \\
& \leqslant & a \|D_F h_1\|^2 + a |\mu|^2 \|h_1\|^2  + 2\|C\|^2\|h\|^2 \notag \\
& \leqslant & a\|Ah\|^2 + \big(a |\mu|^2 + 2\|C\|^2\big)\|h\|^2.
\eeqn
This completes the verification of \eqref{kbs-estimate}. Thus $B+C$ is $A$-bounded with $A$-bound less than $1.$ Hence, by Kato-Rellich Theorem, $W_{f, g}$ is a closed operator with domain given by \eqref{domain-W-fg}.

Next suppose that $W_{f, g}$ satisfies the compatibility condition II.
Let \beq \label{Gm} G_{m}~:=\sum_{v \in \supp\,\Hi{b}{u}} \frac{|g(v)|^2}{|\lambda_{uv}|^2+m^2}, \quad m \geqslant 1. \eeq Note that
\beq \label{Gm-lt-G1}
0 \leqslant G_m \leqslant G_1 < \infty, \quad m \geqslant 1.
\eeq
We rewrite $W_{f, g}$ as $A + B + C,$ where $A, B, C$ (with their natural domains) are densely defined operators in $\ell^2(V)$ given by
 \beq \label{ABC}
\begin{array}{lll}
A &=& \left[\begin{array}{cc}
\Db{b}{u} & 0 \\
0 &  0
\end{array}
\right],
\end{array}
\begin{array}{lll}
B &=& \left[\begin{array}{cc}
0 & 0 \\
f \obslash g &  0
\end{array}
\right],
\end{array}
\begin{array}{lll}
C &=& \left[\begin{array}{cc}
0 & 0 \\
0 &  \N{b}{u}
\end{array}
\right].
\end{array}
\eeq
Note that $A$ is a closed linear operator in $\ell^2(V)$ and $C$ is bounded linear operator on $\ell^2(V).$ Given a positive integer $m,$ consider the inner-product space $\mathscr D(A)$ endowed with the inner-product \beqn \inp{x}{y}_{A, m} \,=\, \inp{Ax}{Ay} + m^2\inp{x}{y}, \quad x, y \in \mathscr D(A).\eeqn
Since $A$ is a closed linear operator, $\mathcal H_{A, m}=(\mathscr D(A), \inp{\cdot}{\cdot}_{A, m})$ is a Hilbert space. Further, by Kato's second representation theorem \cite[Theorem 2.23, Chapter VI]{K},
\beq
\label{KRT}
\left.
\begin{array}{lll}
\mathscr D(A) &=& \mathscr D((A^*A+m^2I)^{1/2}), \\
\inp{x}{y}_{A, m} &=& \inp{(A^*A+m^2I)^{1/2}x}{(A^*A+m^2I)^{1/2}y}, \quad x, y \in \mathscr D(A).
\end{array}
\right\}
\eeq
By Proposition \ref{lem-H-transform},  $\mathscr D(A) \subseteq \mathscr D(B).$
Moreover, if $h \in \mathscr D(A)$, then by \eqref{KRT}, we have $$\|h\|^2_{A, m}=\sum_{v \in \supp\,\Hi{b}{u}}(|\lambda_{uv}|^2+m^2)|h(v)|^2 < \infty,$$ and hence by \eqref{g-lambda-1} and \eqref{Gm-lt-G1},
\beq
\label{estimate-m}
\Big|\sum_{v \in \supp\, \Hi{b}{u}} \!\!\!\!\!\!\!\! h(v)\overline{g(v)}\Big|^2 \,\leqslant\,\|h\|^2_{A, m} G_m \,\leqslant\, G_1 \|h\|^2_{A, m}
\eeq
(see \eqref{Gm}). Since $C$ is a bounded linear operator, $\mathscr D(A)
\subseteq \mathscr D(B+C)$.
We claim that for any (arbitrarily small) $a >0$, there exists (large enough) $b$ such that
\beq
\label{kbs-estimate-1}
\|(B+C)h\|^2 \,\leqslant\,a \|Ah\|^2 + b\|h\|^2, \quad h \in \mathscr D(A).
\eeq
To see the claim, let $h=(h_1, h_2) \in \mathscr D(A)$. By repeated applications of $|\alpha+\beta|^2 \leqslant2(|\alpha|^2+|\beta|^2)$, $\alpha, \beta \in \mathbb C$, we obtain
\beqn
\|(B+C)h\|^2 & \leqslant & 2\|(f \obslash g)h_1\|^2 + 2\|C\|^2\|h_2\|^2 \notag \\
&  \overset{\eqref{estimate-m}} \leqslant & 2 \|f\|^2\|h\|^2_{A, m} G_m  + 2\|C\|^2\|h_2\|^2 \notag \\
& \leqslant & 2\|f\|^2\|Ah\|^2G_m + (2m^2 \|f\|^2G_m + 2\|C\|^2)\|h\|^2.
\eeqn
However, an application of Lebesgue dominated convergence theorem together with \eqref{Gm-lt-G1} shows that $G_m \rar 0$ as $m \rar \infty.$ This completes the proof of the claim. Another application of Kato-Rellich Theorem shows that $W_{f, g}$ is closed.

To prove that $D_V$ forms a core for $W_{f, g}$, it suffices to check that $W_{f, g} \subseteq \overline{W_{f, g}|_{\mathscr D_V}}.$ To see that, let $(h, k) \in \mathscr D(W_{f, g}).$ By \eqref{domain-W-fg}, $h \in \mathscr D(\Db{b}{u})$ and $k \in \ell^2(V) \ominus \Hi{b}{u}$. Since $\mathscr D_{\supp\,\Hi{b}{u}}$ is a core for $\Db{b}{u}$, there exists a sequence $\{h_n\}_{n \in \mathbb N} \subseteq \mathscr D_{\supp\,\Hi{b}{u}}$ such that $h_n \rar h$ and $\Db{b}{u}h_n \rar \Db{b}{u}h$ as $n \rar \infty.$ Let $\{k_n\}_{n \in \mathbb N}$ be any sequence in $\mathscr D_{V \setminus \supp\,\Hi{b}{u}}$ converging to $k.$ Note that by Proposition \ref{lem-H-transform}, $h \in \mathscr D(f \obslash g)$. Further, since $g_{_{\lambdab_u, \mu}} \in \Hi{b}{u}$ and $(\Db{b}{u}-\mu)h_n \rar (\Db{b}{u}-\mu)h$ as $n \rar \infty,$ by the Cauchy-Schwarz inequality,
$f \obslash g(h_n) \rar (f \obslash g)(h)$ as $n \rar \infty$.
It is now easy to see that
\beqn
(h_n, k_n) \rar (h, k), \quad W_{f, g}(h_n, k_n) \rar W_{f, g}(h, k)~\mbox{as~}n\rar \infty.
\eeqn
Thus $(h, k) \in \mathscr D(\overline{W_{f, g}|_{\mathscr D_V}})$ and $W_{f, g}(h, k)= \overline{W_{f, g}|_{\mathscr D_V}}\,(h, k)$, as desired.
\end{proof}

The following is immediate from the proof of Theorem \ref{closed-0}.
\begin{corollary}
\label{bdd-rel}
Let $\mathscr T=(V, E)$ be a rooted directed tree with root $\rootb$ and let $\mathscr T_{\infty}=(\V, E_{\infty})$ be the extended directed tree associated with $\mathscr T$.
For $u, \mf b \in V,$  consider the weight system
$\lambdab_u = \{\lambda_{uv}\}_{v\in \V}$ of complex numbers and let $W_{f, g}$ be the rank one extension of the weighted join operator $\W{b}{u}$ on $\mathscr T$, where $f \in \ell^2(V) \ominus \Hi{b}{u}$ is non-zero and $g : \supp\,\Hi{b}{u} \rar \mathbb C$ is given.
Suppose that $W_{f, g}$ satisfies the compatibility condition II. Then
$W_{f, g}$ decomposes as $A + B + C,$ where $A, B, C$ are densely defined operators given by
\eqref{ABC} such that $B+C$ is $A$-bounded with $A$-bound equal to $0.$
\end{corollary}

\section{Adjoints and Gelfand-triplets}

In this section, we try to unravel the structure of the Hilbert
space adjoint of the rank one extension $W_{f, g}$ of weighted join
operators. In particular, we discuss the question of determining the action of the Hilbert space adjoint of $W_{f, g}$. For an interesting discussion on the relationship between
Hilbert space adjoint and formal adjoint of an unbounded operator
matrix, the reader is referred to \cite{MS}. Unfortunately, the
situation in our context suggests that there is no obvious way in
which one can identify the Hilbert space adjoint of $W_{f, g}$ with
its {\it formal adjoint} (that is, the transpose of the matrix
formed after taking entry-wise adjoint) unless $g \in \Hi{b}{u}.$ We
conclude this section with a brief discussion on the role of a
Gelfand-triplet naturally associated with $W_{f, g}$ in the
realization of $W^*_{f, g}$.

We begin with the following proposition which shows that the adjoint
of $W_{f, g}$ coincides with the adjoint of the associated weighted
join operator $\W{b}{u}$ on a possibly non-dense subspace of
$\mathscr D(W^*_{f, g})$.
\begin{proposition}
Let $\mathscr T=(V, E)$ be a rooted directed tree with root $\rootb$ and let $\mathscr T_{\infty}=(\V, E_{\infty})$ be the extended directed tree associated with $\mathscr T$.
For $u, \mf b \in V,$ consider the weight system
$\lambdab_u = \{\lambda_{uv}\}_{v\in \V}$ of complex numbers and let $W_{f, g}$ be the rank one extension  of the weighted join operator $\W{b}{u}$ on $\mathscr T$, where $f \in \ell^2(V) \ominus \Hi{b}{u}$ is non-zero and $g : \supp\,\Hi{b}{u} \rar \mathbb C$ is given.
Then $W^*_{f, g}$ is a closed operator.  Moreover,
\beqn
&& \mathscr D \,:=\,\{k =(k_1, k_2) \in \mathscr D(\Db{b}{u}) \oplus (\ell^2(V) \ominus \Hi{b}{u}) : \inp{k_2}{f}=0\} \subseteq \mathscr D(W^*_{f, g}), \\
&& W^*_{f, g}k \,=\, (\Db{b}{u})^*k_1 + (\N{b}{u})^*k_2, \quad k=(k_1, k_2) \in \mathscr D.
\eeqn
Further,  if $W_{f, g}$ satisfies a compatibility condition, then $W^*_{f, g}$ is densely defined.
\end{proposition}
\begin{proof} Since $W_{f, g}$ is densely defined in $\ell^2(V),$ the Hilbert space adjoint  $W^*_{f, g}$ is a closed linear operator \cite{Sc}.
Let $k =(k_1, k_2) \in \mathscr D$. Then, for any $h=(h_1, h_2) \in \mathscr D(W_{f, g})$,
\beqn
\inp{W_{f, g}h}{k} &=& \inp{\Db{b}{u}h_1}{k_1} + \inp{(f \obslash g)h_1}{k_2} + \inp{\N{b}{u}h_2}{k_2} \\
&=& \inp{h_1}{(\Db{b}{u})^*k_1} +  \inp{h_2}{(\N{b}{u})^*k_2},
\eeqn
where we used the fact that $\mathscr D((\Db{b}{u})^*) =\mathscr D(\Db{b}{u}).$
It follows that $k \in \mathscr D(W^*_{f, g})$ and $W^*_{f, g}k=(\Db{b}{u})^*k_1 + (\N{b}{u})^*k_2.$ Finally, if $W_{f, g}$ satisfies a compatibility condition, then by Theorem \ref{closed-0}, $W_{f, g}$ is closed, and hence, by \cite[Theorem 1.8(i)]{Sc}, $W^*_{f, g}$ is densely defined. This completes the proof.
\end{proof}
\begin{remark}
Note that for any $k \in \ell^2(V) \ominus \Hi{b}{u}$, $(0, k) \notin \mathscr D(W^*_{f, g})$. Otherwise,
\beqn
\phi(h)=\inp{W_{f, g}h}{(0, k)}= \inp{(f \obslash g)h_1}{k} + \inp{\N{b}{u}h_2}{k}
, \quad h=(h_1, h_2) \in \mathscr D(W_{f, g})
\eeqn
extends as a bounded linear functional, which is not possible since $g \notin \Hi{b}{u}$ (see Lemma \ref{lem-inj-ten-unb}). In particular, $(0, f) \notin \mathscr D(W^*_{f, g})$.
\end{remark}
Note that the closability of $W_{f, g}$ is equivalent to the density
of the domain $\mathscr D(W^*_{f, g})$ of $W^*_{f, g}$ (see
\cite[Theorem 1.8]{Sc}). In particular, it would be interesting to
obtain conditions ensuring the density of  $\mathscr D(W^*_{f, g})$ (with or without a
compatibility condition). In view of the decomposition $W_{f,
g}=A+B+C$ as given in \eqref{ABC}, it is tempting to ask whether
$W^*_{f, g}$ can be decomposed as $A^*+B^*+C^*.$ It may be concluded
from \cite[Proposition]{Be} that if $W_{f, g}$ is Fredholm such that
$B$ is $A$-compact and $B^*$ is $A^*$-compact, then $W^*_{f, g}=A^*+
B^* +C^*$ (see \cite[Theorem 2.2]{Mo} for a variant). We will see in
the proof of Theorem \ref{spectrum}(iv) that under the assumption of
compatibility condition I, $A$-compactness of $B$ can be ensured
(Recall that $B$ is {\it $A$-compact} if $\mathscr D(A) \subseteq
\mathscr D(B)$ and $B$ maps $\{h \in D(A) : \|h\|+ \|Ah\| \leqslant
1\}$ into a pre-compact set).

Another natural problem which arises in finding $W^*_{f, g}$ is
whether it is possible to have a matrix decomposition of $W^*_{f,
g}$ similar to the one we have it for $W_{f, g}$. One possible
candidate for $W^*_{f, g}$ is its formal adjoint, that is, the
transpose of the operator matrix $W^{\times}_{f, g}$ obtained by
taking Hilbert space adjoint
 of each entry of $W_{f, g}$.
A direct application of  \cite[Theorem 6.1]{MS} shows the following:
\begin{enumerate}
\item[(i)] $W^{\times}_{f, g}$ is a closable operator such that $W^{\times}_{f, g} \subseteq W^{*}_{f, g}.$
\item[(ii)] If $W^{\times}_{f, g}$ is densely defined, then $W_{f, g}
$ is closable.
\end{enumerate}
It is evident that there is no natural way to recover $W^{*}_{f, g}$ from
$W^{\times}_{f, g}$. One such way has been shown in
\cite[Proposition 6.3]{MS}, which provides a sufficient condition
for the equality Hilbert space adjoint and formal adjoint of an
unbounded operator matrix. Unfortunately, this result is not
applicable to $W_{f, g}$ unless all its entries are closable
operators. Recall that
 $f \obslash g$ is not even closable in case $g \notin \Hi{b}{u}$ (see Lemma \ref{lem-inj-ten-unb}).
 That's why to understand the action of $W^*_{f, g}$, we need to replace $\Hi{b}{u}$ by a larger Hilbert space. We will see below that the notion of Hilbert rigging turns out to be handy in this context.

\index{$\mathcal H_\circ$}
\index{$\mathcal H^\circ$}
\index{$(\mathcal H_\circ, \Hi{b}{u}, \mathcal H^\circ)$}

Consider the inner-product space $\mathscr D(\Db{b}{u})$ endowed with the inner-product \beqn \inp{x}{y}_{\circ} \,:=\, \inp{\Db{b}{u}x}{\Db{b}{u}y} + \inp{x}{y}, \quad x, y \in \mathscr D(\Db{b}{u}).\eeqn
Note that $\mathcal H_{\circ}:=\mathscr D(\Db{b}{u})$ endowed with the inner-product $\inp{\cdot}{\cdot}_{\circ}$ is a Hilbert space. Clearly, the inclusion map $i : \mathcal H_{\circ} \hookrightarrow \Hi{b}{u}$ is contractive.  Consider further the topological dual $\mathcal H^*_{\circ}$ of $\mathcal H_{\circ}$, which we denote by $\mathcal H^{\circ}$.
We claim that any element $h \in \Hi{b}{u}$ can be realized as a bounded conjugate-linear functional in $\mathcal H^\circ$. To see this, consider the mapping $j : \Hi{b}{u} \rar \mathcal H^\circ$ given by $j(h) = \phi_h,$ where
\beqn
\phi_h(k)~=\sum_{v \in \supp\,\Hi{b}{u}}\overline{k(v)}{h(v)}, \quad k \in \mathcal H_\circ.
\eeqn
The contractivity of $j$ follows from
the Cauchy-Schwarz inequality and
\beq \label{quasi-n}
\|\phi_h\| &=& \notag \sup_{\|k\|_\circ =1}\Big|\sum_{v \in \supp\,\Hi{b}{u}}\overline{k(v)}h(v)\Big| \\ &=& \notag \Big(\sum_{v \in \supp\,\Hi{b}{u}}\frac{|h(v)|^2}{1+|\lambda_{uv}|^2}\Big)^{1/2} \\&=& \|((\Db{b}{u})^*\Db{b}{u}+I)^{-1/2}h\|,
\eeq
and hence the claim stands verified.
 This also shows that $\mathcal H^\circ$ can be identified with the completion of $\Hi{b}{u}$ endowed with the inner-product $$\inp{x}{y}^\circ \,=\,\inp{((\Db{b}{u})^*\Db{b}{u}+I)^{-1}x}{y}, \quad x, y \in \Hi{b}{u}.$$
Thus we have the following chain of Hilbert spaces: 
\beqn \mathcal H_\circ ~\subsetneq ~\Hi{b}{u} ~\subsetneq ~\mathcal H^\circ, \eeqn  where $\mathcal H_\circ$ is dense in $\Hi{b}{u}$ and $\Hi{b}{u}$ is dense in  $\mathcal H^\circ$. One may refer to this chain of Hilbert spaces as the {\it Hilbert rigging} of $\Hi{b}{u}$ by $\mathcal H_\circ$ and $\mathcal H^\circ$. The triplet $(\mathcal H_\circ, \Hi{b}{u}, \mathcal H^\circ)$ is known as the {\it Gelfand-triplet}  (refer to \cite[Chapter 14]{BSU} for an abstract theory of rigged spaces). If
$\{(|\lambda_{uv}|^2+1)^{-1/2} : v \in \supp\,\Hi{b}{u}\}$ is square-summable,
then by \eqref{quasi-n}, the above Hilbert rigging is {\it quasi-nuclear} in the sense that the inclusion
$j : \Hi{b}{u} \rar \mathcal H^\circ$ is Hilbert-Schmidt (see \cite[Pg 121]{BSU}).

Suppose  that $W_{f, g}$ satisfies the compatibility condition II. Define $\phi_g : \mathcal H_\circ \rar \mathbb C$ by
\beqn
\phi_g(k)~=\sum_{v \in \supp\,\Hi{b}{u}}\overline{k(v)}{g(v)}, \quad k \in \mathcal H_\circ.
\eeqn
It may be concluded from \eqref{g-lambda-1} that $\phi_g \in \mathcal H^\circ.$
This allows us to introduce the bounded linear transformation $B : \mathcal H_\circ \rar \ell^2(V) \ominus \Hi{b}{u}$ by setting
\beqn
Bk \,=\, \overline{\phi_g(k)}\,f, \quad k \in \mathcal H_\circ,
\eeqn
Note that for any $l \in \ell^2(V) \ominus \Hi{b}{u}$ and $k \in \mathcal H_\circ$,
\beqn
(B^*l)(k)=\inp{l}{Bk}={\phi_g(k)}\,\inp{l}{f} ={(\phi_g \otimes f)(l)}(k).
\eeqn
Thus $B^*$ can be identified with $g \otimes f.$ In particular, the Hilbert space adjoint of $W_{f, g}$ can be identified with the formal adjoint of $W_{f, g}$ after replacing the Hilbert space $\Hi{b}{u}$ by the larger Hilbert space $\mathcal H^{\circ}$.

\section{Spectral analysis}

We now turn our attention to the spectral properties of rank one extensions of weighted join operators.
The main result of this section provides a complete spectral picture for rank one extensions $W_{f, g}$ of weighted join operators
(cf. \cite[Theorem 1]{St-Sz-3}, \cite[Theorem 2.3]{I}, \cite[Theorem 2.3]{JKP}, \cite[Corollary 2.8]{N}, \cite[Theorem 2.2]{ALMS}). It turns out that $W_{f, g}$ has nonempty resolvent set if and only if it satisfies the compatibility condition I. Among various applications, 
we characterize rank one extensions of weighted join operators on leafless directed trees which admit compact resolvent.

\begin{theorem}[Spectral picture] \label{spectrum}
Let $\mathscr T=(V, E)$ be a rooted directed tree with root $\rootb$ and let $\mathscr T_{\infty}=(\V, E_{\infty})$ be the extended directed tree associated with $\mathscr T$.
For $u, \mf b \in V,$ consider the weight system
$\lambdab_u = \{\lambda_{uv}\}_{v\in \V}$ of complex numbers and let $W_{f, g}$ be the rank one extension  of the weighted join operator $\W{b}{u}$ on $\mathscr T$, where $f \in \ell^2(V) \ominus \Hi{b}{u}$ is non-zero and $g : \supp\, \Hi{b}{u} \rar \mathbb C$ is given $($see \eqref{rone-extn}$)$.
Then, we have the following statements:
\begin{enumerate}
\item The point spectrum $\sigma_p(W_{f, g})$ of $W_{f, g}$ is given by
\beqn
\sigma_p(W_{f, g}) \,=\, \begin{cases} 
\{\lambda_{uv} : v \in V\} & \mbox{if~} \mf b =u, \\
\{\lambda_{uv} : v \in \asc{u} \cup \Desb{u}\} \cup \{0\} & \mbox{if~} \mf b \in \mathsf{Des}_u(u), \\
\{\lambda_{uv} : v \in \des{u}\} \cup \{0\} & \mbox{otherwise}.
\end{cases}
\eeqn
\item The spectrum $\sigma(W_{f, g})$ of $W_{f, g}$ is given by
\beqn
\sigma(W_{f, g}) \,=\, \begin{cases} \overline{\sigma_p(W_{f, g})} & \mbox{if~}  W_{f, g}~\mbox{satisfies the compatibility condition}~I,\\
\mathbb C & \mbox{otherwise}.
\end{cases}
\eeqn
\end{enumerate}
If, in addition, $W_{f, g}$ satisfies the compatibility condition I, then we have the following:
\begin{enumerate}
\item[(iii)]  For every $\mu \in \mathbb C \setminus \overline{\sigma_p(W_{f, g})}$,  the resolvent of $W_{f, g}$ at $\mu$ is given by
\beq \label{cpt-resolvent-eq}
\begin{array}{lll}
(W_{f, g}-\mu)^{-1} \,=\, \left[\begin{array}{cc}
(\Db{b}{u}-\mu)^{-1} & 0 \\
-(\N{b}{u}-\mu)^{-1}L_{\lambdab_u, \mu} & (\N{b}{u}-\mu)^{-1}
\end{array}
\right],
\end{array}
\eeq
where the linear transformation $L_{\lambdab_u, \mu}:=(f \obslash g)(\Db{b}{u}-\mu)^{-1}$ defines a Hilbert-Schmidt integral operator from $\Hi{b}{u}$ into $\ell^2(V) \ominus \Hi{b}{u}$.
\item[(iv)]  The essential spectrum $\sigma_e(W_{f, g})$ of $W_{f, g}$ is given by
\beqn
\sigma_e(W_{f, g}) \,=\, \begin{cases}
 \sigma_e(\Db{b}{u}) & \mbox{if~} \dim\big(\ell^2(V) \ominus \Hi{b}{u}\big) < \infty, \\[5pt]
\sigma_e(\Db{b}{u}) \cup \{0\} & \mbox{otherwise}.
\end{cases}
\eeqn
Moreover, $\mbox{ind}_{\,W_{f, g}}=0$ on $\mathbb C \setminus \sigma_e(W_{f, g}).$
\end{enumerate}
\end{theorem}
\begin{remark} By \cite[Theorem 7.1]{Kl}, for any diagonal operator $D$ with simple point spectrum and perfect spectrum, there exists a bounded rank one perturbation $h \otimes k$ of arbitrarily small positive norm such that $D+h\otimes k$ has no point spectrum. This situation does not appear in the context of rank one extensions of weighted join operators, where the point spectrum is always non-empty.
 \end{remark}
 
\index{$L_{\lambdab_u, \mu}$}

\begin{proof}
Let $\mu$ be a complex number and $(h, k) \in
\mathscr D(W_{f, g})$ be a non-zero vector such that $W_{f, g}(h, k)=\mu (h, k).$
By \eqref{rone-extn},
 $h \in \mathscr D(\Db{b}{u}) \cap \mathscr D(f \obslash g),$ $k \in \ell^2(V) \ominus \Hi{b}{u}$ and
\beq
\label{spec-rone}
\Db{b}{u}h=\mu\, h, \quad \Big(\sum_{v \in \supp\, \Hi{b}{u}} \!\!\!\!\!\!\!\! h(v)\overline{g(v)}\Big) f  +
\inp{k}{e_{_{\lambdab_u, A_u}}}\,e_u=\mu\, k.
\eeq

\begin{case}
$h=0:$
\end{case}
In this case, $\inp{k}{e_{_{\lambdab_u, A_u}}}\,e_u=\mu\, k$. Accordingly, any one of the following possibilities occur:
\begin{enumerate}
\item[(1)] $e_u$ (considered as the vector $(0, e_u)$) is an eigenvector of $W_{f, g}$ corresponding to the eigenvalue $\mu=\lambda_{uu}$.
\item[(2)] $k$  is an eigenvector of $W_{f, g}$ corresponding to the eigenvalue $\mu=0$ provided $\mf b \neq u$, where $k \in \big(\ell^2(V) \ominus \Hi{b}{u}\big)\ominus [e_{_{\lambdab_u, A_u}}]$.
\end{enumerate}
Here, in the second assertion, we used the facts that 
$\dim \ell^2(A_u) \geqslant 2$ (since $\mf b \neq u$) and
\beqn \dim\big(\ell^2(V) \ominus \Hi{b}{u}\big) \,\geqslant\,\dim \ell^2(A_u) \eeqn
(see \eqref{A-u}).
\begin{case}
$h \neq 0:$
\end{case}
In this case, $\mu \in \sigma_p(\Db{b}{u})$, and hence $\mu =\lambda_{uw}$ for some $w \in \supp\, \Hi{b}{u}$ and
\beqn
h \in \mathscr E_{\Db{b}{u}}(\mu)=\ell^2(W_w),
\eeqn
where $W_w$ is as given in \eqref{W-u}.
It follows from \eqref{spec-rone} that
\beq
\label{spec-rone-k}
\Big(\sum_{v \in W_w} \!\! h(v)\overline{g(v)}\Big) f  +
\inp{k}{e_{_{\lambdab_u, A_u}}}\,e_u=\lambda_{uw}\, k.
\eeq
Taking inner-product with $e_{_{\lambdab_u, A_u}}$ on both sides, we get
\beq \label{spec-rone-kk}
\Big(\sum_{v \in W_w} \!\! h(v)\overline{g(v)}\Big) \inp{f}{e_{_{\lambdab_u, A_u}}}  \,=\, (\lambda_{uw}-\lambda_{uu}) \inp{k}{e_{_{\lambdab_u, A_u}}}.
\eeq
Accordingly, any one of the following possibilities occur :
\begin{enumerate}
\item[(1)]  $\lambda_{uw}$ is a non-zero number equal to $\lambda_{uu}:$ In this case,   $$\inp{f}{e_{_{\lambdab_u, A_u}}}=0 ~\mbox{or}~\sum_{v \in W_w} \!\! h(v)\overline{g(v)}=0.$$ 
Thus $k$ belongs either to $[f]$ or to $[e_u].$
\item[(2)] $\lambda_{uw}$ is a non-zero number not equal to $\lambda_{uu}:$ By \eqref{spec-rone-k} and \eqref{spec-rone-kk}, $k$ takes the form
\beq \label{k-unique} k\,=\,\frac{\sum_{v \in W_w} \!\! h(v)\overline{g(v)}}{\lambda_{uw}}
\Big( f  +
\frac{\inp{f}{e_{_{\lambdab_u, A_u}}}}{\lambda_{uw}-\lambda_{uu}} \,e_u\Big).
\eeq
In this case, $k$ belongs to the span of $f  +
\frac{\inp{f}{e_{_{\lambdab_u, A_u}}}}{\lambda_{uw}-\lambda_{uu}} \,e_u.$
\item[(3)] $\lambda_{uw}=0$: In this case, any non-zero vector $(h, k)$ with $h \in \ell^2(W_w),$ $k \in \ell^2(V) \ominus \Hi{b}{u}$ satisfying the following identity will be an eigenvector of $W_{f, g}$ corresponding to the eigenvalue $0$:
\beqn
\Big(\sum_{v \in W_w} \!\! h(v)\overline{g(v)}\Big) f  +
\inp{k}{e_{_{\lambdab_u, A_u}}}\,e_u=0.
\eeqn
\end{enumerate}
In particular, the cases above show that \beqn \sigma_p(W_{f, g}) =\begin{cases} \sigma_p(\Db{b}{u}) \cup \{\lambda_{uu}\} & \mbox{if~} \mf b =u, \\
 \sigma_p(\Db{b}{u}) \cup \{\lambda_{uu}, 0\} & \mbox{otherwise}.
 \end{cases}
\eeqn
The conclusion in (i) is now clear from the fact that $\sigma_p(\Db{b}{u})=\big\{\lambda_{uv} : v \in \supp\,\Hi{b}{u}\big\}$, \eqref{H-u} and \eqref{H-u-1}.



To see (ii), let $\mu \in \mathbb C \setminus \overline{\sigma_p(W_{f, g})},$ that is, $\mu$ is a non-zero number such that $\mu \neq \lambda_{uu}$ and
\beq \label{D-inverse} \dist(\lambdab_u, \mu)~=\inf_{v \in \supp\,\Hi{b}{u}} | \lambda_{uv} - \mu| > 0. \eeq
By \eqref{rone-extn}, $(k_1, k_2) \in \mbox{ran}(W_{f, g}-\mu)$ if and only if there exists $(h_1, h_2) \in \mathscr D(W_{f, g})$ such that
\beq \label{surjectivity-W-fg}
(\Db{b}{u}-\mu)h_1=k_1, \quad \Big(\!\!\!\!\!\!\!\!\sum_{v \in \supp\, \Hi{b}{u}} \!\!\!\!\!\!\!\! h_1(v)\overline{g(v)}\Big) f  +
\inp{h_2}{e_{_{\lambdab_u, A_u}}}\,e_u - \mu h_2=k_2.
\eeq
We claim that  
\beq \label{claim-surj}
\mbox{$W_{f, g}-\mu$ is surjective if and only if $
\mathscr D(\Db{b}{u}) \subseteq \mathscr D(f \obslash g).$}
\eeq
To see the claim, suppose that $W_{f, g}-\mu$ is surjective and let $h'_1 \in \mathscr D(\Db{b}{u}).$ Letting $k_1=(\Db{b}{u}-\mu)h'_1$ and $k_2=0$, by surjectivity of $W_{f, g}-\mu$, we get $(h_1, h_2) \in \mathscr D(W_{f, g})$ such that \eqref{surjectivity-W-fg} holds. However, since $\Db{b}{u}-\mu$ is injective, $h'_1=h_1$, and hence $h'_1 \in \mathscr D(f \obslash g).$ To see the reverse implication, assume that $
\mathscr D(\Db{b}{u}) \subseteq \mathscr D(f \obslash g),$ and let $k_1 \in \Hi{b}{u}$ and $k_2 \in \ell^2(V) \ominus \Hi{b}{u}$. By \eqref{D-inverse}, $\Db{b}{u}-\mu$ is invertible, and hence there exists $h_1 \in \mathscr D(\Db{b}{u})$ such that $(\Db{b}{u}-\mu)h_1=k_1.$ By assumption, $h_1 \in \mathscr D(f \obslash g).$ 
Since $\mu \neq \lambda_{uu}$,  the following equation can be uniquely solved for $\inp{h_2}{e_{_{\lambdab_u, A_u}}}$:
\beqn
\inp{h_2}{e_{_{\lambdab_u, A_u}}}(\lambda_{uu} - \mu) \,=\,\inp{k_2}{e_{_{\lambdab_u, A_u}}} - \Big(\!\!\!\!\!\!\!\!\sum_{v \in \supp\, \Hi{b}{u}} \!\!\!\!\!\!\!\! h_1(v)\overline{g(v)}\Big) \inp{f}{e_{_{\lambdab_u, A_u}}}.
\eeqn
Since $\mu \neq 0,$ substituting the above value of $\inp{h_2}{e_{_{\lambdab_u, A_u}}}$ in \eqref{surjectivity-W-fg} determines $h_2 \in \ell^2(V) \ominus \Hi{b}{u}$ uniquely.
This completes the verification of \eqref{claim-surj}. The first part in (ii) now follows from Proposition \ref{lem-H-transform}.

To see the remaining part in (ii), suppose that $W_{f, g}$ does not satisfy the compatibility condition I. Let  $\varGamma_{\lambdab_u}$ be as given in \eqref{varG-0}.
Thus there are two possibilities:

\begin{cas3}
$\varGamma_{\lambdab_u}=\emptyset:$
\end{cas3}
In this case, $\sigma_p(\Db{b}{u})$ is necessarily dense in $\mathbb C$, and hence by (i), $\sigma_p(W_{f, g})$ is also dense in $\mathbb C.$
If possible, then assume that $\sigma(W_{f, g})$ is not equal to $\mathbb C.$ Then, by \cite[Lemma 1.17]{CM}, $W_{f, g}$ is closed. However,
the spectrum of a closed operator is always closed (see \cite[Proposition 2.6]{Sc}). This is not possible since $\sigma_p(W_{f, g})$ is dense and proper subset of $\mathbb C$, and hence we must have $\sigma(W_{f, g})=\mathbb C.$

\begin{cas3}
$\varGamma_{\lambdab_u} \neq \emptyset:$
\end{cas3}
If possible, then suppose that $\sigma(W_{f, g}) \subsetneq \mathbb C.$ Thus there exists $\mu \in \mathbb C \setminus \sigma(W_{f, g})$ and  a linear operator $R(\mu) \in \mathcal B(\ell^2(V))$ such that
\beq \label{resol-eq} (W_{f, g}-\mu) R(\mu)h\,=\,h, \quad h \in \ell^2(V).\eeq
Consider the following decomposition of $R(\mu)$:
\beqn
\begin{array}{lll}
R(\mu) = \left[\begin{array}{cc}
A(\mu) & B(\mu) \\
C(\mu) & D(\mu)
\end{array}
\right] ~ \mbox{on}~ \ell^2(V)=\Hi{b}{u} \oplus (\ell^2(V) \ominus \Hi{b}{u}).
\end{array}
\eeqn
Note that for any $k \in \ell^2(V) \ominus \Hi{b}{u}$,
\beqn
\left[\begin{array}{c}
0 \\
k
\end{array}
\right] & \overset{\eqref{resol-eq}}= &(W_{f, g}-\mu)R(\mu)\left[\begin{array}{c}
0 \\
k
\end{array}
\right] \\ &=& \left[\begin{array}{c}
(\Db{b}{u}-\mu)B(\mu)k \\
f\obslash g (B(\mu)k) + (\N{b}{u}-\mu)(D(\mu)k)
\end{array}
\right],
\eeqn
which yields $(\Db{b}{u}-\mu)B(\mu)k=0$. However, $\mu \notin \sigma(\Db{b}{u})$, and consequently, $B(\mu)=0.$
Also, since for any $h \in \Hi{b}{u}$, \beqn R(\mu)\left[\begin{array}{c}
h \\
0
\end{array}
\right] \in \mathscr D(W_{f, g}-\mu)=\mathscr D(W_{f, g}),
\eeqn
by the definition of domain of $W_{f, g}$, we must have \beq \label{A(mu)}
A(\mu)h \in \mathscr D(f \obslash g),  \quad (\Db{b}{u}-\mu)A(\mu)=I \eeq (see \eqref{resol-eq}). It follows that $A(\mu)=(\Db{b}{u}-\mu)^{-1}$ and hence any arbitrary vector in $\mathscr D(\Db{b}{u})$ is of the form $A(\mu)h$ for some $h \in \Hi{b}{u}$. This together with \eqref{A(mu)} yields the inclusion
$\mathscr D(\Db{b}{u}) \subseteq \mathscr D(f \obslash g).$ An application of Proposition \ref{lem-H-transform}, however, shows that $W_{f, g}$ satisfies the compatibility condition I, which is contrary to our assumption. This completes the proof of (ii).

To see (iii) and (iv), assume that $W_{f, g}$ satisfies the compatibility condition I. Let $\mu \in \mathbb C \setminus \overline{\sigma_p(W_{f, g})}$.
By Proposition \ref{lem-H-transform}, $L_{\lambdab_u, \mu}$ is bounded.
 A routine verification shows that $L_{\lambdab_u, \mu}$
is an integral  operator with kernel $K_{\lambdab_u, \mu}$ given by
\beqn K_{\lambdab_u, \mu}(w,v) \,:=\,
\frac{\overline{g(v)}f(w)}{\lambda_{uv}-\mu}, \quad w \in V \setminus \supp\,\Hi{b}{u}, ~v \in\supp\,\Hi{b}{u}.\eeqn
By the compatibility condition I and the assumption that $f \in \ell^2(V) \ominus \Hi{b}{u}$, $K_{\lambdab_u, \mu}$ belongs to
$\ell^2((V \setminus \supp\,\Hi{b}{u}) \times \supp\,\Hi{b}{u})$. By \cite[Theorem 3.8.5]{Si1},  $L_{\lambdab_u, \mu}$ is a Hilbert-Schmidt operator. We leave it to the reader to verify that the expression given by \eqref{cpt-resolvent-eq} defines the resolvent of $W_{f, g}$ at $\mu$.

To see (iv), let $A, B, C$ be as given by \eqref{ABC} and note that $A+B+C=W_{f, g}$. Since $A=\Db{b}{u} \oplus 0$ on $\Hi{b}{u} \oplus \big(\ell^2(V) \ominus \Hi{b}{u}\big)$, it suffices to check that $\sigma_e(W_{f, g})=\sigma_e(A)$ and $\mbox{ind}_{\,W_{f, g}}=\mbox{ind}_{A}.$
In view of \cite[Theorems 5.26 and 5.35, Chapter IV]{K},
it is sufficient to verify that
$B + C$ is $A$-compact (see also the foot note 1 on \cite[Pg 244]{K}).
Let $\{h_n\}$ be a bounded sequence in $\mathscr D(A) \subseteq \mathscr D(B)$ such that $\{Ah_n\}$ is bounded. Since $C$ is a finite rank operator, it suffices to check that $\{Bh_n\}$ has a convergent subsequence. By part (iii), $L_{\lambdab_u, \mu}$ is a compact operator, and hence so is $B(A-\mu)^{-1}$.  However, $\{(A-\mu)h_n\}$ is bounded, and hence $\{Bh_n\}$ admits a convergent subsequence.
The remaining part follows from the fact that the index function for a diagonal operator is identically $0$.
 \end{proof}
 \begin{remark}
 In this remark, we describe eigenspaces of $W_{f, g}$ (under the hypotheses of Theorem \ref{spectrum}). To see that, we need some notations. Given a subspace $\ell^2(W)$ of $\ell^2(V)$ and a  rank one operator $h \obslash k$ from $\ell^2(V)$ into $\ell^2(U)$, we introduce the linear transformation $h \obslash k|_{\ell^2(W)}$ from $\ell^2(W)$ into $\ell^2(U)$ as follows:
\beqn
\mathscr D(h \obslash k|_{\ell^2(W)}) &=& \mathscr D(h \obslash k) \cap \ell^2(W), \\ h \obslash k|_{\ell^2(W)}(l) &=& h \obslash k(l), \quad l \in \mathscr D(h \obslash k|_{\ell^2(W)}).
\eeqn
For $\mu \in \mathbb C,$ let $W_\mu$ be given by
\index{$W_{\mu}$}
\index{$W_{\lambda_{uv}}=W_v$}
\index{${\tt graph}(T)$}
\index{$h \obslash k\vert_{\ell^2(W)}$}
\beq
 \label{W-v} W_\mu &=& \{w \in  \supp\, \Hi{b}{u} : \lambda_{uw} = \mu\}. \eeq
In case $\mu =\lambda_{uv}$ for some $v \in V,$ we denote $W_{\mu}$ by the simpler notation $W_v.$
Further, we reserve the notation ${\tt graph}(T)$ for the graph of a linear operator $T$ in $\mathcal H.$ 
If $\mathscr E_{W_{f, g}}(\mu)$ denotes the eigenspace corresponding to the  eigenvalue $\mu$ of $W_{f, g}$, then we have the following statements:
\begin{enumerate}
\item[(a)]  If $\lambda_{uv} \neq 0$ for every $v \in \supp\, \Hi{b}{u},$ then
\beqn \mathscr E_{W_{f, g}}(0)\,=\, \{0\} \oplus \ker(\N{b}{u}).
\eeqn
\item[(b)] If $\lambda_{uv} = 0$ for some $v \in \supp\, \Hi{b}{u},$ then
\beqn
 \mathscr E_{W_{f, g}}(0) \,=\, \begin{cases}
 \ker(f \obslash \tilde{g})
  & \mbox{if~}f \in [e_u],
\\
 \ker({f \obslash g}|_{_{\ell^2(W_v)}}) \oplus \ker(\N{b}{u})
& \mbox{otherwise},
\end{cases}
\eeqn
where $\tilde{g} : W_v \cup \big(V \setminus \supp\, \Hi{b}{u}\big) \rar \mathbb C$ is given by
\beqn
\tilde{g}(w) \,=\, \begin{cases} \overline{f(u)}\,g(w) & \mbox{if~}w \in W_v, \\
e_{_{\lambdab_u, A_u}}(w) & \mbox{otherwise}.
\end{cases}
\eeqn
 \item[(c)] If $\mu=\lambda_{uu}$ is non-zero, then
\beqn
 \mathscr E_{W_{f, g}}(\mu) \,=\,\begin{cases} {\gr}(\tilde{f} \obslash g|_{_{\ell^2(W_u)}}) + [e_u] & \mbox{if~}f \in \ker \N{b}{u}, \\
 \ker(f \obslash g|_{_{\ell^2(W_u)}}) \oplus [e_u] & \mbox{otherwise},
 \end{cases}
 \eeqn
where
 $\tilde{f}={f}/{\lambda_{uu}}$.
 \item[(d)] If $\mu =\lambda_{uv}$ for some $v \in \supp\, \Hi{b}{u}$ and $\mu \notin \{0, \lambda_{uu}\},$ then
  \beqn
 \mathscr E_{W_{f, g}}(\mu) \,=\, {\gr}\big( \tilde{f} \obslash g|_{_{\ell^2(W_u)}}\big),
\eeqn
where $\tilde{f}=\frac{1}{\lambda_{uv}} \Big(f + \frac{\inp{f}{e_{_{\lambdab_u, A_u}}}}{\lambda_{uv}-\lambda_{uu}}e_u\Big).$
\end{enumerate}
To see the above statements, suppose that $(h, k) \in \mathscr E_{W_{f, g}}(\mu)$.
Since $\mf b \neq u$,  by Case I(2) of the proof of Theorem \ref{spectrum},
\beq
\label{case2-2}
\big(\ell^2(V) \ominus \Hi{b}{u}\big)\ominus [e_{_{\lambdab_u, A_u}}]=\ker(\N{b}{u}) \subseteq \mathscr E_{W_{f, g}}(0).
\eeq
The desired conclusion in (a) is now immediate from \eqref{spec-rone} and the assumption that $\lambda_{uv} \neq 0$ for every $v \in \supp\, \Hi{b}{u}$. To see (b), assume that
$\lambda_{uv}=0$ for some $v \in \supp\, \Hi{b}{u}$, and suppose that
$f= c\, e_u$ for some non-zero scalar $c.$
By
\eqref{spec-rone-k} and \eqref{case2-2}, $(h, k) \in \mathscr E_{W_{f, g}}(0)$  if and only if
\beqn c\Big(\sum_{v \in W_w} \!\! h(v)\overline{g(v)}\Big)   +
\inp{k}{e_{_{\lambdab_u, A_u}}}=0,\eeqn
which is equivalent to $(h, k) \in \ker(f \obslash \tilde{g}).$
This yields first part in (b). If $f$ and $e_u$ are linearly independent, then
$\sum_{v \in W_w} \!\! h(v)\overline{g(v)}=0$ and $\inp{k}{e_{_{\lambdab_u, A_u}}}=0$. Thus
the other part in (b) follows at once from \eqref{case2-2}. To see (c), suppose that $\mu=\lambda_{uu}$ is non-zero. By Case I(1) and \eqref{spec-rone-k}, $k=\alpha\, f + \beta\, e_u$ for some $\alpha, \beta \in \mathbb C.$ Combining this with \eqref{spec-rone-k} yields
\beqn
\Big(\sum_{v \in W_u} \!\! h(v)\overline{g(v)}-\alpha\, \lambda_{uu}\Big)f +\alpha\, \inp{f}{e_{_{\lambdab_u, A_u}}} e_u= 0
\eeqn
If $f \in \ker(\N{b}{u})$, then $\inp{f}{e_{_{\lambdab_u, A_u}}}=0,$ and the above equation determines $\alpha$ uniquely, whereas $\beta$ can be chosen arbitrarily to get the conclusion in the first part of (c). If $f \notin \ker(\N{b}{u})$, then by \eqref{spec-rone-kk}, $\sum_{v \in W_u} \!\! h(v)\overline{g(v)}=0,$ that is, $h \in
\ker(f \obslash g|_{_{\ell^2(W_u)}}).$ However, in this case, $k \in [e_u],$ which yields the remaining part of (c). To see (d), assume that $\mu=\lambda_{uv} \notin \{0, \lambda_{uu}\}$. Once again, by \eqref{k-unique}, $k= \tilde{f} \obslash g(h)$. This completes the verification of (d).
 \end{remark}

The following sheds more light into the spectral picture of rank one extensions of weighted join operators.
\begin{corollary} \label{coro-spectrum}
Let $\mathscr T=(V, E)$ be a rooted directed tree with root $\rootb$ and let $\mathscr T_{\infty}=(\V, E_{\infty})$ be the extended directed tree associated with $\mathscr T$.
For $u, \mf b \in V,$ consider the weight system
$\lambdab_u = \{\lambda_{uv}\}_{v\in \V}$ of complex numbers and let $W_{f, g}$ be the rank one extension of the weighted join operator $\W{b}{u}$ on $\mathscr T$, where $f \in \ell^2(V) \ominus \Hi{b}{u}$ is non-zero and $g : \supp\,\Hi{b}{u} \rar \mathbb C$ is given.
Then $\sigma(W_{f, g})$ is a proper closed subset of $\mathbb C$ if and only if $W_{f, g}$ satisfies the compatibility condition I. Further, we have the following:
\begin{enumerate}
\item[(a)] In case $W_{f, g}$ satisfies the compatibility condition I, $W_{f, g}$ defines a closed linear operator such that the following hold:
\begin{enumerate}
\item[(a1)]  $\sigma(W_{f, g})=\overline{\sigma_p(W_{f, g})}$.
\item[(a2)]  $\pi(W_{f, g})=\mathbb C \setminus \overline{\sigma_p(W_{f, g})}.$
\end{enumerate}
\item[(b)]  In case $W_{f, g}$ does not satisfy the compatibility condition I, the following hold:
\begin{enumerate}
\item[(b1)] $\sigma(W_{f, g})=\mathbb C$.
\item[(b2)] Either $W_{f, g}$ is not closed or $\pi(W_{f, g}) = \emptyset.$
\end{enumerate}
\end{enumerate}
\end{corollary}
\begin{proof}
Suppose that $\sigma(W_{f, g})$ is a proper closed subset of $\mathbb C$. Thus there exists $\mu \in \mathbb C \setminus \sigma_p(W_{f, g})$ such that $W_{f, g}-\mu$ is surjective. By \eqref{claim-surj}, we obtain the domain inclusion $\mathscr D(\Db{b}{u}) \subseteq \mathscr D(f \obslash g),$ and hence by Proposition \ref{lem-H-transform}, $W_{f, g}$ satisfies the compatibility condition I. Conversely, if $W_{f, g}$ satisfies the compatibility condition I, then $\varGamma_{\lambdab_u}$ is non-empty (see \eqref{varG-0}). It may now be concluded from (i) and (ii) of Theorem \ref{spectrum} that $\sigma(W_{f, g})$ is a proper closed subset of $\mathbb C.$ 

To see (a), assume that $W_{f, g}$ satisfies the compatibility condition I.
Since $\sigma(W_{f, g})$ is a proper subset of $\mathbb C,$ $W_{f, g}$ is closed (see \cite[Lemma 1.17]{CM}).  Further, (a1) follows from Theorem \ref{spectrum}(ii). Since the complement of the regularity domain of a densely defined closed operator is a closed subset of spectrum that contains the point spectrum  (see \cite[Proposition 2.1]{Sc}), the conclusion in (a2) is immediate.

To see (b), assume that $W_{f, g}$ does not satisfy the compatibility condition I.
Clearly, (b1) follows from the first part of this corollary.
To see (b2), assume that $W_{f, g}$ is closed. By \cite[Proposition 2.6]{Sc}, \beq \label{pi-di}
\{\mu \in \pi(W_{f, g}) : d_{W_{f, g}}(\mu) =0\}=\mathbb C \setminus \sigma(W_{f, g}) \overset{(b1)}=\emptyset \eeq
(see \eqref{d-S-mu}).
Let $\mu \in \pi(W_{f, g}).$ Since $\pi(W_{f, g}) \subseteq \mathbb C \setminus {\sigma_p(W_{f, g})},$ $\mu \notin \sigma_p(W_{f, g})$. Then, by the proof of Theorem \ref{spectrum}(i), $\mu \notin \sigma_p(\Db{b}{u}) \cup \sigma(\N{b}{u})$.
It follows that for every $ v \in \supp\,\Hi{b}{u},$
\beqn
 \left[\begin{array}{cc}
\Db{b}{u}-\mu & 0 \\
f \obslash g & \N{b}{u}-\mu
\end{array}
\right]  \left[\begin{array}{c}
(\lambda_{uv}-\mu)^{-1}e_v  \\
-(\lambda_{uv}-\mu)^{-1} \overline{g(v)}(\N{b}{u}-\mu)^{-1}(f)
\end{array}
\right]=\left[\begin{array}{c}
e_v  \\
0
\end{array}
\right].
\eeqn
which implies that $(W_{f, g}-\mu) \mathscr D(W_{f, g})$ is dense in $\Hi{b}{u}$. Also, $$(W_{f, g}-\mu)(\ell^2(V) \ominus \Hi{b}{u})=\ell^2(V) \ominus \Hi{b}{u},$$ which implies that $d_{W_{f, g}}(\mu)=0$. This together with \eqref{pi-di} shows that $\pi(W_{f, g})=\emptyset$ completing the proof.
\end{proof}

In general, the spectrum of $W_{f, g}$ may not be the topological closure of its point spectrum.

\begin{example} \label{exam-spectrum}
Let $g : \supp\,\Hi{b}{u} \rar \mathbb C$ be such that $\sum_{v \in
\supp\,\Hi{b}{u}}|g(v)|^2=\infty$ and
$$\mbox{card}(\supp(g))=\aleph_0=\mbox{card}(\supp\,\Hi{b}{u} \setminus \supp(g))$$
(for instance, one may let $u=v_2$ and $\supp(g)=\des{\bf v_5}$ in the rooted directed tree as given in Figure \ref{fig0}).
Let $\lambdab_u$ be a weight system such that \beq \label{annulus} \{\lambda_{uv} : v \in \supp(g)\} & \subseteq &\{z \in \mathbb C : 1 \leqslant |z| \leqslant 2\}, \\ \label{annulus1} \overline{\{\lambda_{uv} : v \in \supp\,\Hi{b}{u}\}} &\neq &  \mathbb C.\eeq It is easy to see using \eqref{annulus} that $W_{f, g}$ does not satisfy the compatibility condition. Hence, by Corollary \ref{coro-spectrum}, $\sigma(W_{f, g})=\mathbb C.$ Further, since $\sigma_p(W_{f, g})$ is not dense in $\mathbb C$ (see \eqref{annulus1}), we must have $\overline{\sigma_p(W_{f, g})} \subsetneq \sigma(W_{f, g}).$
Thus the spectral picture of a rank one extension $W_{f, g}$ of a weighted join operator can be summarized as follows:
\begin{enumerate}
\item If $g$ satisfies the compatibility condition I, then $\sigma(W_{f, g})=\overline{\sigma_p(W_{f, g})}$ is a proper subset of $\mathbb C.$
\item If $g$ does not satisfy the compatibility condition I, then $\sigma(W_{f, g})=\mathbb C$ and $\overline{\sigma_p(W_{f, g})}$ may be a proper subset of $\mathbb C.$
\end{enumerate}
In the last case, either $W_{f, g}$ is not closed or $\pi(W_{f, g})=\emptyset.$
\eop
\end{example}

As an application to Theorem \ref{spectrum}, we characterize those rank one extensions of weighted join operators on leafless directed trees, which admit compact resolvent.
\begin{corollary} \label{cpt-res}
Let $\mathscr T=(V, E)$ be a rooted directed tree with root $\rootb$
and let $\mathscr T_{\infty}=(\V, E_{\infty})$ be the extended
directed tree associated with $\mathscr T$. For $u \in V$ and $\mf b \in V \setminus \{u\},$ consider the weight system $\lambdab_u =
\{\lambda_{uv}\}_{v\in \V}$ of complex numbers and let $W_{f, g}$ be
the rank one extension  of the weighted join operator $\W{b}{u}$ on
$\mathscr T$, where $f \in \ell^2(V) \ominus \Hi{b}{u}$ is non-zero
and $g : \supp\,\Hi{b}{u} \rar \mathbb C$ is given. Suppose that
$W_{f, g}$ satisfies the compatibility condition I. If $\mathscr T$
is leafless, then the following
 are equivalent:
 \begin{enumerate}
 \item
The rank one extension $W_{f, g}$ of the weighted join operator $\W{b}{u}$ on $\mathscr T$ admits a
compact resolvent.
\item  The set $\{\lambda_{uv} : v \in \supp\,\Hi{b}{u}\}$ has accumulation point only at $\infty$ with each of its entries appearing finitely many times and the set
$V_{\prec}$ of branching vertices
of $\mathscr T$ is disjoint from $\asc{u}$.
\end{enumerate}
\end{corollary}
\begin{proof} We need a couple of general facts in this proof.
\begin{enumerate}
\item[(a)] The diagonal operator $D_{\lambdab}$ has compact resolvent if and only if the weight system $\lambdab$ has accumulation point only at $\infty$ with each of its entries appearing finitely many times.
\item[(b)] A finite block matrix with operator entries being bounded linear is compact if and only if all of its entries are compact.
\end{enumerate}
To see the equivalence of (i) and (ii), assume that $\mathscr T$ is leafless.
In view of (b), the formula \eqref{cpt-resolvent-eq} and Theorem \ref{spectrum}(iii), $W_{f, g}$ has compact resolvent if and only if $\Db{b}{u}$ and $\N{b}{u}$ have compact resolvents. 
On the other hand, by Proposition \ref{lem-lns}(ii) and \eqref{H-u}, $\ell^2(V) \ominus \Hi{b}{u}$ is finite dimensional if and only if $V_{\prec} \cap \asc{u} = \emptyset.$ In view of Lemma \ref{lem-inj-ten}(iii), this is equivalent to the assertion that $\N{b}{u}$ has compact resolvent.
The desired equivalence now follows from (a).
\end{proof}
\begin{remark} 
Assume that $\mathscr T$ is leafless and $\mf b =u$. Then, by \eqref{H-u-1}, $\ell^2(V) \ominus \Hi{b}{u}$ is one-dimensional. One may now argue as the proof of Corollary \ref{cpt-res} to show that $W_{f, g}$ has compact resolvent if and only if the set $\{\lambda_{uv} : v \in V\}$ has accumulation point only at $\infty$ with each of its entries appearing finitely many times.
\end{remark}

It is well-known that given any closed subset $\sigma$ of the complex plane, there exists a diagonal operator $D_{\lambdab}$ on $\ell^2(\mathbb N)$ such that $\sigma(D_{\lambdab})=\sigma.$ Here is a variant of this fact for rank one extensions of weighted join operators. 
\begin{corollary} \label{construction}
Let $\mathscr T=(V, E)$ be a rooted directed tree and let $\mf b \in V.$ Let $u \in V$ be such that $\mathscr T_u=(\des{u}, E_u)$ is an infinite directed subtree of $\mathscr T.$ Then,
for any closed, unbounded proper subset $\sigma$ of the complex plane, there exists a rank one extension $W_{f, g}$ of a weighted join operator $\W{b}{u}$ on $\mathscr T$ such that the following hold:
\begin{enumerate}
\item $g \notin \Hi{b}{u}$,
\item $W_{f, g}$ satisfies the compatibility condition I, and
\item $\sigma(W_{f, g})=\sigma$.
\end{enumerate}
\end{corollary}
\begin{proof} Let $f=e_u$ and let $z_0 \in \mathbb C \setminus \sigma.$
Let $\{\mu_n\}_{n \geqslant 1}$ be a countable dense subset of $\sigma$ and let $\{\nu_n\}_{n \geqslant 1}$ be a subset of $\sigma$ such that  
\beq \label{nu-growth}
|\nu_n- z_0| \geq {2}^{n/2}~\mbox{for every integer} ~n \geqslant 1
\eeq
(which exists since $\sigma$ is unbounded).
Consider the countable dense subset $\{\lambda_n\}_{n \geqslant 1}$ of $\sigma$  defined by
\beqn
\lambda_n \,=\, \begin{cases} \mu_k & \mbox{if~}n=2k, ~k \geqslant 1, \\
\nu_k & \mbox{if~}n=2k-1, ~k \geqslant 1.
\end{cases}
\eeqn
By axiom of choice \cite[Pg 11]{Si}, there exists a sequence $\{v_n\}_{n \geq 1} \subseteq \supp \, \Hi{b}{u}$ such that $\dep_{v_n} =n$ for every integer $n \geqslant 1.$
Set \beqn
\lambda_{uv} \,=\, \begin{cases}\lambda_n & \mbox{if~}v=v_n, \\
0 & \mbox{otherwise}.
\end{cases}
\eeqn
Define $g : \supp \,\Hi{b}{u} \rar \mathbb C$ by 
\beqn
g(v) \,=\, \begin{cases} \frac{\lambda_n-z_0}{\dep_v} & \mbox{if~}v=v_n, \\
0 & \mbox{otherwise}.
\end{cases}
\eeqn
Then, by the choice of vertices $\{v_n\}_{n \geq 1},$ the weight system $\lambdab_u$ and $g,$ \beqn \sum_{v \in \supp \, \Hi{b}{u}} \frac{|g(v)|^2}{|\lambda_{uv}-z_0|^2}=
\sum_{n=1}^{\infty} \frac{1}{n^2} < \infty,
\eeqn
and hence the rank one extension $W_{f, g}$ of the weighted join operator $\W{b}{u}$ satisfies the compatibility condition I. 
This combined with Theorem \ref{spectrum}(ii) yields (iii).
On the other hand, by \eqref{nu-growth} and the definition of $\{\lambda_n\}_{n \geqslant 1}$,
\beqn
\sum_{v \in \supp \, \Hi{b}{u}} |g(v)|^2 \,\geqslant\, \sum_{k=1}^{\infty} \frac{|\nu_{k}-z_0|^2}{|2k-1|^2} \,\geqslant\, \sum_{k=1}^{\infty} \frac{2^k}{|2k-1|^2},
\eeqn
which shows that $g \notin \Hi{b}{u}.$ This completes the proof. 
\end{proof}
It is worth noting that the conclusion of Corollary \ref{construction} does not hold in case $\sigma$ is a bounded subset of $\mathbb C$. Indeed, the boundedness of the spectrum of a rank one extension $W_{f, g}$ of a weighted join operator $\W{b}{u}$ implies that the diagonal operator $\Db{b}{u}$ is bounded. Then, by Remark \ref{rmk-rone-extn}, $W_{f, g}$ is not even closable. So, by Theorem \ref{closed-0}, $W_{f, g}$ can not satisfy a compatibility condition.


\subsection{Spectra of weighted join operators}

A special case of Theorem \ref{spectrum} (the case of $g=0$) provides a complete spectral picture for weighted join operators. This together with some additional properties is summarized in the next result.
We first introduce some notations and definitions.

Let $T$ be a densely defined linear operator in $\mathcal H.$
A complex number $\mu$ is said to be a {\it generalized eigenvalue} of $T$ if there exists a positive integer $k$ and a non-zero vector $f \in \mathscr D(T^k)$ (to be referred to as {\it generalized eigenvector corresponding to $\mu$}) such that $(T-\mu)^kf=0$.
The {\it rootspace} $\mathscr R_T(\mu)$ of $T$ corresponding to the generalized eigenvalue $\mu$ is defined as the closed space spanned by the corresponding generalized eigenvectors of $T$. Any vector in the rootspace of $T$ is referred to as {\it root vector} for $T.$
We say that $T$ is {\it complete} if it has a complete set of root vectors (refer to \cite{GK} for the basics of completeness of root systems and to \cite{BY} for completeness of root systems for the class of rank one perturbations of self-adjoint operators).

\index{$\mathscr R_T(\mu)$}

\begin{theorem} \label{spectrum1}
Let $\mathscr T=(V, E)$ be a rooted directed tree with root $\rootb$ and let $\mathscr T_{\infty}=(\V, E_{\infty})$ be the extended directed tree associated with $\mathscr T$.
For $u \in V$ and $\mf b \in V \setminus \{u\},$ consider the weight system
$\lambdab_u = \{\lambda_{uv}\}_{v\in \V}$ of complex numbers and let $\W{b}{u}$ be the weighted join operator on $\mathscr T$.
Then, we have the following statements:
\begin{enumerate}
\item The point spectrum $\sigma_p(\W{b}{u})$ of $\W{b}{u}$ is given by
\beqn
\sigma_p(\W{b}{u}) \,=\, \begin{cases} 
\{\lambda_{uv} : v \in V\} & \mbox{if~} \mf b =u, \\
\{\lambda_{uv} : v \in \asc{u} \cup \Desb{u}\} \cup \{0\} & \mbox{if~} \mf b \in \mathsf{Des}_u(u), \\
\{\lambda_{uv} : v \in \des{u}\} \cup \{0\} & \mbox{otherwise}.
\end{cases}
\eeqn
\item 
The spectrum $\sigma(\W{b}{u})$ of $\W{b}{u}$ is the topological closure of $\sigma_p(\W{b}{u}).$
\item If $\mathscr E_{\W{b}{u}}(\mu)$ denotes the eigenspace corresponding to the eigenvalue $\mu$ of $\W{b}{u}$, then
\beqn
\mathscr E_{\W{b}{u}}(\mu) \,=\, \begin{cases} \ell^2(W_\mu) & \mbox{if~} \mu \neq 0, ~\mu \neq \lambda_{uu}\\
\ell^2(W_\mu) \oplus [e_u] & \mbox{if~} \mu \neq 0, ~\mu =\lambda_{uu},\\
\ell^2(W_0) \oplus \big(\ell^2(V) \ominus \Hi{b}{u}\big) \ominus [e_{_{\lambdab_u, A_u}}] & \mbox{if~}\mu=0,
  \end{cases}
\eeqn
where $W_\mu$ is given by \eqref{W-v}, $\Hi{b}{u}$ is given by \eqref{H-u} and $A_u$ is given by \eqref{A-u}.
\item The multiplicity function ${\mf m}_{\W{b}{u}} : \sigma_p(\W{b}{u}) \rar \mathbb Z_+ \cup \{\aleph_0\}$ is given by
\beqn
{\mf m}_{\W{b}{u}} (\mu) \,=\,  \mbox{card}\,\{v \in \{u\} \cup \supp\, \Hi{b}{u}  : \lambda_{uv}=\mu\} ~   \mbox{if~} \mu \neq 0.
\eeqn
In addition, if $\mathscr T$ is leafless and if there exists a branching vertex $w \in \asc{u},$ then
${\mf m}_{\W{b}{u}} (0)= \aleph_0.$
\item The rootspace $\mathscr R_{\W{b}{u}}(\mu)$ of $\W{b}{u}$ corresponding to the generalized eigenvalue $\mu$ is given by
\beqn
\mathscr R_{\W{b}{u}}(\mu) \,=\, \begin{cases}   \mathscr E_{\W{b}{u}}(\mu)  & \mbox{if} ~ \mu \in \sigma_p(\W{b}{u}) \setminus \{0\},   \\
\mathscr E_{\W{b}{u}}(0) & \mbox{if~} \mu =0,  ~\lambda_{uu} \neq 0, \\
\mathscr E_{\W{b}{u}}(0) \oplus [e_{_{\lambdab_u, A_u}}]  & \mbox{if~} \mu =0, ~ \lambda_{uu} = 0.
\end{cases}
\eeqn
\end{enumerate}
\end{theorem}
\begin{remark}
In case $\mf b  = u$, the spectral picture of $\W{b}{u}$ coincides with that of the diagonal operator $\D{u}$ (see Remark \ref{rmk-dicho}). We discuss here the spectral picture of $\W{\mf \infty}{u}$.
By \eqref{meet-deco}, $\W{\mf \infty}{u}$ is orthogonal direct sum of rank one operators
\beqn e_{u_j} \obslash e_{_{\lambdab_u,  \des{u_j} \setminus \des{u_{j-1}}}}, \quad j =0, \ldots, \dep_u, \eeqn where
 $\des{u_{-1}}=\emptyset$ and $u_j:=\parentn{j}{u}$ for $j =0, \ldots, \dep_u.$
 If $\lambdab_u \in \ell^2(V)$, then it follows that $\W{\mf \infty}{u} \in B(\ell^2(V)),$ and hence by Lemma \ref{lem-inj-ten}(iii),
\beqn
 \sigma(\W{\mf \infty}{u}) = \{0\} \cup \{\lambda_{uu_j} : j=0, \ldots, \dep_u\} = \sigma_p(\W{\mf \infty}{u}).
\eeqn
Assume now that $\lambdab_u \notin \ell^2(V).$
By Lemma \ref{lem-inj-ten-unb}, $\sigma_p(\W{\mf \infty}{u})$ is given by the same formula as above.
Further, another application of Lemma \ref{lem-inj-ten-unb} shows that $\W{\mf \infty}{u}$ is not closed, and hence $\sigma(\W{\mf \infty}{u})=\mathbb C.$
\end{remark}
\begin{proof} 
The conclusions in (i) and (ii) are immediate from (i) and (iii) of Theorem \ref{spectrum}. Since $\ell^2(A_u) \subseteq \ell^2(V) \ominus \Hi{b}{u}$ and $\mbox{card}(A_u) \Ge 2$ (see \eqref{A-u}),
by Lemma \ref{lem-inj-ten}(iii) and \eqref{til-N}, we obtain
\beqn
&& \sigma_p(\N{b}{u}) = \{0, \lambda_{uu}\}  =\sigma(\N{b}{u}), \\   && {\mf m}_{\N{b}{u}}(0)=\aleph_0  ~\mbox{if~} \dim \big(\ell^2(V) \ominus \Hi{b}{u}\big)=\aleph_0,  \quad {\mf m}_{\N{b}{u}}(\lambda_{uu})=1~\mbox{if~}\lambda_{uu} \neq 0, \\ && \lambda_{uu} \neq 0 ~\Longrightarrow ~\mathscr E_{\N{b}{u}}(\mu) = \begin{cases} [e_{_{\lambdab_u, A_u}}]^{\perp} & \mbox{if~}\mu=0, \\
[e_u] & \mbox{if~}\mu=\lambda_{uu}.
\end{cases}
\eeqn
In view of the fact that $\dim  \big(\ell^2(V) \ominus \Hi{b}{u}\big)=\aleph_0$ if and only if
$\mbox{card}(V_u) = \aleph_0$,
the conclusions in (iii) and (iv) pertaining to the eigenspaces and multiplicities now follow from Proposition \ref{lem-lns}.

 To see (v), let $k$ be a positive integer and let $\mu \in \mathbb C \setminus \{0\}$. By \eqref{nilp-powers},
\beqn
 (\N{b}{u}-\mu)^k &=& \sum_{l=0}^k (-\mu)^{k-l}{k \choose l}{\N{b}{u}}^l  \\ &=&
 (-\mu)^k I +
 \displaystyle \sum_{l=1}^k (-\mu)^{k-l}{k \choose l}  \lambda^{l-1}_{uu} e_u \otimes e_{_{\lambdab_u, A_u}}.
\eeqn
It is not difficult to see that for any $h \in \ell^2(V) \ominus \Hi{b}{u}$ such that
$(\N{b}{u}-\mu)^kh=0,$ we must have $h= h(u)e_u$ and $\mu = \lambda_{uu}$.
 However, $e_u$ is an eigenvector, and hence a root vector.
 This shows that the rootspace of $\N{b}{u}$ corresponding to $\mu$ is spanned by $e_u.$ Since the generalized eigenvalues and eigenvalues coincide for a diagonal operator, the desired conclusion in (v) is immediate provided $\lambda_{uu} \neq 0$. In case $\lambda_{uu}=0,$ by Corollary \ref{coro-c-Jordan}, $\N{b}{u}$ is a nilpotent operator of nilpotency index $2$, and hence any $h \in \ell^2(V) \ominus \Hi{b}{u}$ is a root vector. This completes the verification of (v).
\end{proof}


A result of  Wermer says that the spectral synthesis holds for all normal compact operators \cite{W}. As shown by Hamburger \cite{Ha}, this no longer holds for compact operators.
Interestingly, the following result can be used to construct examples of compact non-normal operators which are not even complete.
\begin{corollary} \label{completeness}
Let $\mathscr T=(V, E)$ be a rooted directed tree with root $\rootb$ and let $\mathscr T_{\infty}=(\V, E_{\infty})$ be the extended directed tree associated with $\mathscr T$.
For $u \in V$ and $\mf b \in V \setminus \{u\}$, consider the weight system
$\lambdab_u = \{\lambda_{uv}\}_{v\in \V}$ of complex numbers and let $\W{b}{u}$ be the weighted join operator on $\mathscr T$. Then $\W{b}{u}$  is complete if and only if $\lambda_{uu}=0$ or $\lambda_{uv}=0$ for every $v \in A_u \setminus \{u\}$.
\end{corollary}
\begin{proof}
If $\lambda_{uu}=0$, then by (iii) and (v) of Theorem \ref{spectrum1}, the root vectors for $\W{b}{u}$ forms a complete set. If $\lambda_{uv}=0$ for every $v \in A_u \setminus \{u\}$, then by Theorem \ref{o-deco}, $\W{b}{u}$ is a diagonal operator, and hence we get the sufficiency part. To see the necessity part, suppose that $\lambda_{uu} \neq 0$ and $\lambda_{uv} \neq 0$ for some $v \in A_u \setminus \{u\}$. Thus $e_{_{\lambdab_u, A_u}}$ is a non-zero vector in $\ell^2(V)$, and by (iv) and (v) of Theorem \ref{spectrum1}, $$\bigvee_{\mu \in \mathbb C} \mathscr R_{\W{b}{u}}(\mu) ~\subseteq ~\ell^2(V) \ominus [e_{_{\lambdab_u, A_u}}].$$ This shows that $\W{b}{u}$ is not complete.
\end{proof}

\chapter{Special classes}

In this chapter, we discuss some special classes of
weighted join operators and their rank one extensions.
In particular, we exhibit families of sectorial operators and infinitesimal generators of quasi-bounded strongly continuous semigroups within these classes. Further, we characterize hyponormal
operators and $n$-symmetric operators within the class of weighted join operators
on rooted directed trees.
We also investigate the classes of hyponormal and $n$-symmetric rank one extensions $W_{f, g}$ of weighted join operators. The complete characterizations of these classes seem to be beyond reach at present, particularly, in view of the fact that the structures of positive integral powers and Hilbert space adjoint of $W_{f, g}$ are complicated.

\section{Sectoriality}


\index{$\mathbb S_{\theta, \alpha}$}

A densely defined linear operator $T$ in $\mathcal H$ is {\it sectorial}~ if
there exist $a \in {\mathbb R}, \ M \in
(0,\infty)$ and $\theta \in (0, \frac {\pi}{2})$ such that 
\vskip-.6cm
\beq \label{r-estimate}
 \lambda \in \rho(T)~\mbox{and~} \|R_T(\lambda)\| \leqslant \frac{M}{|\lambda -a|}\quad \mbox{whenever $\lambda \in {\mathbb C}$ and $|\arg(\lambda-a)| \geqslant
\theta$.}
 \eeq
Sometimes we say that $T$ is sectorial with {\it angle} $\theta$ and {\it vertex} $a$.
Note that there is considerable divergence of terminology in the literature, for example, Kato \cite{K} calls it {\it $m$-sectorial}, while we call it just sectorial, the correspondence being a minus sign between the two.
For the basic theory of sectorial operators, the reader is referred to  \cite{AJS, K, BSU, M, Has, Sc, CM, Si1, SS}).
 The following result yields a family of sectorial rank one extensions of weighted join operators (cf. \cite[Proposition 3]{Ja}).
\begin{proposition} \label{sectorial}
Let $\mathscr T=(V, E)$ be a rooted directed tree with root $\rootb$ and let $\mathscr T_{\infty}=(\V, E_{\infty})$ be the extended directed tree associated with $\mathscr T$.
For $u, \mf b \in V,$ consider the weight system
$\lambdab_u = \{\lambda_{uv}\}_{v\in \V}$ of complex numbers and let $W_{f, g}$ be the rank one extension  of the weighted join operator $\W{b}{u}$ on $\mathscr T$, where $f \in \ell^2(V) \ominus \Hi{b}{u}$ is non-zero and $g : \supp\,\Hi{b}{u} \rar \mathbb C$ is given.
Suppose that $W_{f, g}$ satisfies the compatibility condition II.
If  $\{\lambda_{uv} : v \in \supp\,\Hi{b}{u}\}$ is contained in the sector $\mathbb S_{\theta, \alpha} :=\{z \in \mathbb C : |\arg(z-\alpha)| <
\theta\}$ for some $\theta \in (0, \pi/2)$ and $\alpha \in \mathbb R$, then $W_{f, g}$ is a sectorial operator.
\end{proposition}
\begin{proof}
Assume that $\{\lambda_{uv} : v \in \supp\,\Hi{b}{u}\}$ is contained in $\mathbb S_{\theta, \alpha}$ for some $\theta \in (0, \pi/2)$. After replacing $\theta$ by $\theta + \epsilon \in (0, \pi/2)$ for some $\epsilon >0$, we may assume without loss of generality  that $\sigma(\Db{b}{u}) \subseteq \overline{\mathbb S}_{\theta', \alpha}$ for some $\theta' \in (0, \theta)$.
It is well-known that the diagonal operator $\Db{b}{u}$ satisfies the estimate \eqref{r-estimate}
(see, for instance, \cite[Chapter 2, Section 2.2.1]{Has} and \cite[Example 4.5.2]{M}). Indeed, for any $\mu \in \mathbb C \setminus \overline{\mathbb S}_{\theta}$,
\pagebreak
\beqn
\|(\Db{b}{u}-\mu)^{-1}\| &=& \sup_{v \in \supp\,\Hi{b}{u}}|\lambda_{uv}-\mu|^{-1} \\ &=& \frac{1}{\displaystyle \inf_{v \in \supp\,\Hi{b}{u}}|\lambda_{uv}-\mu|} \\ & \leqslant & \frac{1}{\displaystyle \inf_{t \in \overline{\mathbb S}_{\theta', \alpha}}|t-\mu|} \\ &\leqslant &  \begin{cases} \frac{1}{|\mu-\alpha|} & \mbox{if~}|\arg(\mu-\alpha)| \geqslant \theta + \pi/2, \\
\frac{1}{|\mu-\alpha|} \frac{1}{\sin(\arg(\mu-\alpha)-\theta')} & \mbox{otherwise}
\end{cases}
\eeqn
(see Figure \ref{fig10}).
Thus \eqref{r-estimate} holds with $M=\frac{1}{\sin(\theta-\theta')}$. 
By Corollary \ref{bdd-rel},
$W_{f, g}=A+B+C$, where $B+C$ is $A$-bounded with  $A$-bound equal to $0$ (see \eqref{ABC} and \eqref{kbs-estimate-1}). Since $A$ is sectorial (since so is $\Db{b}{u}$), an application of \cite[Theorem 4.5.7]{M} shows that $W_{f, g}$ is sectorial.
\end{proof}

\begin{figure}
\begin{tikzpicture}[scale=.8, transform shape]
\tikzset{vertex/.style = {shape=circle,draw,minimum size=1em}}
\tikzset{edge/.style = {->,> = latex'}}
\node[] (z) at  (-5,0) {};
\draw [->] (-6, 0) -- (6, 0);
\draw [->] (6, 0) -- (-6, 0);
\draw [->] (0, -5.3) -- (0, 5.3);
\draw [->] (0, 5.3) -- (0, -5.3);
\draw (-1, 0) -- (0, .5) -- (1, 1) -- (2, 1.5) -- (4, 2.5);
\draw (-1, 0) -- (0, -.5) -- (1, -1) -- (2, -1.5) -- (4, -2.5);
\draw (-1, 0) --  (3.5, 3);
\draw (-1, 0) --  (3.5, -3);
\draw [dashed] (-2, 5) -- (.2, .6);
\draw [dashed] (-2, 5) -- (-1, 0);
\node[] (a) at  (3.9,2.8) {$\mathbb S_{\theta', \alpha}$};
\node[] (a) at  (2.9,3.1) {$\mathbb S_{\theta, \alpha}$};
\node[] (a) at  (-1.3,-0.3) {$\alpha$};
\node[] (y) at  (-1,0) {$\bullet$};
\node[] (mu) at  (-2,5) {$\bullet$};
\node[] (y) at  (-1.8,5.3) {$\mu$};
\node[] (y) at  (-5,5) {$\theta < |\arg(\mu-\alpha)|<\theta+\pi/2$};
\node[] (y) at  (-2.3,2.6) {$|\mu-\alpha|$};
\node[] (t1) at  (0.2, 0.2) {$\theta'$};
\node[] (t2) at  (.9,.3) {$\theta$};
\node[] (t2) at  (-.9,.7) {$\eta$};
\node[] (t2) at  (-3.5,.7) {$\eta:=\arg(\mu-\alpha)-\theta'$};
\node[] (a) at  (4,.7) {$\big\{\lambda_{uv} : v \in \supp\,\Hi{b}{u}\big\}$};
\node[] (d) at  (0,1) {};
\node[] (e) at  (0,-1) {};
\def\Radius{.7}
\draw [dashed] (0, .7) arc(90:-90:\Radius) -- cycle;
\def\Radius{.5}
\draw [dashed] (0, .5) arc(90:-90:\Radius) -- cycle;
\def\Radius{.5}
\draw [dashed] (-1.1, .5) arc(100:30:\Radius) ;
\end{tikzpicture}
\caption{The situation of Proposition \ref{sectorial}} \label{fig10}
\end{figure}

\index{$\mathbb H_\alpha$}

For all relevant definitions and basic theory of  strongly continuous quasi-bounded semigroups, the reader is referred to \cite{K, M, SS}. A result similar to the following has been obtained in \cite[Proposition 3.1]{N0} for a family of upper triangular operator matrices on non-diagonal domains (cf. \cite[Theorem 7.11]{BBDHL}).
\begin{proposition} \label{semigroup}
Let $\mathscr T=(V, E)$ be a rooted directed tree with root $\rootb$ and let $\mathscr T_{\infty}=(\V, E_{\infty})$ be the extended directed tree associated with $\mathscr T$.
For $u, \mf b \in V,$ consider the weight system
$\lambdab_u = \{\lambda_{uv}\}_{v\in \V}$ of complex numbers and let $W_{f, g}$ be the rank one extension  of the weighted join operator $\W{b}{u}$ on $\mathscr T$, where $f \in \ell^2(V) \ominus \Hi{b}{u}$ is non-zero and $g : \supp\,\Hi{b}{u} \rar \mathbb C$ is given.
Suppose that $W_{f, g}$ satisfies the compatibility condition II.
If  $\{\lambda_{uv} : v \in \supp\,\Hi{b}{u}\}$ is contained in the right half plane ${\mathbb H}_{\alpha}=\{z \in \mathbb C : \Re z \geqslant  \alpha\}$ for some $\alpha \in \mathbb R$, then $W_{f, g}$ is the generator of a strongly continuous semigroup $\{Q(t)\}_{_{t \geqslant 0}}$ satisfying $\|Q(t)\| \leqslant M e^{-\alpha t}$ for $t \geqslant 0$.
\end{proposition}
\begin{proof} Assume that $\{\lambda_{uv} : v \in \supp\,\Hi{b}{u}\}$ is contained in the right half plane ${\mathbb H}_{\alpha}=\{z \in \mathbb C : \Re z \geqslant  \alpha\}$ for some $\alpha \in \mathbb R$. Thus $\sigma(\Db{b}{u}) \subseteq \mathbb H_{\alpha}$. It is easy to see that
\beqn
\|(\Db{b}{u}-\mu)^{-n}\| \,\leqslant\,\frac{1}{|\alpha-\mu|^n}, \quad \mu \in (-\infty, \alpha), ~n=1, 2, \ldots.
\eeqn
Hence, by the Hille-Yoshida Theorem (see \cite[Theorem 4.3.5]{M}, \cite[Theorem 2.3.3]{SS}), $\Db{b}{u}$ is the generator of a strongly continuous semigroup $\{S(t)\}_{_{t \geqslant 0}}$ satisfying the quasi-boundedness condition $\|S(t)\| \leqslant M e^{-\alpha t}$, $t \geqslant 0$.
By Corollary \ref{bdd-rel},
$W_{f, g}=A+B+C$, where $B+C$ is $A$-bounded with  $A$-bound equal to $0$ (see \eqref{ABC} and \eqref{kbs-estimate-1}).
Note that $A$ is the generator of the strongly continuous semigroup $\{S(t)\oplus I\}_{_{t \geqslant 0}}.$ The desired conclusion may now be derived from \cite[Corollary 2.5, Chapter IX]{K}.
\end{proof}
If $W_{f, g}$ is as in the preceding result, one can define fractional powers of $W_{f, g}$ (refer to \cite[Chapter 6]{M}). Further, one may obtain a counter-part of \cite[Corollary 2]{Ja} ensuring $H^{\infty}$-functional calculus for rank one extensions of weighted join operators (cf. \cite[Proposition 3.4]{Ch}). We refer the reader to \cite{Has} for more details on this topic.

\section{Normality}
A densely defined linear operator $T$ in $\mathcal H$ is said to be
{\it hyponormal} if $\mathscr D(T) \subseteq \mathscr D(T^*)$ and $\|T^*x\| \Le \|Tx\|$ for all $x \in \mathscr D(T).$ We say that $T$ is
{\it cohyponormal} if $T$ is closed and $T^*$ is hyponormal.
There has been significant literature on the classes of hyponormal and cohyponormal operators  (refer to \cite{Ja, OS, Ja2, Ja3, CL, JL, Jablonski, JBS-2, JBS-3}).

We begin with a rigidity result stating
that no weighted join operator can be hyponormal unless it is
diagonal. A variant of this fact in the context of bounded operators has been obtained in \cite[Theorem 2.3]{JL} (cf. \cite[Proposition 3.1]{I}).
\begin{proposition} \label{hypo}
Let $\mathscr T=(V, E)$ be a rooted directed tree with root $\rootb$ and let $\mathscr T_{\infty}=(\V, E_{\infty})$ be the extended directed tree associated with $\mathscr T$. For $u, \mf b \in V,$ consider the weight system
$\lambdab_u = \{\lambda_{uv}\}_{v\in \V}$ of complex numbers and let $\W{b}{u}$ be the weighted join operator on $\mathscr T$.
 Then the following statements are equivalent:
\begin{enumerate}
\item[(i)] $\W{b}{u}$ is normal.
\item[(ii)] $\W{b}{u}$ is hyponormal.
\item[(iii)] $\W{b}{u}$ is cohyponormal.
\item[(iv)] $\W{b}{u}$ is diagonal with respect to the orthonormal basis $\{e_v\}_{v \in V}.$
\item[(v)] $\mf b = u$ or $\lambda_{uv}=0$ for every $v \in A_u \setminus \{u\}$, where $A_u$ is given by \eqref{A-u}.
\end{enumerate}
\end{proposition}
\begin{proof} By Remark \ref{rmk-dicho}, $\W{b}{u}$  is a diagonal operator if $\mf b = u$ or $\lambda_{uv}=0$ for every $v \in A_u \setminus \{u\}$, and
hence (v) implies (i)-(iv).
By Theorem \ref{o-deco}, $\W{b}{u}$ admits the decomposition $(\Db{b}{u}, \N{b}{u}, \Hi{b}{u})$. Thus
the Hilbert space adjoint of $\W{b}{u}$ is given by
\beqn
\mathscr D({\W{b}{u}}^*) &=& \mathscr D(\W{b}{u}), \\
{\W{b}{u}}^* &=&  {\Db{b}{u}}^* \oplus {\N{b}{u}}^*.
\eeqn
Since $\Db{b}{u}$ is normal,
$\W{b}{u}$ is hyponormal (resp. cohyponormal) if and only if so is $\N{b}{u}.$
Note that
by \eqref{adjoint-ro},
${\N{b}{u}}^* =
 e_{_{\lambdab_u, A_u}} \otimes e_u.$
Also, by \eqref{adjoint-ro}, for $x, y \in \ell^2(V),$
\beqn [x \otimes y, (x \otimes y)^*]  &=& (x \otimes y )(y \otimes x) - (y \otimes x)(x \otimes y) \\ &=&
\|y\|^2 x \otimes x - \|x\|^2 y \otimes y.
\eeqn
It follows that
\beqn
[{\N{b}{u}}^*, \N{b}{u}]  \,=\,   \begin{cases}
   \quad e_{_{\lambdab_u, [u, \mf b]}} \otimes e_{_{\lambdab_u, [u, \mf b]}} - \|e_{_{\lambdab_u, [u, \mf b]}}\|^2 e_u \otimes e_u
  & \mbox{if~} \mf b \in \des{u}, \\
  e_{_{\lambdab_u, \asc{u} \cup \{u, \mf b\}}} \otimes e_{_{\lambdab_u, \asc{u} \cup \{u, \mf b\}}} \\ \quad - \quad \|e_{_{\lambdab_u, \asc{u} \cup \{u, \mf b\}}}\|^2 e_u \otimes e_u
   & \mbox{otherwise}.
\end{cases}
\eeqn
This yields
\beqn
\inp{[{\N{b}{u}}^*, \N{b}{u}]e_u}{e_u} \,=\, \begin{cases} - \|e_{_{\lambdab_u, (u, \mf b]}}\|^2
 & \mbox{if~} \mf b \in \des{u}, \\
- \|e_{_{\lambdab_u, \asc{u} \cup \{\mf b\}}}\|^2
  & \mbox{otherwise},
\end{cases}
\eeqn
which is always negative provided $\mf b \neq u$ and $\lambda_{uv} \neq 0$ for some $v \in A_u \setminus \{u\}$.
Further,  in this case,
\beqn
\inp{[{\N{b}{u}}^*, \N{b}{u}]e_{v}}{e_{v}} \,=\,  |\lambda_{uv}|^2, \quad v \in A_u \setminus \{u\},
\eeqn
and hence $\N{b}{u}$ is not cohyponormal.
This completes the proof.
\end{proof}

Let us now investigate the class of normal rank one extensions $W_{f, g}$ of weighted join operators.

\begin{proposition} \label{hypo-rone}
Let $\mathscr T=(V, E)$ be a rooted directed tree with root $\rootb$ and let $\mathscr T_{\infty}=(\V, E_{\infty})$ be the extended directed tree associated with $\mathscr T$.
For $u, \mf b \in V,$ consider the weight system
$\lambdab_u = \{\lambda_{uv}\}_{v\in \V}$ of complex numbers and let $W_{f, g}$ be the rank one extension of the weighted join operator $\W{b}{u}$ on $\mathscr T$, where $f \in \ell^2(V) \ominus \Hi{b}{u}$ is non-zero and $g : \supp\,\Hi{b}{u} \rar \mathbb C$ is given.
Suppose that $W_{f, g}$ satisfies the compatibility condition I.
Then the following statements are equivalent:
\begin{enumerate}
\item[(i)] $W_{f, g}$ is normal.
\item[(ii)] $W_{f, g}$ is diagonal with respect to the orthonormal basis $\{e_v\}_{v \in V}$.
\item[(iii)]  $g=0$ and
\beq
\label{hypo-nc}
\mbox{ either $\mf b = u$ or $\lambda_{uv}=0$ for every $v \in A_u \setminus \{u\}$.}
\eeq
\end{enumerate}
\end{proposition}
\begin{proof}
Assume that $W_{f, g}$ is normal. Note that $\N{b}{u},$ being the restriction of $W_{f, g}$ to $\ell^2(V) \ominus \Hi{b}{u},$ is hyponormal. By Proposition \ref{hypo}, (iii) holds. Thus $\N{b}{u}=\lambda_{uu} e_u \otimes e_u.$ Since $W_{f, g}$ satisfies the compatibility condition I, by Corollary \ref{coro-spectrum}, $\sigma(W_{f, g})$ is a proper closed subset of $\mathbb C$. Let $\mu \in \mathbb C \setminus \sigma(W_{f, g})$.
By Theorem \ref{spectrum}(iii),
\beq \label{ff}
\begin{array}{lll}
(W_{f, g}-\mu)^{-1} \,=\, \left[\begin{array}{cc}
(\Db{b}{u}-\mu)^{-1} & 0 \\
-(\N{b}{u}-\mu)^{-1}L_{\lambdab_u, \mu} & (\N{b}{u}-\mu)^{-1}
\end{array}
\right],
\end{array}
\eeq
where the linear transformation $L_{\lambdab_u, \mu}$ is given by $L_{\lambdab_u, \mu}=(f \obslash g)(\Db{b}{u}-\mu)^{-1}$. On the other hand, by \cite[Proposition 3.26(v)]{Sc}, $(W_{f, g}-\mu)^{-1}$ is normal. Let $A=(\Db{b}{u}-\mu)^{-1}$, $B=-(\N{b}{u}-\mu)^{-1}L_{\lambdab_u, \mu}$ and $C=(\N{b}{u}-\mu)^{-1} .$
Since $A$ and $C$ are normal, it may be concluded from \eqref{ff} that
\beqn
[((W_{f, g}-\mu)^{-1})^*, (W_{f, g}-\mu)^{-1}] 
&=&
\left[\begin{array}{cc}
B^*B & B^*C-AB^* \\
C^*B-BA^* & BB^*
\end{array}
\right].
\eeqn
Since $W_{f, g}$ is normal, $B=0$. It follows that $L_{\lambdab_u, \mu}=0$, and hence $f \obslash g=0$ on $\mathscr D(\Db{b}{u}).$ Since $f$ is non-zero, we must have $g=0.$
The remaining implications follow from Proposition \ref{hypo}.
\end{proof}
The methods of proofs of Propositions \ref{hypo} and \ref{hypo-rone} are different. In particular, unavailability of a formula for the Hilbert space adjoint of $W_{f, g}$ necessitated us to characterize normality of $W_{f, g}$ with the help of the resolvent function.
An inspection of the proof of Proposition \ref{hypo-rone} shows that \eqref{hypo-nc} is a necessary condition for $W_{f, g}$ to be a hyponormal operator.

\section{Symmetricity}

A densely defined linear operator $T$ in $\mathcal H$ is said to be  {\it $n$-symmetric} if \beqn \sum_{j=0}^n (-1)^{n-j} {n \choose j}
\inp{T^jx}{T^{n-j}y}=0, \quad x, y \in \mathscr D(T^n),\eeqn
where $n$ is a positive integer.
We refer to $1$-symmetric operator as {\it symmetric operator}.
We say that $T$ is {\it strictly $n$-symmetric} if it is $n$-symmetric, but not $(n-1)$-symmetric, where $n$ is a positive integer bigger than $1.$
For the basic properties of $n$-symmetric operators and its connection with the theory of differential equations, the reader is referred to \cite{He0, He, BH, Sc0, BF, Ag-St, Ru}.

The following proposition describes all $n$-symmetric weighted join operators.
It reveals the curious fact that there is no strictly $2$-symmetric weighted join  operator. On the other hand, strictly $3$-symmetric weighted join operators exist in abundance.
\begin{proposition} \label{symm}
Let $\mathscr T=(V, E)$ be a rooted directed tree with root $\rootb$ and let $\mathscr T_{\infty}=(\V, E_{\infty})$ be the extended directed tree associated with $\mathscr T$. For $\mf b, u \in V$ and the weight system
$\lambdab_u = \{\lambda_{uv}\}_{v\in \V}$ of complex numbers,
 let $\W{b}{u}$ denote
 the weighted join operator  on ${\mathscr T}$. Then, for any positive integer $n$, the following are equivalent:
\begin{enumerate}
\item $\W{b}{u}$ is $n$-symmetric.
\item The weight system $\lambdab_u$ satisfies
\beqn
\lambda_{uv} \in \mathbb R, \quad v \in \begin{cases}
\asc{u} \cup  \Desb{u}
 & \mbox{if~} \mf b \in \des{u},  \\
 \des{u} & \mbox{otherwise},
\end{cases}
\eeqn
and one of the following holds:
\begin{enumerate}
\item $\lambda_{uu}=0$ and $n \Ge 3$.
\item $
\lambda_{uv} =0  ~\mbox{for~} v \in A_u \setminus \{u\},
$ where $A_u$ is as given in \eqref{A-u}.
\end{enumerate}
\item One of the following holds:
\begin{enumerate}
\item $\lambda_{uu}=0$ and $n \Ge 3$
\item The weighted join operator $\W{b}{u}$ is symmetric.
\end{enumerate}
\end{enumerate}
\end{proposition}
\begin{proof} By Theorem \ref{o-deco}, the weighted join operator $\W{b}{u}$ admits the decomposition $(\Db{b}{u}, \N{b}{u}, \Hi{b}{u})$.
Thus $\W{b}{u}$ is $n$-symmetric if and only if $\Db{b}{u}$ and
$\N{b}{u}$ are $n$-symmetric. It is easy to see that a diagonal
operator is $n$-symmetric if and only if it is symmetric, which in
turn is equivalent to the assertion that its diagonal entries are
real. Assume that $\N{b}{u}$ is $n$-symmetric and let $v, w \in
A_u.$ Note that by \eqref{nilp-powers}, \beq \label{7.1}
 & & \sum_{j=0}^n (-1)^{n-j} {n \choose j} \inp{{\N{b}{u}}^je_v}{{\N{b}{u}}^{n-j}e_w} \\
&=& \begin{cases} \notag
(\lambda_{uu}-\bar{\lambda}_{uu})^n & \mbox{if~}v =u, ~w=u, \\
 (-1)^n \bar{\lambda}_{uw}  \bar{\lambda}^{n-1}_{uu} + \displaystyle \sum_{j=1}^{n-1} (-1)^{n-j} {n \choose j} \bar{\lambda}_{uw}  \lambda^{j}_{uu} \bar{\lambda}^{n-j-1}_{uu}  & \mbox{if~} v = u, ~w \ne u, \\
\displaystyle \sum_{j=1}^{n-1} (-1)^{n-j} {n \choose j} \lambda_{uv}\bar{\lambda}_{uw} \lambda^{j-1}_{uu} \bar{\lambda}^{n-j-1}_{uu}  & \mbox{if~} v \neq u, ~w \ne u.
\end{cases}
\eeq Thus, by the first identity, $\lambda_{uu}$ is real, and hence,
by the second identity, for every $w \in A_u \setminus \{u\}$, \beqn
\lambda^{n-1}_{uu} \bar{\lambda}_{uw} \sum_{j=0}^{n-1}
(-1)^{n-j} {n \choose j}  = 0. \eeqn Thus \beq \label{uu-0}
\lambda_{uu} =0~ \mbox{or}~ \lambda_{uw}=0 ~\mbox{for every} ~w \in
A_u \setminus \{u\}. \eeq Suppose $\lambda_{uw} \neq 0$ for some $w
\in A_u \setminus \{u\}$ and $n \Le 2$. Then, by \eqref{uu-0},
$\lambda_{uu}=0$, and hence, by the third identity in \eqref{7.1},
\beqn \displaystyle \sum_{j=1}^{n-1} (-1)^{n-j} {n \choose j}
\lambda_{uv}\bar{\lambda}_{uw} \lambda^{n-2}_{uu} =0, \quad v
\in A_u \setminus \{u\}. \eeqn If $n=2$, then
$-2\lambda_{uv}\bar{\lambda}_{uw} =0$, which is not possible
for $v=w$ (since by assumption $\lambda_{uw} \neq 0$). This proves the implication (i) $\Rightarrow$ (ii). Also,
since condition (b) of (ii) is equivalent to the assertion that
$\N{b}{u}$ is equal to the normal rank one operator $\lambda_{uu}
e_u \otimes e_u$ (see \eqref{til-N}), we also obtain the equivalence
of (ii) and (iii). Finally, the implication (ii) $\Rightarrow$ (i)
may be easily deduced from \eqref{7.1} and \eqref{inv-dom}.
\end{proof}
\begin{remark} The weighted join operator
 $\W{b}{u}$ is either symmetric or strictly $3$-symmetric. In particular, $\W{b}{u}$ is never strictly $2$-symmetric.
\end{remark}

We capitalize on the last proposition to exhibit a family of $n$-symmetric rank one extensions of weighted join operators.
\begin{proposition} \label{symm-r-one}
Let $\mathscr T=(V, E)$ be a rooted directed tree with root $\rootb$ and let $\mathscr T_{\infty}=(\V, E_{\infty})$ be the extended directed tree associated with $\mathscr T$.
For $u, \mf b \in V,$ consider the weight system
$\lambdab_u = \{\lambda_{uv}\}_{v\in \V}$ of complex numbers and let $W_{f, g}$ be the rank one extension of the weighted join operator $\W{b}{u}$ on $\mathscr T$, where $f \in \ell^2(V) \ominus \Hi{b}{u}$ is non-zero and $g : \supp\,\Hi{b}{u} \rar \mathbb C$ is a non-zero function.
Assume that $\inp{f}{e_{_{\lambdab_u, A_u}}}=0.$ Then the following statements hold:
\begin{enumerate}
\item If $n \geqslant 2$ and $\supp(g) \cap \{v \in \supp\,\Hi{b}{u} : \lambda_{uv} \neq 0\} = \emptyset$, then $W_{f, g}$ is $n$-symmetric if and only if $\{\lambda_{uv} : v \in \supp\,\Hi{b}{u}\}$ is contained in $\mathbb R$ and either $f(u)=0$ or \eqref{uu-0} holds.
\item If $\supp(g) \cap \{v \in \supp\,\Hi{b}{u} : \lambda_{uv} \in \mathbb R \setminus \{0\}\} \neq \emptyset$, then $W_{f, g}$ is never $n$-symmetric.
\end{enumerate}
\end{proposition}
\begin{proof} By \eqref{rone-extn},  $W^n_{f, g}$ can be decomposed as
\beqn
\begin{array}{lll}
W^n_{f, g} &=& \left[\begin{array}{cc}
(\Db{b}{u})^n & 0 \\
L_n & (\N{b}{u})^n
\end{array}
\right],
\end{array}
\eeqn
where $L_n,$ $n \geqslant 0$ is defined inductively as follows:
\beqn
L_0=0, ~L_1=f \obslash g, ~L_n=L_{n-1}\Db{b}{u}+ (\N{b}{u})^{n-1}f \obslash g, ~n \geqslant 2. 
\eeqn
Recall from \eqref{nilp-powers} that
${\N{b}{u}}^k =
 \lambda^{k-1}_{uu} \N{b}{u},$
$k \geqslant 1.$ An inductive argument now shows that
\beq
\label{Lk}
L_k = f \obslash g(\Db{b}{u})^{k-1}+ \sum_{j=1}^{k-1} \lambda^{j-1}_{uu} \N{b}{u} (f \obslash g)(\Db{b}{u})^{k-j-1}, \quad k \geqslant 1.
\eeq
Note that $W_{f, g}$ is $n$-symmetric if and only if $\N{b}{u}$ is $n$-symmetric and for every $(h_1, h_2),$ $(k_1, k_2)$ in $\mathscr D(W^n_{f, g}),$
\beq
\sum_{j=0}^{n} (-1)^{n-j} {n \choose j} \inp{(\Db{b}{u})^jh_1}{(\Db{b}{u})^{n-j}k_1} \notag & +& \\  \label{n-sym-eq-1} \sum_{j=1}^{n-1} (-1)^{n-j} {n \choose j} \inp{L_jh_1}{L_{n-j}k_1} &=& 0, \\ \label{n-sym-eq-2}
\sum_{j=0}^{n-1} (-1)^{n-j} {n \choose j} \inp{(\N{b}{u})^jh_2}{L_{n-j}k_1} &=& 0.
\eeq

Assume now that $n \geqslant 2$ and  $\inp{f}{e_{_{\lambdab_u, A_u}}}=0$. By \eqref{Lk}, \beq \label{Lk-2} L_k=f \obslash g(\Db{b}{u})^{k-1}, \quad k \geqslant 1. \eeq
It follows that for any $v, w \in \supp\,\Hi{b}{u}$,
\beqn
\sum_{j=1}^{n-1} (-1)^{n-j} {n \choose j} \inp{L_je_v}{L_{n-j}e_w}  &=& \overline{g(v)} g(w) \|f\|^2 \sum_{j=1}^{n-1} (-1)^{n-j} {n \choose j} \lambda^{j-1}_{uv} \bar{\lambda}^{n-j-1}_{uw}.
\eeqn
Hence \eqref{n-sym-eq-1} holds if and only if for any $v, w \in \supp\,\Hi{b}{u}$,
\beq \label{4.3.4-1}
(\lambda_{uv} -\bar{\lambda}_{uv})^n + |{g(v)}|^2 \|f\|^2 \sum_{j=1}^{n-1} (-1)^{n-j} {n \choose j} \lambda^{j-1}_{uv} \bar{\lambda}^{n-j-1}_{uv}=0,  \\ \notag
 \overline{g(v)} g(w)  \sum_{j=1}^{n-1} (-1)^{n-j} {n \choose j} \lambda^{j-1}_{uv} \bar{\lambda}^{n-j-1}_{uw}=0, \quad v \neq w.
\eeq
Further, \eqref{n-sym-eq-2} holds if and only if for every $v \in V \setminus \supp\,\Hi{b}{u}$ and $w \in \supp\,\Hi{b}{u},$
\beq
\label{4.3.5}
(-1)^n \lambda^{n-1}_{uw} g(w)\overline{f(v)} \,+\, \overline{f(u)} g(w) \inp{e_v}{e_{_{\lambdab_u, A_u}}} \sum_{j=1}^{n-1} (-1)^{n-j} {n \choose j} \lambda^{j-1}_{uu} \bar{\lambda}^{n-j-1}_{uw}=0.
\eeq

Assume that $\supp(g) \cap \{v \in \supp\,\Hi{b}{u} : \lambda_{uv} \neq 0\} = \emptyset$.
In this case, $(f \obslash g)\Db{b}{u}=0$. Hence, by \eqref{Lk-2}, $L_k=0$ for $k \geqslant 2.$
Let $v \in \supp(g)$, $h_1=e_v=k_1$ in
\eqref{n-sym-eq-1}. Then
\beqn
\inp{L_1e_v}{L_{n-1}e_v}=0, \quad n \geqslant 2,
\eeqn
which is possible only if $n \geqslant 3.$ In this case, \eqref{n-sym-eq-1} holds if and only if $$\{\lambda_{uv} : v \in \supp\,\Hi{b}{u}\} \subseteq \mathbb R.$$ Further, \eqref{n-sym-eq-2} holds if and only if
\beqn
\inp{(\N{b}{u})^{n-1}e_w}{L_{1}e_v}=0, \quad v \in \supp\,\Hi{b}{u},~ w \in V \setminus \supp\,\Hi{b}{u},
\eeqn
which is possible if and only if either $f(u)=0$ or $\lambda_{uu}=0.$
This completes the verification of (i).

Assume next that
$\supp(g) \cap \{v \in \supp\,\Hi{b}{u} : \lambda_{uv} \in \mathbb R \setminus \{0\}\} \neq \emptyset$.
In this case, there exists $\eta \in \supp(g)$ such that $\lambda_{u\eta} \in \mathbb R \setminus \{0\}.$ If possible, then assume that $W_{f, g}$ is $n$-symmetric. By \eqref{4.3.4-1} with $v=\eta$,
\beqn
(\lambda_{u\eta} -\bar{\lambda}_{u\eta})^n \,+\, \frac{|{g(\eta)}|^2}{|\lambda_{u\eta}|^2} \|f\|^2 \big((\lambda_{u\eta} -\bar{\lambda}_{u\eta})^n -  (-1)^{n}   \lambda^n_{u\eta} -\bar{\lambda}^n_{u\eta} \big)=0.
\eeqn
However, $\lambda_{u\eta} \in \mathbb R \setminus \{0\}$, and hence $n$ is necessarily an odd integer.
Further, since $\N{b}{u}$ is also $n$-symmetric, by \eqref{7.1} and \eqref{uu-0}, $\lambda_{uu} \in \mathbb R$ and \beq \label{l-uu-l-uv}
\lambda_{uu} =0~ \mbox{or}~ \lambda_{uw}=0 ~\mbox{for every} ~w \in
A_u \setminus \{u\}. \eeq
By \eqref{4.3.5} with $w=\eta$, for any $v \in  V \setminus (\supp\,\Hi{b}{u} \cup A_u),$ $f(v)=0.$ Further, if $\lambda_{uu} \neq 0,$ then by \eqref{l-uu-l-uv}, $e_{_{\lambdab_u, A_u}}=\lambda_{uu}e_u,$ which by assumption is orthogonal to $f.$ This forces $f(u)$ to be equal to $0.$
In that case, by \eqref{4.3.5} with $w=\eta$, $f(v)=0$ for all $v \in A_u$. This is not possible since $f \neq 0.$ Thus $\lambda_{uu}=0$ and $f(u) \neq 0$.
Once again, by \eqref{4.3.5} with $v=u$ and $w=\eta$, \beqn
(-1)^n \lambda^{n-1}_{uw}  \,+\,   \sum_{j=1}^{n-1} (-1)^{n-j} {n \choose j} \lambda^{j}_{uu} {\lambda^{n-j-1}_{uw}}=0.
\eeqn
It follows that $(\lambda_{uu}-\lambda_{u\eta})^n=\lambda^n_{uu}$. This is not possible, since
$n$ is an odd integer and $\lambda_{uu}, \lambda_{u\eta} \in \mathbb R \setminus \{0\}$. Thus we arrive at a contradiction to the assumption that $W_{f, g}$ is $n$-symmetric.
\end{proof}
\begin{remark}
Let $g : \supp\,\Hi{b}{u} \rar \mathbb C$ be given. Then
 $W_{f, g}$ is symmetric if and only if $\{\lambda_{uv} : v \in \supp\,\Hi{b}{u}\}$ is contained in $\mathbb R$, $g=0$ and $
\lambda_{uv} =0  ~\mbox{for~} v \in A_u \setminus \{u\},
$ where $A_u$ is as given in \eqref{A-u}. This may be concluded from \eqref{n-sym-eq-1}, \eqref{n-sym-eq-2} (with $n=1$) and
Proposition \ref{symm}. 
\end{remark}

\subsection{Proof of Theorem \ref{w-x}.}
We are now in a position to present a proof of Theorem \ref{w-x} (recall the notations $U^{(\mf b)}_u$, $g_x$, $W_{w, x}$ as introduced in the Prologue).
\begin{proof}
Note that $\ell^2(U^{(\mf b)}_u)$ is nothing but $\Hi{b}{u}$, while
$W_{w, x}$ is a rank one extension of a weighted join operator in $\ell^2(V)$.
Clearly, $W_{w, x}$ is densely defined with $\{e_v : v \in V\}$ contained in $\mathscr D(W_{w, x})$.
By \eqref{wt-Dub} and Theorem \ref{spectrum}(i), $$\sigma_p(W_{w, x})=
\{\dep_v-\dep_u : v \in U^{(\mf b)}_u \cup \{u\}\}.$$
Let us find conditions on $x \in \mathbb R$ which ensure that $W_{w, x}$ satisfies the compatibility condition I. To see that, note first that $\sigma_p(W_{w, x})$ is closed subset of $\mathbb C$ and $\mu_0=-\dep_u-1 \notin \sigma_p(W_{w, x})$. By assumption, $(\des{u}, E_u)$ is a  narrow tree of width $\mf m$, and hence by \eqref{wt-Dub} and \eqref{gxv}, we have the estimate
\beq \label{narrow-estimate}
 \sum_{v \in \supp\,\Hi{b}{u}} \frac{|g_x(v)|^2}{|\lambda_{uv}-\mu_0|^2} \,=\, \sum_{v \in U^{(\mf b)}_u} \frac{\dep^{2x}_v}{(\dep_v+1)^2} \,\leqslant\, \mf m \sum_{n=0}^{\infty} \frac{n^{2x}}{(n+1)^2}.
\eeq Since $\mbox{card}(\des{u})=\aleph_0$, we must have \beqn
\sum_{v \in \supp\,\Hi{b}{u}}
\frac{|g_x(v)|^2}{|\lambda_{uv}-\mu_0|^2}  \,\geqslant\,
\sum_{n=0}^{\infty} \frac{n^{2x}}{(n+1)^2}. \eeqn 
It follows 
from the last two estimates that $W_{w, x}$ satisfies the
compatibility condition I if and only if $x < 1/2.$ It now follows
from Corollary \ref{coro-spectrum} that $\sigma(W_{w, x})$ is a
proper closed subset of $\mathbb C$ if and only if $x < 1/2.$ The
conclusion in (i), (ii) and (iii) now follow from Theorem \ref{closed-0}
and Corollary \ref{coro-spectrum}. Since the spectrum of $W_{w, x}$
is a subset of $\{-\dep_u + k : k \in \mathbb N\}$, parts (iv) and
(v) may be deduced from Propositions \ref{sectorial},
\ref{semigroup} and \ref{hypo-rone}. Finally, part (vi) may be
deduced from Corollary \ref{cpt-res}.
\end{proof}
\begin{remark} \label{remark-narrow}
Note that the conclusion of Theorem \ref{w-x} may not hold in case $(\des{u}, E_u)$ is not narrow. For instance, if $\mathscr T$ is the binary tree, then an examination of \eqref{narrow-estimate} shows that
$$\sum_{v \in \supp\,\Hi{b}{u}} \frac{|g_0(v)|^2}{|\lambda_{uv}-\mu_0|^2}\,=\,\infty.$$ In this case, $\sigma(W_{w, 0})=\mathbb C$ (see Corollary \ref{coro-spectrum}). Finally, note that by Proposition \ref{symm}, the weighted join operator $\W{b}{u}$ with weight system given by \eqref{wt-Dub} is strictly $3$-symmetric, while by Proposition \ref{symm-r-one}(ii), its rank one extension $W_{w, x},$ $x \in \mathbb R$ is never $n$-symmetric.
\end{remark}

\chapter{Weighted Join operators on rootless directed trees}

In this chapter, we extend the notion of join operation at a given
base point to rootless directed trees and study the associated weighted join
operators. This is achieved by introducing a
partial order relation on a rootless directed tree.

\section{Semigroup structures on extended rootless directed trees}

Let $\mathscr T=(V, E)$ be a rootless directed tree
and let $\mathscr T_{\infty}=(\V, E_{\infty})$ be the extended directed tree associated with $\mathscr T$. Fix $u, v \in \V$. Define
\beqn
u \leq v ~\mbox{if there exists a directed path} ~[u, v] \mbox{~from~} u ~\mbox{to~} v.
\eeqn
Note that $\leq $ defines a partial order on $\V.$ Further, $\leq$ is anti-symmetric, since $\mathscr T$ is a directed tree. Moreover, since $v \leq \infty$ for every $v \in \V,$ $\infty$ can be considered as a maximal element of $\mathscr T_{\infty}$, whereas $\mathscr T_{\infty}$, being rootless, has no minimal element.
One may now define the join operation $u \sqcup v$ on $\mathscr T_\infty$ by setting
\beqn
u \sqcup v \,=\, \begin{cases} u & \mbox{if~} v \leq u, \\
 v & \mbox{if~} u \leq v, \\ \infty & \mbox{otherwise}.
\end{cases}
\eeqn
As in the case of rooted directed trees (see Lemma \ref{sg-1}), $(\V, \sqcup)$ is a commutative semigroup admitting $\infty$ as an absorbing element.

Let us now define the meet operation. Let $u \in V$ and $v \in \V$. Note that by \cite[Proposition 2.1.4]{Jablonski} and the definition of extended directed tree, there exists $w_0 \in V$ such that $\{u, v\} \subseteq \des{w_0}.$ Thus $\parentn{n}{u} = w_0=\parentn{m}{v} ~\mbox{for some}~ m, n \in \mathbb N$.
In particular, the set $\parset{u}{v}$ given by
\beqn
\parset{u}{v} &:=&\{w \in V : \parentn{n}{u} = w=\parentn{m}{v} ~\mbox{for some}~ m, n \in \mathbb N\}
\eeqn
is non-empty. Further, $w \leq u$ for every $w \in \parset{u}{v}$.
Moreover, it is totally ordered, that is, for any $w_1, w_2 \in \parset{u}{v}$, $w_1 \leq w_2$ or $w_2 \leq w_1.$
This follows since each vertex in $V$ has a unique parent. One may now define $u \sqcap v$ by
\beq
\label{u-v-des-w}
u \sqcap v \,=\, \max \big([w_0, u] \cap [w_0, v]\big),
\eeq
where $w_0$ is any element in $\parset{u}{v}.$
Note that \beq \label{u cap v -} u \sqcap v \in \asc{u} \cap \asc{v}.\eeq
Since $\parset{u}{v}$ is totally ordered, $u \sqcap v$ is independent of the choice of $w_0.$
Further,  we set $\infty \sqcap \infty = \infty.$
Once again, $(\V, \sqcap)$ is a commutative semigroup  admitting identity element as $\infty$ (cf. Lemma \ref{sg-2}). We leave the verification to the reader.

Before we define the join operation at a given base point, we present a decomposition of the extended directed tree.

\begin{lemma} \label{lem-v-infty-deco}
Let $\mathscr T=(V, E)$ be a rootless directed tree
and let $\mathscr T_{\infty}=(\V, E_{\infty})$ be the extended directed tree associated with $\mathscr T$. Then, for any $u \in V$, we have
\beqn
\V \,=\,  \bigcupdot_{j=0}^{\infty} \big(\des{\parentn{j}{u}} \setminus \des{\parentn{j-1}{u}}\big),
\eeqn
where we used the convention that $\des{\parentn{-1}{u}}=\emptyset.$
\end{lemma}
\begin{proof}
Let $v \in V.$  By \eqref{u-v-des-w}, there exists a non-negative integer $j$ such that $u \sqcap v = \parentn{j}{u}.$ It is now immediate from \eqref{u cap v -} and the uniqueness of $u \sqcap v$ that $$v \in \des{\parentn{j}{u}} \setminus \des{\parentn{j-1}{u}}.$$ Consequently, we get the inclusion $$V ~\subseteq ~ \bigcup_{j=0}^{\infty} \des{\parentn{j}{u}} \setminus \des{\parentn{j-1}{u}}.$$
Since $\infty \in \des{u}$, we get the desired equality.
\end{proof}
\begin{definition}[Join operation at a base point]
Let $\mathscr T=(V, E)$ be a rootless directed tree and let $\mathscr T_{\infty}=(\V, E_{\infty})$ be the extended directed tree associated with $\mathscr T$. Fix $\mf b  \in \V$ and let $u, v \in \V.$
Define the binary operation $\sqcup_{\mf b}$ on $\V$ by
\beqn
u \sqcup_{\mf b} v \,=\, \begin{cases} u \sqcap v & \mbox{if~} u, v \in \asc{\mf b}, \\
 u & \mbox{if~} v = \mf b, \\
 v & \mbox{if~} \mf b = u, \\
 u \sqcup v & \mbox{otherwise}.
\end{cases}
\eeqn
\end{definition}
Note that $(\V, \sqcup_{\mf b})$ is a commutative semigroup admitting identity element as $\mf b.$ Further, $\sqcup_{\infty}=\sqcap.$ The table for join operation $u \sqcup_{\mf b} v$ at the base point $\mf b$ for a rootless directed tree is identical with Table \ref{Table1}.

\begin{definition}
Let $\mathscr T=(V, E)$ be a rootless directed tree and let $\mathscr T_{\infty}=(\V, E_{\infty})$ be the extended directed tree associated with $\mathscr T$.
Fix $u, \mf b \in \V$ and the weight system
$\lambdab_u = \{\lambda_{uv}\}_{v\in \V}$ of complex numbers,
we define the {\em weighted join operator $\W{b}{u}$ $($based at $\mf b$$)$  on ${\mathscr T}$} by
   \begin{align*}
   \begin{aligned}
{\mathscr D}(\W{b}{u}) & := ~\big\{f \in \ell^2(V) \colon
\varLambda^{(\mf b)}_{u} f \in \ell^2(V)\big\},
   \\
\W{b}{u} f & := ~ \varLambda^{(\mf b)}_{u} f, \quad f \in {\mathscr
D}(\W{b}{u}),
   \end{aligned}
   \end{align*}
where $\varLambda^{(\mf b)}_{u}$ is the mapping defined on
complex functions $f$ on $V$ by
   \begin{align*}
(\varLambda^{(\mf b)}_{u} f) (w) ~:=
\displaystyle \sum_{v \in M^{(\mf b)}_u(w)}\Lmd{u}{v}\,  f(v), \quad  w \in
V\end{align*}
with $M^{(\mf b)}_u(w)$ given by
\beqn \label{M-u-b-rl} M^{(\mf b)}_u(w)\,:=\,\{v \in V : u \w v=w\}. \eeqn
\end{definition}
\begin{remark}
As in the rooted case, it can be seen that $\mathscr D(\W{b}{u})$ forms a subspace of $\ell^2(V)$.
Clearly, $e_v \in {\mathscr D}(\W{b}{u})$ and
$(\W{b}{u} e_v)(w) =  \lambda_{uv}\, e_{u \w v}(w),$ $w \in V.$
Thus
\beqn \label{D-V}
\mathscr D_{V}:=\mbox{span}\,\{e_v : v \in V\} \subseteq {\mathscr D}(\W{b}{u}), \quad \W{b}{u} \mathscr D_{V} \subseteq \mathscr D_{V}.
\eeqn
Thus all positive integral powers of $\W{b}{u}$ are densely defined and the Hilbert space adjoint ${\W{b}{u}}^*$ of $\W{b}{u}$ is defined.
\end{remark}

\begin{figure}
\begin{tikzpicture}[scale=.8, transform shape]
\tikzset{vertex/.style = {shape=circle,draw, minimum size=1em}}
\tikzset{edge/.style = {->,> = latex'}}
\node[vertex] (x) at  (-5.5,0) {$\ldots$};
\node[vertex] (y) at  (-3.5,0) {$6$};
\node[vertex] (z) at  (-1.5,0) {$3$};
\node[vertex] (a) at  (0.5,0) {$0$};
\node[vertex] (b) at  (2,-1) {$1$};
\node[vertex] (c) at  (2,1) {$2$};
\node[vertex] (d) at  (4, -1) {$4$};
\node[vertex] (e) at  (4, 1) {$\ldots$};
\node[vertex] (f) at  (6, -1) {$\ldots$};
\node[vertex] (h) at  (8, 1) {$\ldots$};
\node[vertex] (i) at  (8, -1) {$\ldots$};

\tikzset{vertex/.style = {shape=circle,draw, blue, minimum size=1em}}
\node[vertex] (g) at  (6, 1) {$m$};
\draw[edge] (x) to (y);
\draw[edge] (y) to (z);
\draw[edge] (z) to (a);
\draw[edge] (a) to (b);
\draw[edge] (a) to (c);
\draw[edge] (b) to (d);
\draw[edge] (c) to (e);
\draw[edge] (d) to (f);
\draw[edge] (e) to (g);
\draw[edge] (g) to (h);
\draw[edge] (f) to (i);
\end{tikzpicture}
\caption{The rootless directed tree $\mathscr T_3$ with prescribed vertex $m$} \label{fig4}
\end{figure}

To get an idea about the structure of weighted join operators on rootless directed trees, let us discuss one example.
\begin{example}[With one branching vertex]
Let $\mathscr T_3$ denote the directed tree as shown in Figure \ref{fig4} (see \cite[Eqn (6.2.10)]{Jablonski}).
Consider the ordered orthonormal basis $$\{e_{3n} : n \in \mathbb N\} \cup \{e_{3n + 2} : n \in \mathbb N\} \cup \{e_{3n+1} : n \in \mathbb N\}$$ of $\ell^2(V).$
The matrix representation of the weighted join operator $\W{m}{0}$  and weighted meet operator $\W{\mf \infty}{m}$ on $\mathscr T_3$ are given by
\beqn
\W{m}{0} &=&    \begin{psmallmatrix}
\cdots & 0 & \cdots &     &   &   &  &\\
 & \vdots &  &   &   &   &  \\
\cdots & 0 &  \cdots &  &   &   &   &  \\
\ldots & \lambda_{{\mf m}3}  & \lambda_{{\mf m}0} &   \lambda_{{\mf m}2} & \cdots & \lambda_{{\mf m}{\mf m}} & 0 & \cdots & &\\
\cdots & 0 & 0 &  0 & 0 & \cdots & 0 & \lambda_{{\mf m}{\mf m}+3} &   0 & \cdots   & \\
 & & &  &\vdots  & &  \vdots & 0 &  \lambda_{{\mf m}{\mf m}+6}   & 0 &\\
 & & & & & & &  \vdots & 0  & \lambda_{{\mf m}{\mf m}+9}  & \\
 & & & & & & &  &   \vdots &   & \ddots \end{psmallmatrix} \oplus {\bf 0},
\\ \\ \\
 \W{\mf \infty}{m} &=& \begin{psmallmatrix}
\ddots & \\
 & \lambda_{m3} & \\
  & &  \lambda_{m0} & 0  & \cdots &  &  & \cdots & 0 & \lambda_{m1} & \lambda_{m4} & \cdots \\
  & & 0 & \lambda_{m2} & 0 & \cdots &    &    &\\
 & & \vdots & 0 & \ddots &  0 & \cdots  &   & \\
  & & & \vdots & & \lambda_{mm-3}  & 0  &  \cdots & \\
  & & & & &  0 &   \lambda_{mm} & \lambda_{mm+3}  & \cdots \\
  & & & & &  \vdots &   0 & 0  & \cdots \\
  & & & & &   &   \vdots & \vdots  &  \end{psmallmatrix}.
\eeqn
These expressions should be compared with the matrix representations of weighted join operator and weighted meet operator on $\mathscr T_2$ as discussed in Example \ref{ex-T2}.
\eop
\end{example}

\section{A decomposition theorem and spectral analysis}

Note that $\W{b}{\mf b}=\D{\mf b}$, the diagonal operator with diagonal entries $\lambdab_{\mf b}$. The structure of the weighted join operator $\W{b}{u}$, $u \neq \mf b$ turns out to be quite involved in case of rootless directed trees.
We present below a counter-part of Theorem \ref{o-deco} for weighted join operators on rootless directed trees.
\begin{theorem} \label{o-deco-rl}
Let $\mathscr T=(V, E)$ be a rootless directed tree and let $\mathscr T_{\infty}=(\V, E_{\infty})$ be the extended directed tree associated with $\mathscr T$.
For $\mf b \in \V$, $u \in V \setminus \{\mf b\}$ and the weight system
$\lambdab_u = \{\lambda_{uv}\}_{v\in \V}$ of complex numbers,
let $\D{u}$ be the diagonal operator on $\mathscr T$ and let $\W{b}{u}$ be the weighted join operator on $\mathscr T$.
Then the following hold:
\begin{enumerate}
\item Assume that $\mf b \in V$.
Consider
the subspace $\Hi{b}{u}$ of $\ell^2(V)$ given by
\beq \label{H-u-rl}
\Hi{b}{u} &=& \begin{cases}
\displaystyle \ell^2(\asc{u} \cup  \desb{u})
 & \mbox{if~} \mf b \in \des{u},  \\
 \displaystyle \ell^2(\mathsf{Des}_u(u))   & \mbox{otherwise}.
\end{cases}
\eeq
Then the weighted join operator $\W{b}{u}$ admits the decomposition \beq \label{deco-W-rl} \W{b}{u}\,=\,\Db{b}{u} \oplus \N{b}{u}~ \mbox{on~} \ell^2(V)\,=\,\Hi{b}{u} \oplus \big(\ell^2(V) \ominus \Hi{b}{u}\big),\eeq
where $\Db{b}{u}$ is a densely defined diagonal operator in $\Hi{b}{u}$ and $\N{b}{u}$ is a rank one densely defined linear operator on $\ell^2(V) \ominus \Hi{b}{u}$ with invariant domain. Further,
$\Db{b}{u}$ and $\N{b}{u}$ are given by
\beq \label{til-D-rl}
\Db{b}{u} &=& \D{u}|_{_{\Hi{b}{u}}}, \quad \mathscr D(\Db{b}{u})\,=\,\{f \in \Hi{b}{u} : \D{u}f \in \Hi{b}{u}\},
\\
\label{til-N-rl}
\N{b}{u} &=&
e_u \obslash e_{_{\lambdab_u, A_u}},
\eeq
where the subset $A_u$ of $V$ is given by
\beqn
\label{A-u-rl}
A_u &=& \begin{cases}
 [u, \mf b]
 & \mbox{if~} \mf b \in \des{u}, \\
\asc{u} \cup \{\mf b, u\}  & \mbox{otherwise}.
\end{cases}
\eeqn
\item Assume that $\mf b = \infty.$ Consider the orthogonal decomposition of $\ell^2(V)$ $($as ensured by Lemma \ref{lem-v-infty-deco}$)$ given by
\beqn
\label{infty-l2}
\ell^2(V) \,=\,   \bigoplus_{j=0}^{\infty} \ell^2(\des{u_j} \setminus \des{u_{j-1}}),
\eeqn
where $\des{u_{-1}}=\emptyset$ and $u_j:=\parentn{j}{u},$ $j \in \mathbb N.$
Further, with respect to the above decomposition, $\W{b}{u}$ decomposes as
\beq
\label{infty-W}
\mathscr D(\W{b}{u}) & =  & \notag  \bigoplus_{j=0}^{\infty} \mathscr D(e_{u_j} \obslash e_{_{\lambdab_u,  \des{u_j} \setminus \des{u_{j-1}}}}), \\
\W{b}{u}  & =  & \bigoplus_{j=0}^{\infty} e_{u_j} \obslash e_{_{\lambdab_u,  \des{u_j} \setminus \des{u_{j-1}}}}.
\eeq
\end{enumerate}
\end{theorem}
\begin{proof} Let $V_u$ be the complement of $\des{u} \cupdot \asc{u}$ in $V.$ We divide the proof into three cases.
\begin{case}
$\mf b \notin \des{u}:$
\end{case}
\noindent
Consider the following decomposition of $V$:
\beqn
V \,=\, \mathsf{Des}_u(u) \cupdot A_u \cupdot \Big(V_u \setminus \{\mf b\}\Big).
\eeqn
Thus
\beqn
\ell^2(V) \,=\, \ell^2(\mathsf{Des}_u(u)) \oplus \ell^2(A_u) \oplus \ell^2(V_u \setminus \{\mf b\}).
\eeqn
Note that $\ell^2(\mathsf{Des}_u(u))$, $\ell^2(A_u)$ and $\ell^2(V_u \setminus \{\mf b\})$ are invariant subspaces of $\W{b}{u}.$
We claim that the weighted join operator $\W{b}{u}$ is given by
\beqn
\mathscr D(\W{b}{u}) &=&  \mathscr D(\D{u}|_{_{\ell^2(\mathsf{Des}_u(u))}}) \oplus \mathscr D(e_u \obslash e_{_{\lambdab_u, A_u}}) \oplus \ell^2(V_u \setminus \{\mf b\}), \\
\W{b}{u} &=& \D{u}|_{\ell^2(\mathsf{Des}_u(u))} \oplus  e_u \obslash e_{_{\lambdab_u, A_u}} \oplus 0.
\eeqn
To see the above decomposition, let $f \in \ell^2(V)$ be of the form $f_1 \oplus f_2 \oplus f_3$ with $f_1 \in \ell^2(\mathsf{Des}_u(u)), f_2 \in \ell^2(A_u), f_3 \in \ell^2(V_u \setminus \{\mf b\}).$ Note that $f \in \mathscr D(\W{b}{u})$ if and only if $\W{b}{u}f \in \ell^2(V),$ where $\W{b}{u}f$ takes the form
\beqn
&& \sum_{v \in \mathsf{Des}_u(u)} f_1(v) \lambda_{uv}\, e_{u \w v} \oplus \sum_{v \in A_u} f_2(v) \lambda_{uv}\, e_{u \w v} \oplus \sum_{v \in V_u \setminus \{\mf b\}} f_3(v) \lambda_{uv}\, e_{u \w v} \\ &=&\sum_{v \in \mathsf{Des}_u(u)} f_1(v) \lambda_{uv}\, e_{v} \oplus \sum_{v \in A_u} f_2(v) \lambda_{uv}\, e_{u} \oplus 0.
\eeqn
It follows that $f \in \mathscr D(\W{b}{u})$ if and only if
\beqn
f_1 \in \mathscr D(\D{u}|_{_{\ell^2(\mathsf{Des}_u(u))}}), \quad f_2 \in \mathscr D(e_u \obslash e_{_{\lambdab_u, A_u}}), \quad f_3 \in \ell^2(V_u \setminus \{\mf b\}).
\eeqn
This yields the desired orthogonal decomposition of $\W{b}{u}$.

\begin{case}
$\mf b \in \des{u} \setminus \{\infty\}:$
\end{case}
\noindent
Consider the following decomposition of $V$:
\beqn
V &=& \Big(\asc{u} \cup \desb{u} \Big) \cupdot A_u \cupdot V_u.
\eeqn
 Thus
\beqn
\ell^2(V) &=& \ell^2(\asc{u} \cup \desb{u}) \oplus \ell^2(A_u) \oplus \ell^2(V_u).
\eeqn
Note that $\ell^2(\asc{u} \cup \desb{u})$, $\ell^2(A_u)$ and $\ell^2(V_u)$ are invariant subspaces of $\W{b}{u}.$
As in the previous case, one can verify that the weighted join operator $\W{b}{u}$ is given by
\beqn
\mathscr D(\W{b}{u}) &=&  \mathscr D(\D{u}|_{_{\ell^2(\asc{u} \cup \desb{u})}}) \oplus \mathscr D(e_u \obslash e_{_{\lambdab_u, A_u}}) \oplus \ell^2(V_u), \\
\W{b}{u} &=& \D{u}|_{_{\ell^2(\asc{u} \cup \desb{u})}} \oplus  e_u \obslash e_{_{\lambdab_u, A_u}} \oplus 0.
\eeqn

\begin{case}
$\mf b =\infty:$
\end{case}
\noindent
The decomposition \eqref{infty-W} follows from $$\W{b}{u}e_v = \lambda_{uv}e_{u \sqcap v} = \lambda_{uv}e_{u_j}, \quad v \in \des{u_j} \setminus \des{u_{j-1}}.$$
This completes the proof.
\end{proof}

Here are some immediate consequences of Theorem  \ref{o-deco-rl}.

\begin{corollary}[Dichotomy]
Let $\mathscr T=(V, E)$ be a rootless directed tree and let $\mathscr T_{\infty}=(\V, E_{\infty})$ be the extended directed tree associated with $\mathscr T$.
For $\mf b, u \in V$ and the weight system
$\lambdab_u = \{\lambda_{uv}\}_{v\in \V}$ of complex numbers,
the weighted join operator $\W{b}{u}$ on ${\mathscr T}$ is at most rank one (possibly unbounded) perturbation of a diagonal operator, while the weighted meet operator $\W{\mf \infty}{u}$ on ${\mathscr T}$ is an infinite rank operator provided $\lambdab_u \subseteq \mathbb C \setminus \{0\}$.
\end{corollary}

The orthogonal decomposition \eqref{deco-W-rl} of $\W{b}{u}$, as ensured by Theorem \ref{o-deco-rl}, is given by the triple $(\Db{b}{u}, \N{b}{u}, \Hi{b}{u})$, where $\Hi{b}{u}$, $\Db{b}{u}$ and $\N{b}{u}$ are given by \eqref{H-u-rl}, \eqref{til-D-rl} and \eqref{til-N-rl} respectively.

\begin{corollary} \label{coro-bdd-rl}
Let $\mathscr T=(V, E)$ be a rootless directed tree and let $\mathscr T_{\infty}=(\V, E_{\infty})$ be the extended directed tree associated with $\mathscr T$.
For $\mf b \in \V$, $u \in V \setminus \{\mf b\}$ and the weight system
$\lambdab_u = \{\lambda_{uv}\}_{v\in \V}$ of complex numbers,
let $\D{u}$ be the diagonal operator on $\mathscr T$ and let $\W{b}{u}$ be a weighted join operator on $\mathscr T$. Then
the following holds true:
\begin{enumerate}
\item If $\mf b \notin \des{u}$, then $\W{b}{u}$ is bounded  if and only if \beqn \lambdab_u  \in \ell^{\infty}(\des{u})~ \mbox{and}~ \sum_{v \in \asc{u}}|\lambda_{uv}|^2 < \infty.\eeqn
\item If $\mf b \in \des{u} \setminus \{\infty\}$, then $\W{b}{u}$ is bounded  if and only if $\lambdab_u  \in \ell^{\infty}(\asc{u} \cup \desb{u})$.
\item If $\mf b = \infty,$ then $\W{b}{u}$ is bounded  if and only if \beqn \sup_{j \Ge 0} \sum_{v \in \des{u_j} \setminus \des{u_{j-1}}}\!\!\!\!\!\!\!\!\!|\lambda_{uv}|^2 < \infty, \eeqn
where $\des{u_{-1}}=\emptyset$ and $u_j:=\parentn{j}{u}$ for $j \in \mathbb N.$
\item If $\mf b \notin \des{u} $ and $\sum_{v \in \asc{u}}|\lambda_{uv}|^2 = \infty$, then $\W{b}{u}$ is not closable.
\end{enumerate}
\end{corollary}
\begin{proof}
The desired conclusions in (i)-(iii) are immediate from Theorem \ref{deco-W-rl}, while (iv) follows from (i) and Lemma \ref{lem-inj-ten-unb}.
\end{proof}
\begin{remark}
Let us briefly discuss the spectral picture for a weighted join operator $\W{b}{u}$ on the rootless directed tree $\mathscr T$. Consider the orthogonal decomposition $(\Db{b}{u}, \N{b}{u}, \Hi{b}{u})$ of $\W{b}{u}$.
In case $\mf b \in \des{u} \setminus \{\infty\}$, the operator $\N{b}{u}$ in the decomposition of $\W{b}{u}$ is bounded. Hence the spectral picture of $\W{b}{u}$ can be described as in the rooted case (see Theorem \ref{spectrum}). We leave the details to the reader.   In case $\mf b \notin \des{u}$ and $\sum_{v \in \asc{u}}|\lambda_{uv}|^2 < \infty$, $\N{b}{u}$ is bounded and the same remark as above is applicable. Suppose now that $\mf b \notin \des{u}$ and $\sum_{v \in \asc{u}}|\lambda_{uv}|^2 = \infty.$ Then, by Corollary \ref{coro-bdd-rl}, $\N{b}{u}$ is unbounded. Hence, by Lemma \ref{lem-inj-ten-unb},
\beqn
\sigma(\W{b}{u})=\mathbb C, \quad \sigma_p(\W{b}{u}) =  \sigma_p(\Db{b}{u}) \cup \{0, \lambda_{uu}\}.
\eeqn
\end{remark}

The verification of the following is similar to that of Corollary \ref{coro-c-Jordan}, and hence we skip its verification.
\begin{corollary} \label{C-Jordan-rl}
Let $\mathscr T=(V, E)$ be a rootless directed tree  and let $\mathscr T_{\infty}=(\V, E_{\infty})$ be the extended directed tree associated with $\mathscr T$.
For $u \in V$ and $\mf b \in V \setminus \{u\}$, consider the weight system
$\lambdab_u = \{\lambda_{uv}\}_{v\in \V}$ of complex numbers and
let $\W{b}{u}$ denote the weighted join operator on ${\mathscr T}$. Consider the orthogonal decomposition $(\Db{b}{u}, \N{b}{u}, \Hi{b}{u})$ of $\W{b}{u}$ as ensured by Theorem \ref{o-deco-rl}. Then the following statements hold:
\begin{enumerate}
\item If $\lambda_{uu}=0$, then $\W{b}{u}$ is a complex Jordan operator of index $2$ provided $\Db{b}{u}$ belongs to $B(\Hi{b}{u})$ or $\N{b}{u}$ belongs to $B(\ell^2(V) \ominus \Hi{b}{u})$.
\item If $\lambda_{uu} \neq 0$ and $\N{b}{u} \in B(\ell^2(V) \ominus \Hi{b}{u})$, then $\W{b}{u}$ admits a bounded Borel functional calculus.
\end{enumerate}
\end{corollary}

As in the case of rooted directed trees (see Definition \ref{def-r}), one may introduce the rank one extension $W_{f, g}$ of the weighted join operator $\W{b}{u}$ on a rootless directed tree in a similar fashion. In case the operator $\N{b}{u}$ appearing in the decomposition of $\W{b}{u}$ is unbounded, it turns out (due to the fact that $\Db{b}{u}$ has no ``good influence" on $\N{b}{u}$) that $W_{f, g}$ is not even closable. On the other hand, in case $\N{b}{u}$ is bounded, one can obtain counter-parts of Theorems \ref{closed-0},  \ref{spectrum} and Propositions \ref{sectorial}, \ref{semigroup} for rank one extensions of weighted join operators on rootless directed trees along similar lines.
We leave the details to the reader.
\chapter{Rank one perturbations}

The considerations in Chapter 4 around the notion of rank one extensions of weighted join operators were mainly motivated by the graph-model developed in earlier chapters. Some of these can be replicated in a general set-up simply by replacing the vertex set of the underlying rooted directed tree by a countably infinite directed set. The results in this chapter give a few glimpses of this general scenario.  In particular, we discuss the role of some compatibility conditions (differing from compatibility conditions I and II as introduced in Chapter 4) in the sectoriality of rank one perturbations of diagonal operators. We also discuss
sectoriality of the form-sum of the form associated with a sectorial diagonal operator  and a form associated with not necessarily square-summable functions $f$ and $g.$
\section{Operator-sum}

Throughout this chapter, $J$ denotes a countably infinite directed set and let $\{e_j : j \in J\}$ be the standard orthonormal basis of $\ell^2(J).$ Let
$D_{\lambdab}$ stand for the diagonal operator in $\ell^2(J)$ with diagonal entries $\lambdab=\{\lambda_j : j \in J\}$ given by
\beqn
D_{\lambdab}e_j = \lambda_j e_j, \quad j \in J.
\eeqn

The following main result of this section shows that a compatibility condition ensures the sectoriality of the operator-sum of a sectorial diagonal operator and an unbounded rank one operator.

\begin{theorem} \label{operator-sum}
Let $D_{\lambdab}$ be a sectorial operator in $\ell^2(J)$ 
and let $f \in \ell^2(J)$. Let $g : J \rar \mathbb C$ be such that for some $z_0 \in \rho(D_{\lambdab}),$ \beq \label{CC-II-gen} \sum_{j \in J}\frac{|g(j)|^2}{|\lambda_{j}-z_0|^2} ~< ~\infty. \eeq Then $D_{\lambdab} + f \obslash g$ defines a sectorial operator in $\ell^2(J)$ with domain $\mathscr D(D_{\lambdab}).$
\end{theorem}

Clearly, Theorem \ref{operator-sum} generalizes Proposition \ref{sectorial}.
In its proof, we need a couple of observations of independent interests. The first of which characterizes the $B$-boundedness of $A$ in terms of the strict contractivity of $B(A-z)^{-1}$ for some $z \in \rho(A)$, where $A$ is a normal operator satisfying certain growth condition.

\begin{proposition} \label{rel-bdd-bdd}
Let $A$ be a normal operator in $\mathcal H$ and let $B$ be a linear operator in $\mathcal H$ with $\mathscr D(A)\subseteq \mathscr D(B)$. 
If there exists $z\in \rho(A)$ such that $\|B(A-z)^{-1}\|<1,$ then 
\begin{eqnarray}\label{Kato-Relich-Assume}
\|B x\|\,\leqslant\,a \|Ax\| + b\|x\|,\quad x\in\mathscr D(A),
\end{eqnarray}
where $a=\|B(A-z)^{-1}\|$ and $b \in (0, \infty)$.
Conversely, if there exist $a\in (0,1)$ and $b \in (0, \infty)$ such that \eqref{Kato-Relich-Assume} holds and if for some $\theta \in \mathbb R$,
\beq \label{sp-growth}
\max\{|\mu|, ~n\} \,\leqslant\,|\mu - e^{i \theta}n|, \quad n \in \mathbb N, ~\mu \in \sigma(A),
\eeq 
then $\|B(A-z)^{-1}\|<1$ for some  $z\in \rho(A)$.  
\end{proposition}
\begin{remark}
There are two particular instances in which \eqref{sp-growth} can be ensured.
\begin{enumerate}
\item If $A$ is self-adjoint, then by \cite[Corollary 3.14]{Sc}, $\sigma(A) \subseteq \mathbb R$, and hence \eqref{sp-growth} holds with $\theta =\pm \pi/2.$
\item If $A$ is sectorial with vertex at $0$, then \eqref{sp-growth} holds with $\theta=\pi.$
\end{enumerate}
\end{remark}
\begin{proof}
Note that if $z\in \rho(A)$ is such that $\|B(A-z)^{-1}\|<1,$ then for any $x\in\mathscr D(A)$
\begin{eqnarray*}
\|Bx\|=\|B(A-z)^{-1}(A-z)x\|
&\leqslant & a \|Ax\| + b \|x\|, 
\end{eqnarray*}
where $a=\|B(A-z)^{-1}\|$ and $b=|z|\|B(A-z)^{-1}\|.$ This yields \eqref{Kato-Relich-Assume}. 

To see the converse, suppose that \eqref{Kato-Relich-Assume} and \eqref{sp-growth} hold.
For any $z\in \rho(A),$ $(A-z)^{-1}$ is a bounded operator on $\mathcal H$ with
range equal to $\mathscr D(A).$ Thus \eqref{Kato-Relich-Assume} becomes 
\begin{eqnarray}\label{B(A-z)}
\|B(A-z)^{-1}y\|\,\leqslant\,a \|A(A-z)^{-1}y\| + b\|(A-z)^{-1}y\|, \quad z\in \rho(A), ~y\in\mathcal H.
\end{eqnarray}
Let $E(\cdot)$ denote the spectral measure of $A$ 
and let $n\in \mathbb N.$ Clearly, by \eqref{sp-growth}, $e^{i \theta} n \in \rho(A).$
It follows from the spectral theorem \cite[Theorem 13.24]{R} and \eqref{B(A-z)}, that for any $y \in \mathcal H,$ 
\beqn \|B(A- e^{i \theta} n)^{-1}y\| & \leqslant & a \sqrt{\int_{\sigma(A)}\Big|\frac{\mu}{\mu - e^{i \theta}n}\Big|^2 \|E(d\mu)(y)\|^2} \, + \, b\|(A-e^{i \theta} n)^{-1}y\|\\ & \overset{\eqref{sp-growth}} \leqslant & a\|y\| + \frac{b}{n}\|y\|. \eeqn
Thus, for sufficiently large integer $n,$
\[\|B(A-e^{i \theta} n)^{-1}\|\,\leqslant\,a + \frac{b}{n}<1.\]
This completes the proof.
\end{proof}

 We need one more fact in the proof of Theorem \ref{operator-sum} (cf. Theorem \ref{spectrum}(iii)).
 \begin{proposition} \label{prop-CCI}
 Let $D_{\lambdab}$ be a sectorial operator in $\ell^2(J)$ and let $f \in \ell^2(J)$. Let $g : J \rar \mathbb C$ be such that for some $z_0 \in \rho(D_{\lambdab}),$  \eqref{CC-II-gen} holds. 
Then, for any $z\in\rho(D_{\lambdab})$, $G_z:=f \obslash g(D_{\lambdab}-z)^{-1}$ is a Hilbert-Schmidt integral operator with square-summable kernel \beqn K_z(j,k)\,:=\,\frac{f(j)\overline{g(k)}}{\lambda_{k}-z}, \quad j, k \in J.\eeqn Moreover,
 there exists a sequence $\{z_n\}_{n \in \mathbb N} \subseteq \rho(D_{\lambdab})$  such that 
\beqn
\lim_{n \rar \infty} \|G_{z_n}\|_2 \,=\, 0,
\eeqn
where $\|\cdot\|_2$ denotes the Hilbert-Schmidt norm.
\end{proposition}
\begin{proof}
Let $z\in\rho(D_{\lambdab})$. Then, as in the proof of Theorem \ref{spectrum}, it can be seen that $G_z$ is a Hilbert-Schmidt integral operator with kernel $K_z \in \ell^2(J \times J).$ Moreover, 
\begin{eqnarray}\label{Norm Gz}
\|G_z\|_2^2\,=\,\|f\|^2 \sum_{j \in J} \frac{|g(j)|^2}{|\lambda_{j}-z|^2}.
\end{eqnarray}
On the other hand, it is easily seen that there exists a sequence $\{z_n\}_{n \in \mathbb N} \subseteq \rho(D_{\lambdab})$ with the only accumulation point at $\infty$  such that 
\beqn |\lambda_{j}-z_0|\,\leqslant\,|\lambda_{j}-z_n|, \quad n \in \mathbb N, ~j \in J.\eeqn 
Using Lebesgue dominated convergence theorem, we see that 
\[\sum_{j \in J}\frac{|g(j)|^2}{|\lambda_{j}-z_n|^2} ~\to ~0 ~\mbox{ as }~n\to\infty.\]
Hence, by \eqref{Norm Gz}, we obtain the remaining part.
\end{proof}

\begin{proof}[Proof of Theorem \ref{operator-sum}]
As in the proof of Proposition \ref{lem-H-transform}, it is easily seen that \beq \label{D-I} \mathscr D(D_{\lambdab}) \subseteq \mathscr D(f \obslash g).\eeq Also,
by Proposition \ref{prop-CCI}, for any $a \in (0, 1)$, there exists $z \in \rho(D_{\lambdab})$ such that the Hilbert-Schmidt norm of $f \obslash g(D_{\lambdab}-z)^{-1}$ is less than $a$.
Since the operator norm of any Hilbert-Schmidt operator is less than or equal to its Hilbert-Schmidt norm, it follows from Proposition \ref{rel-bdd-bdd} that $f \obslash g$ is $D_{\lambdab}$-bounded with $D_{\lambdab}$-bound equal to $0.$ The desired conclusion is now immediate from \cite[Theorem 4.5.7]{M}.
\end{proof}

The following provides a variant of Corollary \ref{bdd-rel}. Since the bounded component $\N{b}{u}$ in the rank one extension $W_{f, g}$ has no effect in the  $\Db{b}{u}$-boundedness of $f \obslash g,$ this variant may be obtained by imitating the proof of Theorem \ref{operator-sum}.
\begin{corollary}
Let $\mathscr T=(V, E)$ be a rooted directed tree with root $\rootb$ and let $\mathscr T_{\infty}=(\V, E_{\infty})$ be the extended directed tree associated with $\mathscr T$.
For $u, \mf b \in V,$  consider the weight system
$\lambdab_u = \{\lambda_{uv}\}_{v\in \V}$ of complex numbers and let $W_{f, g}$ be the rank one extension of the weighted join operator $\W{b}{u}$ on $\mathscr T$, where $f \in \ell^2(V) \ominus \Hi{b}{u}$ is non-zero and $g : \supp\,\Hi{b}{u} \rar \mathbb C$ is given.
Suppose that $W_{f, g}$ satisfies the compatibility condition I. Then
$W_{f, g}$ decomposes as $A + B + C,$ where $A, B, C$ are densely defined operators given by
\eqref{ABC} such that $B+C$ is $A$-bounded with $A$-bound equal to $0.$
\end{corollary}

We conclude this section with a brief discussion on some spectral properties of rank one perturbations of the diagonal operator $D_{\lambdab}$. 
Assume that there exists $z_0 \in \rho(D_{\lambdab})$ such that $g : J \rar \mathbb C$ satisfies \eqref{CC-II-gen}.
By \eqref{D-I}, $D_{\lambdab}+ f \obslash g$ is a densely defined operator in $\ell^2(J)$ with domain $\mathscr D(D_{\lambdab})$. Let $\mu \in \mathbb C \setminus \lambdab$ be an eigenvalue of $D_{\lambdab}+ f \obslash g.$ Thus there exists a non-zero vector $h$ in $\ell^2(J)$ such that for every $j \in J,$
\begin{eqnarray*}
\lambda_jh(j)+ \Big(\sum_{k \in J} h(k)\overline{g(k)}\Big) f(j)&=&\mu h(j)\\
\Longrightarrow ~(\mu - \lambda_{j})h(j)&=&\Big(\sum_{k \in J} h(k)\overline{g(k)}\Big)f(j)\\
\Longrightarrow ~ h(j)&=& a \frac{f(j)}{\mu - \lambda_{j}},
\end{eqnarray*}
where $a = \displaystyle \sum_{k \in J} h(k)\overline{g(k)}$ is non-zero. Therefore, we have 
\begin{eqnarray}\label{Hilbert-Eigen}
\sum_{j \in J} \frac{f(j)\overline{g(j)}}{\mu-\lambda_{j}}\,=\,1.
\end{eqnarray}
Notice that expression in \eqref{Hilbert-Eigen} is an analytic function in $\mu$ outside the spectrum of $D_{\lambdab}.$ Also, for $\mu \in \mathbb C \setminus \lambdab$ to be an eigenvalue for $D_{\lambdab}+ f \obslash g$, it has to satisfy \eqref{Hilbert-Eigen}. Therefore, the set of all eigenvalues of $D_{\lambdab}+ f \obslash g$ outside the set $\sigma(D_{\lambdab})$ has to be discrete (cf. \cite[Corollary 2.5]{I}). One may argue now as in the proof of Theorem \ref{spectrum}(iv)  using Weyl's theorem to conclude that $\sigma_e(D_{\lambdab}+ f \obslash g)= \sigma_e(D_{\lambdab}).$

\section{Form-sum}

Consider a sectorial diagonal operator
 $D_{\lambdab}$ in $\ell^2(J)$. By \cite[Theorem 3.35, Chapter V]{K} $D_{\lambdab}$ (and hence $D^*_{\lambdab}$  as well) has a unique square root; let us denote it  by $R_{\lambdab}$ and note that $R_{\lambdab}$ is a sectorial operator with $\mathscr D(R_{\lambdab})=\mathscr D(R^*_{\lambdab})$. Consider the form $Q_R$ given by
  \beq \label{QR}
  Q_R(h, k)\,:=\, \inp{R_{\lambdab}h}{R^*_{\lambdab}k}, \quad h, k \in \mathscr D(R_{\lambdab}).
  \eeq
 Then  having the unbounded form-perturbation of  $Q_R$ by $Q_{f,g}$ is to find conditions on $f, g$, so that the form  
 \beq \label{Qfg} Q_{f,g}(h,k)\,:=\,\sum_{j \in J} h(j)\overline{g(j)} \ \sum_{j \in J} f(j)\overline{k(j)}  \eeq 
 is well defined for  all $h,k \in  \mathscr D(R_{\lambdab})$  and  to ensure that  the perturbation by $Q_{f, g}$ is small, (so that an application of  \cite[Theorems 1.33 and 2.1, Chapter VI]{K}  can be made through a choice of large enough $z$ in the  appropriate sector).   
 In such a case, the form-sum $Q_R + Q_{f, g}$ is closed and defines a sectorial operator  with domain contained in the domain of $Q_R$ (the reader is referred to \cite[Chapter VI]{K} for all definitions pertaining to sesquilinear forms in Hilbert spaces). This is made precise in the following theorem.
 
 \begin{theorem} \label{form-sum-1}
 Let $D_{\lambdab}$ be a sectorial diagonal operator in $\ell^2(J)$ with angle $\theta \in (0, \pi/2)$ and vertex $0$. Let $f  :  J \rar \mathbb C $ and $g : J \rar \mathbb C$  be such that  for some $z_0 \in (-\infty, 0)$,
 \beq \label{CC-III} \sum_{j \in J} \frac{|g(j)|^2}{|\sqrt{\lambda_{j}}-z_0|^2}~<~ \infty, \quad \sum_{j \in J} \frac{|f(j)|^2}{|\sqrt{\lambda_{j}}-z_0|^2}~<~ \infty, \eeq
where the square-root is obtained by the branch cut at the non-positive real axis.
Let  $Q_R$ and $Q_{f, g}$ be as given by \eqref{QR} and \eqref{Qfg}.
Then $Q_{f, g}(h, k)$ is defined for all $h, k \in \mathscr D(R_{\lambdab}).$ Moreover, the form $Q_R+Q_{f,g}$ is sectorial and there exists a unique sectorial operator $T$ in $\ell^2(J)$ with domain contained in the domain of $Q_R$ such that 
\[Q_R(h, k)\,+\,Q_{f,g}(h,k)\,=\,\inp{Th}{k},\quad h\in \mathscr D(T),~k\in \mathscr D(Q_R).\]   
 \end{theorem}
 \begin{proof}
 Note that $z_0 \in \rho(R_{\lambdab})$, and hence by \eqref{CC-III} and the definition of $\mathscr D(R_{\lambdab})$, $Q_{f, g}(h, k)$ is well-defined for all $h, k \in \mathscr D(R_{\lambdab}).$ To see the remaining part, 
we make some general observations. Notice first that \beq \label{real-im-form}
\Re Q_R(h, h)=\sum_{j \in J} \Re \lambda_{j} |h(j)|^2, \quad \Im Q_R(h, h)=\sum_{j \in J} \Im \lambda_{j} |h(j)|^2.
\eeq 
Further, since $R_{\lambdab}$ is normal,
\beq
\label{R-QR} \|R_{\lambdab}h\|^2 \,\geqslant\,|Q_R(h, h)|^2, \quad h \in \mathscr D(R_{\lambdab}).
\eeq
Furthermore,
since $D_{\lambdab}$ is a sectorial operator with angle $\theta$,
\beq \label{angle} 
\left.
  \begin{minipage}{60ex}
\begin{enumerate}
\item[] $|\Im \lambda_j|  \,\leqslant\, \tan\theta \ \Re \lambda_j, \quad j \in J,$
\vskip.1cm
\item[] $|\Im Q_R(h, h)|  \,\leqslant\, \tan \theta \ \Re Q_R(h, h), \quad h \in \mathscr D(R_{\lambdab}).$
\end{enumerate}
 \end{minipage}
   \right\}
\eeq
 
We  claim
that $Q_R$ is a closed form.  It suffices to check that $\Re Q_R$ is closed (see \cite[Pg 336]{K}). 
Let $h\in \ell^2(J)$, $\{h_n \}_{n \in \mathbb N} \subseteq \mathscr D(Q_R)$ be such that $h_n\to h$ as $n\to\infty$ and $\Re Q_R(h_n-h_m, h_n-h_m)\to 0$ as $n$  and $m$ tend to $\infty$. 
It follows that
\begin{eqnarray*}
 \Re Q_R(h_n-h_m, h_n-h_m) &\overset{\eqref{real-im-form}
 } =& \sum_{j \in J} \Re \lambda_{j} |h_n(j)-h_m(j)|^2\\
& \overset{\eqref{angle}}\geqslant & \frac{1}{\sqrt{1+\tan^2 \theta} } \sum_{j \in J}  |\lambda_j| 
|h_n(j)-h_m(j)|^2
\\ & = & \frac{1}{\sqrt{1+\tan^2 \theta} } \sum_{j \in J}   
|\sqrt{\lambda_j} h_n(j)-\sqrt{\lambda_j} h_m(j)|^2 
\\
&=&  \frac{1}{\sqrt{1+\tan^2 \theta} } \|R_{\lambdab}(h_n-h_m)\|^2.
\end{eqnarray*}
This shows that $\{R_{\lambdab}(h_n)\}_{ n \in \mathbb N}$ is a Cauchy sequence in $\ell^2(J)$. Thus there exists $g \in \ell^2(J)$ such that $R_{\lambdab}(h_n) \to g$ as $n\to\infty.$ Since $R_{\lambdab}$ is closed, $h \in \mathscr D(R_{\lambdab})=\mathscr D(Q_R)$ and $g=R_{\lambdab}h$.
By \eqref{R-QR}, $Q_R(h_n-h, h_n-h)\to 0$ as $n\to\infty$. This completes the verification of the claim.

We next show that $Q_{f,g}$ is $Q_{R}$-bounded with $Q_R$-bound less than $1.$
To see this, let $z \leqslant z_0$ be a negative real number. 
Note that $|w-z_0| \leqslant |w-z|$ for any $w \in  \mathbb C$ such that  $|\arg w| <
\theta.$ It follows from \eqref{CC-III} that
\beq
\label{z-z0}
\left.
  \begin{minipage}{60ex}
\begin{enumerate}
\item[] $\displaystyle \sum_{j \in J} \frac{|f(j)|^2}{|\lambda_{j}^{1/2}-z|^2} \,\leqslant\,\displaystyle \sum_{j \in J} \frac{|f(j)|^2}{|\lambda_{j}^{1/2}-z_0|^2} < \infty,$
\vskip.3cm
\item[] $\displaystyle \sum_{j \in J} \frac{|g(j)|^2}{|\lambda_{j}^{1/2}-z|^2}  \,\leqslant\, \displaystyle \sum_{j \in J} \frac{|g(j)|^2}{|\lambda_{j}^{1/2}-z_0|^2} < \infty.$
\end{enumerate}
 \end{minipage}
   \right\}
   \eeq
For any $k \in \mathscr D(R_{\lambdab}),$ note that
\begin{eqnarray}
\label{dot-1}
\Big|\sum_{j \in J} f(j)\overline{k(j)} \Big|^2 
&\leqslant & \sum_{j \in J} \frac{|f(j)|^2}{|\lambda_{j}^{1/2}-z|^2} \sum_{j \in J} |\lambda_{j}^{1/2}-z|^2|k(j)|^2 \notag \\
&\leqslant & 2\sum_{j \in J} \frac{|f(j)|^2}{|\lambda_{j}^{1/2}-z|^2} \big(\sum_{j \in J} |\lambda_{j}| |k(j)|^2 + \sum_{j \in J} |z|^2 |k(j)|^2\big) \notag \\ 
& \overset{\eqref{angle}} \leqslant & 2 \sum_{j \in J} \frac{|f(j)|^2}{|\lambda_{j}^{1/2}-z|^2} \Big(\frac{1}{\sqrt{1+\tan^2 \theta}} |Q_{R}(k, k)| +  |z|^2 \|k\|^2 \Big). \quad \quad 
\end{eqnarray}
Similarly, we can conclude that for any $h \in \mathscr D(R_{\lambdab}),$
\beq \label{dot-2} \Big|\sum_{j \in J} h(j)\overline{g(j)} \Big|^2  \leqslant 
2 \sum_{j \in J} \frac{|g(j)|^2}{|\lambda_{j}^{1/2}-z|^2} \Big(\frac{1}{\sqrt{1+\tan^2 \theta}} |Q_{R}(h, h)| +  |z|^2 \|h\|^2 \Big).\eeq
Combining \eqref{dot-1} and \eqref{dot-2} together,  for any $h \in \mathscr D(R_{\lambdab}),$ we obtain
\begin{eqnarray}\label{relat-Q}
|Q_{f,g}(h, h)|\,\leqslant\,2\Big(\frac{|Q_{R}(h, h)|}{\sqrt{1+\tan^2 \theta}} +|z|^2\|h\|^2 \Big) \sqrt{\sum_{j \in J} \frac{|f(j)|^2}{|\lambda_{j}^{1/2}-z|^2}}\sqrt{\sum_{j \in J} \frac{|g(j)|^2}{|\lambda_{j}^{1/2}-z|^2}}.\quad \quad 
\end{eqnarray}
Also, note that for any $j \in J,$
\beqn \lim_{\substack{z\to -\infty \\ z\leqslant z_0}}\frac{f(j)}{\lambda_{j}^{1/2}-z}= 0, \quad \lim_{\substack{z\to -\infty \\ z\leqslant z_0}} \frac{g(j)}{\lambda_{j}^{1/2}-z} = 0. \eeqn 
We now conclude from \eqref{z-z0} and the Lebesgue dominated convergence theorem that
\begin{eqnarray*}
\lim_{\substack{z\to -\infty \\ z\leqslant z_0}} \sum_{j \in J} \frac{|f(j)|^2}{|\lambda_{j}^{1/2}-z|^2} = 0, \quad
\lim_{\substack{z\to -\infty \\ z\leqslant z_0}} \sum_{j \in J} \frac{|g(j)|^2}{|\lambda_{j}^{1/2}-z|^2} =  0.
\end{eqnarray*}
Let $h \in \mathscr D(R_{\lambdab})$. Then, for any $a \in (0,1),$ there exists $z\leqslant z_0$ such that 
\begin{eqnarray*}
\sum_{j \in J} \frac{|f(j)|^2}{|\lambda_{j}^{1/2}-z|^2}  <  \frac{a\sqrt{1+\tan^2 \theta}}{2},\quad
\sum_{j \in J} \frac{|g(j)|^2}{|\lambda_{j}^{1/2}-z|^2}  <  \frac{a\sqrt{1+\tan^2 \theta}}{2},
\end{eqnarray*}
and hence by \eqref{relat-Q}, we conclude that for some $b >0$,
\begin{eqnarray}\label{Qfg-QR}
|Q_{f,g}(h, h)|\leqslant a|Q_{R}(h, h)|+b\|h\|^2.
\end{eqnarray}
Since $Q_R$ is a sectorial form, \eqref{Qfg-QR} together with \cite[Theorem 1.33, Chapter VI]{K} implies that $Q_R+Q_{f,g}$ is a sectorial form. The remaining assertion about the existence of $T$ is immediate from \cite[Theorem 2.1, Chapter VI]{K}. 
\end{proof}

We next present a variant of Theorem \ref{form-sum-1}, where we discuss sectoriality of form-sum with $Q_R$ replaced by the form $Q$ 
defined as
\beq \label{form-Q} Q(h,k):=\inp{|D_{\lambdab}|^{1/2}h}{|D_{\lambdab}|^{1/2}k}, \quad h, k \in \mathscr D(|D_{\lambdab}|^{1/2}),
\eeq
where $|A|$ denotes the modulus of a densely defined closed operator $A.$
Needless to say, we alter compatibility conditions as per the requirement. 
\begin{theorem} \label{form-sum2}
 Let $D_{\lambdab}$ be a sectorial diagonal operator in $\ell^2(J)$ with angle $\theta$ and vertex $0$. Let $f  :  J \rar \mathbb C $ and $g : J \rar \mathbb C$  be such that  for some $\beta_0  > 0$, 
 \begin{eqnarray}\label{f-g-beta}
\sum_{j \in J} \frac{|f(j)|^2}{(|\lambda_{j}|^{1/2}+\beta_0)^2} < \infty, \quad
\sum_{j \in J} \frac{|g(j)|^2}{(|\lambda_{j}|^{1/2}+\beta_0)^2} < \infty.
\end{eqnarray}
Let  $Q$ and $Q_{f, g}$ be as given by \eqref{form-Q} and \eqref{Qfg}.
Then $Q_{f, g}(h, k)$ is defined for all $h, k \in \mathscr D(|D_{\lambdab}|^{1/2}).$ Moreover, the form $Q+Q_{f,g}$ is sectorial and there exists a unique sectorial operator $T$ in $\ell^2(J)$ with domain contained in the domain of $Q$ such that 
\[Q(h, k)+Q_{f,g}(h,k)=\inp{Th}{k},\quad h\in \mathscr D(T),\, k\in \mathscr D(Q).\]   
 \end{theorem}
 \begin{proof}
For $h, k \in \mathscr D(|D_{\lambdab}|^{1/2}),$
note that
\begin{eqnarray}\label{K-beta}
Q_{f,g}(h,k) 
= \sum_{p, q \in J} K_\beta (p,q)\tilde{h}(p)\overline{\tilde{k}(q)},   
\end{eqnarray}
where, for $p, q \in J$ and $\beta \in [\beta_0, \infty),$ 
\beqn 
\left.
  \begin{minipage}{65ex}
\begin{enumerate}
\item[] 
$K_\beta(p,q) := \frac{f(p)}{|\lambda_{p}|^{1/2}+\beta}\frac{\overline{g(q)}}{|\lambda_{q}|^{1/2}+\beta}, $
\vskip.1cm
\item[] $\tilde{h}(p):=(|\lambda_{p}|^{1/2}+\beta)h(p),  \quad \tilde{k}(p):=(|\lambda_{p}|^{1/2}+\beta)k(p).$
\end{enumerate}
\end{minipage}
\right.
\eeqn 
Let $G_{\beta}$ denote the integral operator with kernel $K_\beta$ and notice that 
\begin{eqnarray*}
\|G_\beta\|_2^2=\sum_{p \in J} \frac{|f(p)|^2}{(|\lambda_{p}|^{1/2}+\beta)^2}\,\sum_{q \in J} \frac{|g(q)|^2}{(|\lambda_{q}|^{1/2}+\beta)^2}.
\end{eqnarray*}
From \eqref{f-g-beta}, it is clear that $G_\beta$ is a Hilbert-Schmidt operator for any $\beta \geqslant \beta_0$. One may argue as in Proposition \ref{prop-CCI}, using Lebesgue dominated convergence theorem, to conclude that $\|G_\beta\|_2\to 0$ as $\beta \to \infty.$
Now one may use \eqref{K-beta} to see that for any $h \in \mathscr D(|D_{\lambdab}|^{1/2}),$ 
\begin{eqnarray*}
|Q_{f,g}(h,h)| &\leqslant &  \sum_{q \in J} |h(q)||{g(q)}| \ \sum_{p \in J} |f(p)||{k(p)}| \\ & \leqslant & 
\|G_\beta\|_2 \sum_{p \in J}  (|\lambda_{p}|^{1/2}+\beta)^2 |h(p)|^2\\
&\leqslant & 2\|G_\beta\|_2 (\sum_{p \in J} |\lambda_{p}||h(p)|^2+\beta^2 \|h\|^2)
\\ & = & 2\|G_\beta\|_2 ({Q}(h, h)+\beta^2 \|h\|^2).
\end{eqnarray*}
Since $\lim_{\beta \rar \infty} \|G_\beta\|_2 = 0$, with an arbitrarily small $a>0$, we obtain for some $b >0$,
\[|Q_{f,g}(h,h)|\leqslant a\, {Q}(h, h)+ b\|h\|^2, \quad h \in \mathscr D(|D_{\lambdab}|^{1/2}).\]
Since $Q$ defines a  sectorial form, one may now argue as in the proof of Theorem \ref{form-sum-1} to complete the proof.
\end{proof}
In general, the form $Q_{f, g}$ is far from being closable. 
In fact, since $Q_{f, g}$ is associated with the rank one operator $f \obslash g,$ it may be derived from Kato's first representation theorem \cite[Theorem 2.1, Chapter VI]{K} and Lemma \ref{lem-inj-ten-unb} that $Q_{f, g}$ is closable if and only if $g \in \ell^2(J)$. Under suitable compatibility conditions, Theorems \ref{form-sum-1} and \ref{form-sum2} above ensure that the form-sums $Q_R + Q_{f, g}$ and $Q + Q_{f, g}$ are indeed closed. Finally, we note that as in the case of operator-sum, under the compatibility condition \eqref{CC-II-gen}, it can be seen that the eigenvalues of the sectorial operator $T$ associated with $Q+Q_{f, g}$ outside $\sigma(D_{\lambdab})$ is discrete and the essential spectrum of $T$ coincides with that of $D_{\lambdab}$.

\chapter*{Epilogue}


The present paper capitalizes on the order structure of directed trees to introduce and study the classes of weighted join operators and their rank one extensions.
In particular, we discuss the issue of closedness, unravel the structure of Hilbert space adjoint and identify various spectral parts of members of these classes. Certain discrete Hilbert transforms arise naturally in the spectral theory of rank one extensions of weighted join operators.
The assumption that the underlying directed trees are rooted or rootless brings several prominent differences in the structures of these classes. Further, these classes overlap with the well-studied classes of complex Jordan operators, $n$-symmetric operators and sectorial operators. This work also takes a brief look into the general theory of rank one perturbations.
As a natural outgrowth of this work, the study of finite rank extensions of weighted join operators would be desirable. In this regard, we would like to draw attention to the very recent work \cite{LT-1} on finite rank (self-adjoint) perturbations of self-adjoint operators.


In the remaining part of this section, we discuss some problems pertaining to the theory
of weighted join operators and their rank one extensions, which arise naturally in our efforts to understand these operators.
In what follows, let $\W{b}{u}$ be a weighted join operator on a rooted directed tree $\mathscr T=(V, E)$ and let $W_{f, g}$ be its rank one extension. Here $u \in V$, $\mf b \in V$, $f \in \ell^2(V) \ominus \Hi{b}{u}$ is non-zero and $g : \supp\,\Hi{b}{u} \rar \mathbb C$ is given.

%

\subsection*{Numerical range and Friedrichs extension}

\index{$\Theta(T)$}

Let  $(\Db{b}{u}, \N{b}{u}, \Hi{b}{u})$ denote the orthogonal decomposition of $\W{b}{u}$.
Recall that the {\it numerical range} $\Theta(T)$ of a densely defined linear operator $T$ is given by $$\Theta(T)\,:=\,\{\inp{Tf}{f} : f \in \mathscr D(T), ~\|f\|=1\}.$$
Since the numerical range of a diagonal operator is contained in the closed convex hull of its diagonal entries, we obtain
$\Theta(\Db{b}{u})  ~\subseteq ~  \mbox{conv} \{\lambda_{uv} : v \in \supp\,\Hi{b}{u} \}.$
where $\mbox{conv}(A)$ denotes the closed convex hull of $A.$
Further, for any $f \in \ell^2(V) \ominus \Hi{b}{u}$ of unit norm, we have
\beqn \inp{e_u \otimes e_{_{\lambdab_u, A_u}}f}{f}=\inp{f}{e_{_{\lambdab_u, A_u}}}\overline{f(u)}=\overline{f(u)} \sum_{v \in A_u}\lambda_{uv} f(v), \eeqn
where $A_u$ is as given in \eqref{A-u}.
Thus the numerical range $\Theta(\N{b}{u})$ of $\N{b}{u}$ satisfies
\beqn
\Theta(\N{b}{u})  &=&
\Big\{ \overline{f(u)} \sum_{v \in A_u}\lambda_{uv} f(v) : \|f\|_{\ell^2(V)}=1\Big\} \\
&\subseteq & 
\Big\{z \in \mathbb C : |z|^2 \Le \sum_{v \in A_u}|\lambda_{uv}|^2 \Big\}.
\eeqn
The numerical range of the rank one extension $W_{f, g}$ of $\W{b}{u}$ is given by
\beqn
\left\{
\inp{\Db{b}{u}h}{h}  +
\inp{(f \obslash g)h}{k} + \inp{\N{b}{u}k}{k}
:   (h, k) \in \mathscr D(W_{f, g}),~
 \|h\|^2+\|k\|^2=1
\right\}.
\eeqn
Recall that if the numerical range is a proper subset of the complex plane, then  the underlying operator is closable (see \cite[Theorem 3.4, Chapter V]{K}). It would be interesting otherwise also to find conditions on $g$ (different from the compatibility conditions) so that the numerical range of $W_{f, g}$ is a proper subset of the complex plane or is contained in a sector.

Let us now discuss the so-called Friedrichs extensions of weighted join operators and their rank one extensions \cite{M}.
Suppose, for some $r \in {\mathbb R}$ and $M \in (0,\infty)$, we have
\beq \label{F-e} |{\Im} \langle \W{b}{u}h,h \rangle| \,\leqslant\, M \ {\Re}
\langle (\W{b}{u}-r)h, h \rangle, \quad
h \in  {\mathscr D}(\W{b}{u}).
\eeq
By \cite[Theorem 2.12.1]{M}, there exist a subspace $\Gamma$ of $\ell^2(V)$, an inner product
$\langle \cdot,\cdot\rangle_{\Gamma}$ on $\Gamma$ with the corresponding norm $\|\cdot\|_{\Gamma}$,
and a sectorial sesquilinear form $\mathcal F$ on $\Gamma$ such that the following assertions hold:
\begin{enumerate}
\item[(a)] ${\mathscr D}(\W{b}{u})$ is a dense subspace of $\Gamma$ (in $\|\cdot\|_{\Gamma}$).
\item[(b)] $\langle h, k \rangle_{\Gamma} = (1/2)\big(\langle \W{b}{u}h, k \rangle +
\langle h, \W{b}{u}k \rangle\big) +
(1-r)\langle h, k \rangle,$ $h, k \in {\mathscr D}(\W{b}{u})$.
\item[(c)] ${\mathcal F}(h, k) = \langle \W{b}{u}h, k\rangle$ for all $h, k \in {\mathscr D}(\W{b}{u})$.
\end{enumerate}
It turns out that the linear operator $A$ associated with the sectorial sesquilinear form $\mathcal F$, referred to as the {\it Friedrichs extension} of $\W{b}{u}$, turns out to be $\W{b}{u}$ itself. Since $\W{b}{u}$ is a closed linear operator (Proposition \ref{closed}), this fact may be deduced from \cite[Lemma 1.6.14]{M} and Theorem \ref{spectrum} (see also \cite[Theorem 2.9, Chapter VI]{K}).
It would be desirable to find conditions on $g$ (similar to compatibility conditions) so that \eqref{F-e} is ensured for the rank one extension $W_{f, g}$ of $\W{b}{u}$.  In this case,  $W_{f, g}$ admits a Friedrichs extension. In case $W_{f, g}$ is closed, then it can be seen once again that this extension coincides with $W_{f, g}$ itself. The Friedrichs extension would be of interest, particularly, in case either $W_{f, g}$ is not closed or $\sigma(W_{f, g})$ is the entire complex plane. One may be keen to know whether or not there exists a rank one extension $W_{f, g}$, which is closable but not closed.

\subsection*{Hyponormality and $n$-Symmetricity} It has been seen in Proposition \ref{hypo} that the notions of hyponormality and normality coincide in the context of rank one extensions of weighted join operators. We do not know whether or not there exists any non-normal hyponormal rank one extension $W_{f, g}$ of a weighted join operator. The essential difficulty in this problem is unavailability of an explicit expression for the Hilbert space adjoint of $W_{f, g}$.
As evident from Propositions \ref{symm} and \ref{symm-r-one}, $n$-symmetric rank one extensions of weighted join operators exist in abundance. The problem of classifying all $n$-symmetric rank one extensions of weighted join operators remains unsolved.

\subsection*{$C^*$-algebras}

Let $\mathscr T$ be a  leafless rooted directed tree and let $\W{b}{u}$ be a weighted join operator on $\mathscr T$. 
Assume that $\W{b}{u}$ is bounded and $V_u \neq \emptyset$ (see \eqref{cg-dec0}). Recall that the essential spectrum of an orthogonal direct sum of two bounded operators $A, B \in B(\mathcal H)$ is a union of essential spectra of $A$ and $B$.
Also, since essential spectrum is invariant under compact perturbation \cite{Co}, Theorem \ref{o-deco} together with Proposition \ref{lem-lns}(iii) implies that
\beqn
\sigma_e(\W{b}{u}) \,=\, \begin{cases}   \sigma_e(\D{u}|_{_{\ell^2(\Desb{u})}})  \cup \{0\}  &  \mbox{if~} \mf b \in \des{u} \setminus \{\infty\}, \\
\{0\} & \mbox{if~} \mf b  = \infty, \\
\sigma_e(\D{u}|_{_{\ell^2(\des{u})}}) \cup \{0\}
 & \mbox{otherwise}.
\end{cases}
\eeqn
On the other hand, the essential spectrum of a normal operator is the complement of isolated eigenvalues of finite multiplicity in its spectrum \cite{Co}.
Thus the weight system $\lambdab_u$ completely determines the essential spectra of bounded weighted join operators.
Let $C^*(\W{b}{u})$ denote the $C^*$-algebra generated by $\W{b}{u}$. Since $\W{b}{u}$ is essentially normal (see Theorem \ref{o-deco}), the quotient $C^*$-algebra $C^*(\W{b}{u})/\mathcal K$ can be identified with $C(\sigma_e(\W{b}{u}))$, where $\mathcal K$ is the $C^*$-algebra of compact operators and $C(X)$ denotes the $C^*$-algebra of continuous functions on a compact Hausdorff space $X$ endowed with sup norm.

\index{$C^*(\W{b}{u})$}

We conclude this paper with another possible line of investigation.  For $\mf b \in V,$ consider the family $\mathscr F_{\mf b}:=\{\W{b}{u}\}_{u \in V}$ of bounded linear  weighted join operators $\W{b}{u}$ on a directed tree $\mathscr T=(V, E).$ A routine verification shows that $\W{b}{u}\W{b}{v} = \W{b}{v}\W{b}{u}$ if and only if $$\lambda_{vw} \lambda_{u v \sqcup w} \,=\, \lambda_{uw} \lambda_{vu\sqcup w}, \quad w \in V.$$
The later condition holds, in particular, for the constant weight systems $\lambdab_u,$ $u \in V$ with value $1$. Assume that the family $\mathscr F_{\mf b}$ is commuting. By Theorems \ref{o-deco} and \ref{o-deco-rl}, the family $\mathscr F_{\mf b}$
is essentially normal.
Motivated by \cite[Theorem 2.11]{KK}, one may ask whether the $C^*$-algebra $C^*(\mathscr F_{\mf b})$ generated by $\mathscr F_{\mf b}$ is completely determined by the directed tree $\mathscr T$ and weight systems $\lambdab_u$, $u \in V$ ? In case the answer is no, what are the complete invariants which determine $C^*(\mathscr F_{\mf b})$ ?

\medskip \textit{Acknowledgement}. \
The authors would like to thank Shailesh Trivedi and Soumitra Ghara for some stimulating conversations related to the subject of this paper.
We also convey our sincere thanks to Sumit Mohanty for drawing our attention to the references \cite{Fr} and \cite{Pr}, where respectively the notions of graph with boundary and spiral-like ordering, relevant to the present investigations, were introduced.
\bibliographystyle{amsalpha}

\vskip.2cm

\noindent


\printindex

\end{document}